%% file: Generalized_affine_building.tex
\documentclass[12pt]{amsart}
\usepackage{geometry}    
\usepackage{amsmath, amscd} 
\usepackage{amssymb} 
\usepackage{xcolor} 

\usepackage{pst-node}
\usepackage{tikz-cd} 
\usepackage{graphicx} 
\usepackage{comment}
\usepackage[parfill]{parskip}    
\usepackage{xr}
\usepackage{hyperref} 

\usepackage{fancyhdr}

\usepackage{tikz-cd}
\usepackage{subcaption}
\usepackage{amsthm}
\usepackage{mathtools}
\usepackage{xfrac}
\usepackage{stmaryrd} 
\usepackage{wrapfig} 
\usepackage{epigraph} 
\usepackage{ wasysym } 
\usepackage{mathrsfs} 
\usepackage{setspace} 

\newcommand{\F}{\mathbb{F} }
\newcommand{\R}{\mathbb{R} }
\newcommand{\Q}{\mathbb{Q} }
\newcommand{\N}{\mathbb{N} }
\newcommand{\Z}{\mathbb{Z} }

\newcommand{\X}{\mathcal{X}}

\newcommand{\K}{\mathbb{K}}

\newcommand{\frakg}{\mathfrak{g}}
\newcommand{\frakp}{\mathfrak{p}}
\newcommand{\frakk}{\mathfrak{k}}
\newcommand{\fraka}{\mathfrak{a}}

\newcommand{\Ap}{A^+}
\newcommand{\ApF}{A_\F^+}

\newcommand{\frakn}{\mathfrak{n}}

\newcommand{\A}{\mathbb{A}}
\newcommand{\Fun}{\mathscr{F}} 
\newcommand{\B}{\mathcal{B}}

\newcommand{\Id}{\operatorname{Id} }
\newcommand{\diag}[1]{\operatorname{Diag}\left( #1_1, \ldots , #1_n\right) }

\newcommand{\tran}{{^{\mathsf{T}}\!}}
\newcommand{\KPhi}{\vphantom{\Phi}_{\K}\Phi}

\newcommand{\KW}{\vphantom{W}_{\K}W}

\numberwithin{equation}{section}
\newtheorem{theorem}{Theorem}[section]
\newtheorem{lemma}[theorem]{Lemma}
\newtheorem{proposition}[theorem]{Proposition}
\newtheorem{corollary}[theorem]{Corollary}

\newtheorem{question}[theorem]{Question}

\makeatletter
\newtheorem*{rep@theorem}{\rep@title}
\newcommand{\newreptheorem}[2]{%
	\newenvironment{rep#1}[1]{%
		\def\rep@title{#2 \ref{##1}}%
		\begin{rep@theorem}}%
		{\end{rep@theorem}}}
\makeatother

\newreptheorem{theorem}{Theorem}
\newreptheorem{proposition}{Proposition}
\newreptheorem{lemma}{Lemma}
\newreptheorem{corollary}{Corollary}

\theoremstyle{remark} 

\title[Affine buildings for algebraic groups over real closed fields]{Generalized affine buildings for semisimple algebraic groups over real closed fields}
\author{Raphael Appenzeller}
\address{Department of Mathematics, Heidelberg University, Germany}
\email{rappenzeller@mathi.uni-heidelberg.de}
\thanks{\textit{Address}: Department of Mathematics, Heidelberg University, Germany}
\thanks{\textit{E-mail}: \texttt{rappenzeller@mathi.uni-heidelberg.de}}
\date{\today}      

\begin{document}

\def\subjclassname{\textup{2020} Mathematics Subject Classification}
\expandafter\let\csname subjclassname@1991\endcsname=\subjclassname
\date{\today}

\begin{abstract}
	We use real algebraic geometry to construct an affine $\Lambda$-building $\B$ associated to the $\F$-points of a semisimple algebraic group, where $\F$ is a valued real closed field. We characterize the spherical building at infinity and the local building at a base point. We compute stabilizers of various subsets of $\B$ and obtain group decompositions. 
\end{abstract}


\maketitle

\input{introduction}

\begingroup
	\setlength{\parskip}{1.5pt} 
	\tableofcontents
\endgroup

\input{affine_L_buildings}

\input{definition_B}

\input{proof_B}

\input{local_global}

\input{appendixSLn}

 \newpage
 
    \bibliographystyle{alpha_noand} 
    \bibliography{refs} 
 
	
\end{document}

%% file: introduction.tex
\section{Introduction}

Buildings were introduced by Jacques Tits as an analogue of Riemannian symmetric spaces for algebraic groups. In their seminal work \cite{BrTi,BrTi84}, Bruhat and Tits associate (not necessarily discrete) affine buildings to reductive groups over fields with real valuations. Other constructions of affine buildings have since been found in terms of lattices, seminorms \cite{BrTi84_norms, BrTi87, Par} and asymptotic cones \cite{KlLe}. In this work, we study symmetric spaces from a real algebraic point of view to construct non-discrete affine buildings associated to algebraic groups over real closed valued fields following ideas from \cite{Bru2,KrTe}. We actually work in the more general framework of affine $\Lambda$-buildings, which were introduced by \cite{Ben1} to allow for valuations to arbitrary ordered abelian groups $\Lambda$. Recently, the Bruhat-Tits construction has also been generalized to this setting by \cite{HIL23}. 

Riemannian symmetric spaces of non-compact type are determined by their symmetry groups which are Lie groups that are (essentially) the $\R$-points $G_{\R}$ of semisimple algebraic groups $G < \operatorname{SL}_n$. In this work, we replace the reals $\R$ by a real closed field $\F$ with a valuation to some ordered abelian group $\Lambda$. In analogy to the construction of symmetric spaces from $G_{\R}$, we use $G_\F$ to define a $\Lambda$-metric space $\B$ on which $G_\F$ acts by isometries. Under a technical assumption, we show that $\B$ admits the structure of an affine $\Lambda$-building.

\begin{reptheorem}{thm:B_is_building}
	If the root system $\Sigma$ of the Lie group $G_\R$ is reduced ($\alpha \in \Sigma$ implies $2\alpha \notin \Sigma$), then $\B$ is an affine $\Lambda$-building of type $\A = \A(\Sigma^\vee, \Lambda, \Lambda^{\operatorname{rk}(G)})$.
\end{reptheorem} 

In Section \ref{sec:assoc-buildings}, we investigate the building at infinity $\partial_\infty \B$ and the local building $\Delta_o\B$ at a base point $o\in \B$, as introduced in \cite{Ben1} and \cite{Sch09}.

\begin{reptheorem}{thm:building_at_infty}
	The group $G_\F$ acts strongly transitively on $\partial_\infty\B$, which can therefore be identified with the spherical building associated to $G_\F$.
\end{reptheorem}
\begin{reptheorem}{thm:germ_building}
	Let $k$ be the residue field of the valued field $\F$ and $G_k$ the $k$-extension of $G_\F$. The group $G_k$ acts strongly transitively on $\Delta_o \B$, which can therefore be identified with the spherical building associated to $G_k$.
\end{reptheorem}

Real closed fields $\F$ are ordered fields that have the same first-order logic theory as the reals. In \cite{AppAGRCF}, the author used a transfer principle from model theory to recover many results about Lie groups in the setting of $G_{\F}$; these include group decompositions, subgroup structure and Kostant convexity. The paper \cite{AppAGRCF} may be considered as a prerequisite to this paper, as all the results there are used in this paper.

\subsection{The construction of \texorpdfstring{$\B$}{B}}
Since $\operatorname{SL}(n,\R)$ acts transitively on 
$$
P_\R \coloneq \left\{ x \in \mathbb{R}^{n\times n} \colon x = x\tran , \, \det(x)=1 , \, x \text{ is positive definite } \right\}
$$
(by $g.x \coloneq gxg\tran$) and the stabilizer of $\operatorname{Id} \in P_\R$ is $\operatorname{SO}(n)$, $P_\R$ is a semialgebraic (defined by equalities and inequalities of polynomials) model for the symmetric space $\operatorname{SL}(n,\R)/\operatorname{SO}(n,\R)$. 
Every symmetric space of non-compact type can be realized as a totally geodesic submanifold of $P_\R$. In fact, for every symmetric space $M$ of non-compact type there exists a semisimple self-adjoint ($\forall g \in G_\R, g\tran \in G_\R$) algebraic subgroup $G_\R <\operatorname{SL}(n,\R)$ such that the orbit $\X_\R \coloneq G_\R .\operatorname{Id} \subseteq P_\R$ is a semialgebraic model for $M$ \cite[Theorem 2.6.5]{Ebe}. The symmetric space $\X_\R$ comes equipped with a $G_\R$-invariant Riemannian distance: the maximal flats are isometric to a Euclidean vector space and for any two points there is a maximal flat containing the two, so the distance is given as the norm of the difference of the two corresponding points in the vector space.

The construction of the building $\B$ mimics the above description closely, the main difference being, that we now work over a valued real closed field $\F$, see Sections \ref{sec:valued_real_closed_fields} and \ref{sec:building_def} for definitions. In real algebraic geometry, semialgebraic sets are considered over general real closed fields. We call the $\F$-extension $\X_\F$ of $\X_\R$ the \emph{non-standard symmetric space}. The $\F$-extension $G_\F$ of $G_\R$ acts transitively on $\X_\F$ with stabilizer
$$
K_\F \coloneq G_\F \cap \operatorname{SO}(n,\F) = \operatorname{Stab}_{G_\F}(\Id).
$$ 
Let $A_\F < G_\F$ be the $\F$-extension of the semialgebraic connected component of a maximal self-adjoint $\F$-split torus of $G_\F$ containing the identity \cite[Theorem \ref{I-thm:split_tori}]{AppAGRCF}. For any two points $x,y \in \X_\F$, let $g\in G_\F$ such that $y=g.x$. The Cartan decomposition $G_\F = K_\F A_\F K_\F$ \cite[Theorem \ref{I-thm:KAK}]{AppAGRCF} can then be used to obtain $a \in A_\F$ such that $g=kak'$ for some $k$, $k' \in K_\F$. In fact, $a$ can be taken in lie in the non-standard fundamental Weyl cone
$$
A_\F^+ \coloneq \left\{ a \in A_\F \colon \chi_\alpha(a) \geq 1 \text{ for all } \alpha \in \Sigma_{>0} \right\},
$$
where $\chi_\alpha$ are algebraic characters associated to the root system $\Sigma$ of $G_\R$. We prove that this defines a map.
\begin{replemma}{lem:Cartan_projection}
	The Cartan-projection $\delta_\F \colon \X_\F \times \X_\F  \to \ApF , \,(x,y) \mapsto a$
	is well-defined and invariant under the action of $G_\F$.
\end{replemma} 
We then define the semialgebraic multiplicative norm
\begin{align*}
	N_\F \colon A_\F & \to \F_{\geq 1} , \ 
	a \mapsto \prod_{\alpha \in \Sigma} \max \left\{ \chi_\alpha(a), \chi_\alpha(a)^{-1} \right\}.
\end{align*}
Over the reals, $d \coloneq \log \circ N_\R \circ \delta_\R$ defines a Finsler-metric on $\X_\R$. However, the logarithm may not be defined on $\F$, since it is not a semialgebraic function, but we can replace $\log$ by a order compatible valuation $v \colon \F_{\neq 0} \to \Lambda$, where $\Lambda$ is an ordered abelian group, to obtain a symmetric non-negative function
\begin{align*}
	d \colon \X_\F \times \X_\F &\to \Lambda , \ (x,y) \mapsto (-v)\left( N_\F \left( \delta_\F (x,y)\right) \right).
\end{align*}
The function $d$ is not quite a $\Lambda$-metric, since there are points $x\neq y \in \X_\F$ with $d(x,y) = 0$, but we have the following.
\begin{reptheorem}{thm:pseudodistance}
	The function $d \colon \X_\F \times \X_\F \to \Lambda$ is a pseudo-distance.
\end{reptheorem}
The proof of Theorem \ref{thm:pseudodistance} relies on the Iwasawa retraction whose properties follow from Kostant convexity, see \cite[Theorems \ref{I-thm:KAU} and \ref{I-thm:kostant_F}]{AppAGRCF}.
\begin{reptheorem}{thm:UAKret}
	The Iwasawa retraction $\rho \colon \X_\F \to A_\F.\operatorname{Id}$ is a $d$-diminishing retraction to $A_\F.\operatorname{Id}$.
\end{reptheorem} 
The main object of our investigation is defined as the quotient $\B \coloneq \X_\F/\!\!\sim$, where $x \sim y$ whenever $d(x,y)=0 \in \Lambda$. By Theorem \ref{thm:pseudodistance}, $\B$ is a $\Lambda$-metric space. We note that $G_\F$ acts transitively by isometries on $\B$. Let $o \coloneq [\operatorname{Id}] \in \B$. In Section \ref{sec:Xapartment} we verify that $\A \coloneq A_\F.o  \subseteq \B$ is an apartment in the sense of affine $\Lambda$-buildings. The inclusion $f_0 \colon \A  \to \B$ can be used to define an atlas $\Fun \coloneq \left\{ g.f_0 \colon \A \to \B \colon g \in G_\F \right\}$. The main result, Theorem \ref{thm:B_is_building}, states that when $\Sigma$ is reduced, the $\Lambda$-metric space $\B$ together with the atlas $\Fun$ is an affine $\Lambda$-building.

The definition of $\B$ required relatively few results about the structure of algebraic groups compared with the definition of the building in Bruhat-Tits \cite{BrTi}. However, the proof of our main result that $\B$ is an affine $\Lambda$-building does rely on a deep understanding of $G_\F$ and the action of $G_\F$ on $\B$. Most of the results about $G_\F$ and $\X_\F$ can be deduced from known results for Lie groups using the transfer principle in real algebraic geometry. However, since the valuation $v$ is not semialgebraic, statements about $\B$ typically have to be proven more directly. We highlight some of these results of independent interest.
 
\subsection{The action of \texorpdfstring{$G_\F$}{G(F)} on \texorpdfstring{$\B$}{B}}

A central theme in the proof of Theorem \ref{thm:B_is_building} is to study the pointwise stabilizers of subsets of the apartment $\A \subseteq \B$. Recall from \cite[Section \ref{I-sec:decompositions}]{AppAGRCF} that an order on the root system $\Sigma$ allows us to define $U_\F$ as the exponential of the sum of root spaces corresponding to positive roots. Let $N_\F$ and $M_\F$ be the $\F$-extensions of the semialgebraic groups $N_\R \coloneq \operatorname{Nor}_{K_\R}(A_\R)$ and $M_\R \coloneq \operatorname{Cen}_{K_\R}(A_\R)$. Let $O \coloneq \left\{ a \in \F \colon (-v)(a) \leq 0 \right\}$ be the valuation ring associated to the valuation $v$ and let $E_\F(O) \coloneq E_\F \cap \operatorname{SL}(n,O)$ for any semialgebraic subset $E_\F \subseteq G_\F$.

The stabilizer of $o \in \B$ was calculated by \cite{Tho} for the special case when $\F$ is a Robinson field. For general fields, it has been suggested by \cite{KrTe} and independently proven by \cite{BIPP23arxiv}.
 
\begin{reptheorem}{thm:stab} 
	The stabilizer of a base point $o\in\B$ in $G_\F$ is $G_\F(O)$.
\end{reptheorem}

As a consequence of the Iwasawa retraction Theorem \ref{thm:UAKret}, we give an Iwasawa decomposition of the stabilizer of $o$.

\begin{repcorollary}{cor:UAKO}
	There is an Iwasawa decomposition $G_\F(O) = U_\F(O) A_\F(O) K_\F$, meaning that for every $g \in G_\F(O)$ there are unique $u \in U_\F(O)$, $a\in A_\F(O)$ and $k\in K_\F=K_\F(O)$ with $g=uak$.
\end{repcorollary}

In Section \ref{sec:Bisbuilding} we implement some ideas from \cite{BrTi} in the semialgebraic setting. For $\alpha \in \Sigma$, the root groups are $(U_\alpha)_\F \coloneq \exp( (\frakg_\alpha \oplus \frakg_{2\alpha})_\F)$, 
and we introduce explicit root group valuations $\varphi_\alpha \colon (U_\alpha)_\F \to \Lambda \cup \{-\infty\}$ in Subsection \ref{sec:root_group_valuation}. Taking a closer look at the unipotent group, we notice that if $u \in U_\F$ sends some point of the apartment to some other point of the apartment, then the two points have to be the same.

\begin{repproposition}{prop:UonA}
	For all $u\in U_\F$ and $a \in A_\F$, $
	ua.o \in \mathbb{A} \iff ua.o = a.o. 
	$
\end{repproposition}

In fact, the fixed point set of elements in $(U_\alpha)_\F$ is a half-apartment.

\begin{repproposition}{prop:Ualphaconv}
	Let $\alpha \in \Sigma$. For $u \in (U_\alpha)_\F$ we have
	$$
	\left\{ p \in \mathbb{A} \colon u.p  \in \mathbb{A}  \right\} = \left\{ a.o \in \A \colon \varphi_\alpha(u) \leq (-v)\left(\chi_\alpha\left(a\right)\right)  \right\}
	$$
	and therefore this set is a half-apartment when $u \neq \operatorname{Id}$. 
\end{repproposition}
 
 For $U_\F$, this can be upgraded to finite intersections of half-apartments.
 
 \begin{repproposition}{prop:Uconv}
 	For $u\in U_\F$ there are $k_\alpha \in \Lambda \cup \{-\infty \}$ for $\alpha \in \Sigma_{>0}$ such that
 	$$
 	\{p \in \A \colon u.p \in \A\} = \left\{ a.o \in \A \colon  k_\alpha \leq (-v)\left(\chi_\alpha \left(a\right)\right) \text{ for all } \alpha \in \Sigma_{>0}   \right\}
 	$$
 	and therefore the set of fixed points is a finite intersection of half-apartments. If $u$ fixes all of $\A$, then $u = \operatorname{Id}$. 
 \end{repproposition}
 
 We proceed to describe the pointwise stabilizers of the whole apartment $\A$, the half-apartments
 $$
 H_\alpha^+ \coloneq \left\{ a.o \in \A \colon (-v)(\chi_\alpha(a)) \geq 0 \right\}
 $$
 and the fundamental Weyl chamber
 $$
 C_0 \coloneq \bigcap_{\alpha >0} H_\alpha^+.
 $$
 
 \begin{reptheorem}{thm:stab_A}
 	The pointwise stabilizer of $\A$ in $G_\F$ is $ A_\F(O)M_\F$.
 \end{reptheorem}
\begin{reptheorem}{thm:NalphaO_fixes_H}
	Let $\alpha \in \Sigma$. The pointwise stabilizer of $
	H_\alpha^+ $ is $(U_\alpha)_\F(O) A_\F(O) M_\F$.
\end{reptheorem} 
\begin{reptheorem}{thm:NO_fixes_s0}
	The pointwise stabilizer of $C_0$ in $G_\F$ is 
	$U_\F(O)A_\F(O)M_\F.$
\end{reptheorem}

In Subsection \ref{sec:BT_rank_1} we take a closer look at the rank one subgroup $(L_{\pm \alpha})_\F < G_\F$ introduced in \cite[Section \ref{I-sec:rank1}]{AppAGRCF}. We now restrict to groups $G_\F$ with reduced root systems, to be able to use the Jacobson-Morozov Lemma to find elements $m(u)\in N_\F$ representing a reflection along some affine hyperplane $M_{\alpha, \ell} = \left\{ a.o \in \A \colon (-v)(\chi_\alpha(a)) = \ell \right\}$ for $\alpha \in \Sigma$ and $\ell \in \Lambda$. A careful analysis of the rank one subgroup $L = \left\langle  (U_\alpha)_\F, (U_{-\alpha})_\F \right\rangle < (L_{\pm \alpha})_\F$ results in the following decomposition of its stabilizer.

	\begin{repproposition}{prop:stab_L_Omega} Assume $\Sigma$ is reduced. Let $\alpha \in \Sigma$, $\Omega \subseteq \A$ a non-empty finite subset. The pointwise stabilizers
		\begin{align*}
			L_\Omega &\coloneq L \cap \operatorname{Stab}_{G_\F}(\Omega), \\
			U_{\alpha,\Omega} &\coloneq (U_\alpha)_\F \cap \operatorname{Stab}_{G_\F}(\Omega), \\
			U_{-\alpha, \Omega} &\coloneq (U_{-\alpha})_\F \cap \operatorname{Stab}_{G_\F}(\Omega), \\
			N_\Omega &\coloneq \operatorname{Nor}_{G_\F}(A_\F) \cap \operatorname{Stab}_{G_\F}(\Omega)
		\end{align*}
	satisfy $L_\Omega = \left\langle U_{\alpha,\Omega}, U_{-\alpha, \Omega} , (N_\Omega \cap L) \right\rangle = U_{\alpha,\Omega} U_{-\alpha, \Omega}(N_\Omega \cap L)$.
\end{repproposition}
In Subsection \ref{sec:BT_higher_rank} we upgrade the rank one result to $G_\F$. For $\Omega \subseteq \A$, let
\begin{align*}
	\hat{P}_{\Omega} &\coloneq \left\langle N_\Omega, U_{\alpha, \Omega} \colon \alpha \in \Sigma \right\rangle, \\
	U_\Omega^+ &\coloneq \left\langle U_{\alpha,\Omega} \colon \alpha \in \Sigma_{>0}\right\rangle, \\
	U_{\Omega}^-&\coloneq \left\langle U_{\alpha,\Omega} \colon \alpha \in \Sigma_{<0} \right\rangle.
\end{align*}
\begin{reptheorem}{thm:BTstab}
	Assume $\Sigma$ is reduced. The pointwise stabilizer of any subset $\Omega \subseteq \A$ satisfies
	$
	\operatorname{Stab}_{G_\F}(\Omega) = \hat{P}_\Omega = U_\Omega^+ U_{\Omega}^- N_{\Omega} .
	$
\end{reptheorem}
The proof of Theorem \ref{thm:BTstab} is tightly intertwined with the proof of axiom (A2) of affine buildings. It first relies on the following \emph{mixed Iwasawa} decomposition for $G_\F$. 

\begin{reptheorem}{thm:BT_mixed_Iwasawa}
	Assume $\Sigma$ is reduced. Then $G_\F = U_\F \cdot \operatorname{Nor}_{G_\F}(A_\F) \cdot G_\F(O)$.
\end{reptheorem}

With the help of the mixed Iwasawa decomposition we obtain Theorem \ref{thm:BTstab} in the case of finite subsets $\Omega \subseteq \A$. This statement for finite subsets is then enough to prove the second part of axiom (A2). The full Theorem \ref{thm:BTstab} then follows from this part of (A2) and the rest of axiom (A2) follows from the full Theorem \ref{thm:BTstab}. We also obtain a description of the (not necessarily pointwise) stabilizer of $\A$.
\begin{repproposition}{thm:stab'_A}
	The stabilizer of $\A$ is $\operatorname{Nor}_{G_\F}(A_\F) = A_\F M_\F$.
\end{repproposition}

\subsection{Related work}

There are multiple connections between symmetric spaces and buildings. Mostow associated a building at infinity to symmetric spaces of non-compact type to prove rigidity results on lattices in higher rank Lie groups \cite{Mos73}. Kleiner and Leeb showed that the asymptotic cone of a symmetric space of non-compact type is a non-discrete affine $\R$-building and used this result to show rigidity of quasi-isometries in symmetric spaces of non-compact type \cite{KlLe}. The asymptotic cone is a special case of the building $\B$, when the field $\F$ is a Robinson field \cite[Theorem 1.11]{BIPP23arxiv}. Motivated to give a simpler proof of the rigidity of quasi-isometries, Kramer and Tent \cite{KrTe} suggested the construction of $\B$, which is carried out in detail in the present work. 

A newer promising direction where $\B$ plays an important role is the study of geometric structures on surfaces, and more generally the study of character varieties, for a survey see \cite{Wie18}. Character varieties are spaces of (equivalence classes of) representations of discrete groups into Lie groups. Their properties can be studied using various compactifications. Thurston \cite{Thu88} constructed a compactification of the space of discrete and faithful representations of a fundamental group of a surface of genus at least two into the isometry group of the hyperbolic plane. Thurston described the boundary points in terms of actions on certain $\R$-trees and used a fixed point theorem to classify elements of the mapping class group. In higher rank, similar compactifications have been developed, such as the marked length compactification \cite{Bon88,Par12}. The real spectrum compactification \cite{Bru1,BIPP21,BIPP23arxiv} is a finer compactification with some good properties, for instance it preserves connected components. Similarly to Thurston's description, \cite{Bru1,Bru2} and \cite[Theorem 1.1]{BIPP23arxiv} show that the boundary points correspond to an action on $\B$, which is a $\Lambda$-tree in rank one, or more generally a $\Lambda$-building, as we show in Theorem \ref{thm:B_is_building}, see \cite{Jae25arxiv} for more interpretations of boundary points. The real algebraic approach taken for the real spectrum compactification harmonizes closely with the construction of $\B$. This approach has also proved useful to study Hitchin components and projective structures \cite{Fla22arxiv,FlPa25arxiv}. A unifying approach to geometric structures are $\Theta$-positive representations: it is now known that all $\Theta$-positive components consist of discrete and faithful representations \cite{BGLPW24}. We would like to point out that the technical condition on the root system in Theorem \ref{thm:B_is_building} is always satisfied in the case of $\Theta$-positive representations \cite{GuWi24}.  Theorems \ref{thm:germ_building} and \ref{thm:building_at_infty} about the global and local structure of $\B$ are promising tools to understand degenerations of geometric structures. When $\Lambda < \R$ it is known that $\B$ is contained in some affine $\R$-building \cite{ScSt12} which in turn is contained in a complete affine $\R$-building \cite[Lemma 4.4]{Str11}, which is a CAT(0) space. The completion $\overline{\B}$ of $\B$ (viewed as a convex subset of a complete CAT(0)-space) is then a CAT(0) space and fixed point theorems can be applied.

The construction of $\B$ is also interesting from the point of view of the theory of affine $\Lambda$-buildings. In the beginning, the only known examples of affine $\Lambda$-buildings with $\Lambda \neq \R$ were the ones of type $A_n$ in Bennet's paper \cite{Ben1}, where he introduced the concept. In their paper on functoriality \cite{ScSt12}, Schwer and Struyve showed how to construct new $\Lambda$-buildings from others. Recently, H\'ebert, Izquierdo and Loisel \cite{HIL23} generalized the Bruhat-Tits construction to $\Lambda\neq \R$. Their construction works for split groups, or quasi-split groups over Henselian fields. The real closed fields we consider are not always Henselian and the groups we consider need to be semi-simple, but not necessarily quasi-split, so our construction $\B$ gives new examples of affine $\Lambda$-buildings. While the definition of the Bruhat-Tits building \cite{BrTi,BrTi84} and the generalization in \cite{HIL23} requires quite a deep understanding of algebraic groups, the definition of $\B$ is relatively elementary: it is a quotient of a non-standard symmetric space. However for the proof that $\B$ is an affine $\Lambda$-building, we rely extensively on the theory of algebraic groups over real closed fields as developed in \cite{AppAGRCF}. To our knowledge, all examples of non-discrete affine buildings come from algebraic groups. It would be interesting to exhibit exotic non-discrete affine buildings, analogous to the discrete setting. Discrete affine buildings have been fully classified \cite{Wei09}, a similar classification in the non-discrete case is open, though partial results in the case of $\mathbb{R}$-buildings can be found in \cite{Tit86}.

\subsection{Acknowledgements}

This paper (except for Section \ref{sec:assoc-buildings} which is new) is the main second part of the author's doctoral thesis \cite{App24thesis}, but Section \ref{sec:Xapartment} has been substantially reformulated. The first part of the thesis is contained in \cite{AppAGRCF} which at the time of writing is under review for publication. The author would like to thank Marc Burger for continued support and is grateful for useful discussions with Petra Schwer and Auguste Hébert, in particular for their insistence on reading \cite{BrTi} carefully. The author is thankful for valuable inputs by Linus Kramer, Gabriel Pallier, Luca de Rosa, Xenia Flamm, Victor Jaeck and Segev Gonen Cohen. The author greatly appreciated the detailed and careful comments by an anonymous reviewer.

%% file: affine_L_buildings.tex
	\section{Affine \texorpdfstring{$\Lambda$}{Λ}-buildings}\label{sec:affine_L_buildings} 
	
	
	\subsection{\texorpdfstring{$\Lambda$}{Λ}-metric spaces}
	
	An abelian group $(\Lambda, +)$ with a total order such that $x,y \geq 0$ implies $x+y \geq 0$ for all $x,y \in \Lambda$ is called an \emph{ordered abelian group}. If an ordered abelian group $\Lambda$ is isomorphic (as an ordered group) to a subgroup of $(\R,+)$, it is called \emph{Archimedean}. Hahn's embedding theorem classifies all ordered abelian groups.
	
	\begin{theorem}[\cite{Hah07}] For every ordered abelian group $(\Lambda,+)$ there is a ordered set $\Omega$ such that $\Lambda<\R^{\Omega}$ as an ordered subgroup, where 
		$$
		\R^\Omega = \{ f\colon \Omega \to \R \colon \operatorname{supp}(f) \text{ is contained in a well-ordered set } \}
		$$
		is equipped with the lexicographical ordering.  
	\end{theorem} 
	Let $(\Lambda,+)$ be a non-trivial ordered abelian group. In particular, $\Lambda$ has no torsion. 
	We will use the following generalization of metric spaces.
	If $X$ is a set and $d \colon X \times X \to \Lambda$ is a function, we call $(X,d)$ a \emph{$\Lambda$-pseudometric space} if for all $x,y,z \in X$
	\begin{enumerate}
		\item [(1)] $d(x,x) = 0$, $d(x,y) \geq 0$
		\item [(2)] $d(x,y) = d(y,x)$
		\item [(3)] $d(x,y) \leq d(x,z) + d(z,y)$. 
	\end{enumerate}
	If in addition, $d(x,y) = 0$ implies $x=y$, then $(X,d)$ is called a \emph{$\Lambda$-metric space}. The axioms are direct generalizations of the notions of (pseudo)metric spaces when $\Lambda = \R$. Any $\Lambda$-pseudometric space can be turned into a $\Lambda$-metric space by quotienting out the equivalence relation of having distance $0$.
	
	Important examples of $\Lambda$-metric spaces are $\Lambda$-trees \cite{Chi01}, which serve as rank one examples of affine $\Lambda$-buildings.
	
	\subsection{Root systems}\label{sec:root_system}
	
	A detailed treatment of root systems can be found in \cite{Bou08}. Let $V$ be a finite dimensional Euclidean vector space with scalar product $\langle \cdot, \cdot \rangle$. For $\alpha \in V$, the reflection along the hyperplane $$M_\alpha = \{ \beta \in~V \colon \langle \alpha , \beta \rangle =0\}$$ is given by
	$$
	r_{\alpha} (\beta) = \beta - 2 \frac{\langle \alpha , \beta \rangle}{ \langle \alpha, \alpha \rangle} \alpha . 
	$$
	A pair $(\Phi,V)$ where $\Phi \subseteq V$ is called a \emph{root system} if 
	\begin{enumerate}
		\item [(R1)] $\Phi$ is finite, symmetric ($\Phi = -\Phi$), spans $V$ and does not contain $0$.
		\item [(R2)] For every $\alpha \in \Phi$ the reflection $r_\alpha \colon V \to V$ preserves $\Phi$.
	\end{enumerate} 
	When $V$ can be determined from the context, the root system may be denoted by just $\Phi$. A root system is called \emph{crystallographic} if it satisfies the integrality condition
	\begin{enumerate}
		\item [(R3)] If $\alpha, \beta \in \Phi$, then $2\langle \alpha, \beta \rangle / \langle \alpha, \alpha \rangle \in \Z$.
	\end{enumerate} 
	For crystallographic root systems, $\operatorname{Span}_{\mathbb{Z}}(\Phi)$ is a lattice in $V$. A root system is called \emph{reduced} if
	\begin{enumerate}
		\item [(R4)] For every $\alpha \in \Phi$, $\R \cdot \alpha \cap \Phi = \{ \pm \alpha\}$.
	\end{enumerate}
	
	\begin{figure}[h]
		\centering
		\includegraphics[width=0.8\linewidth]{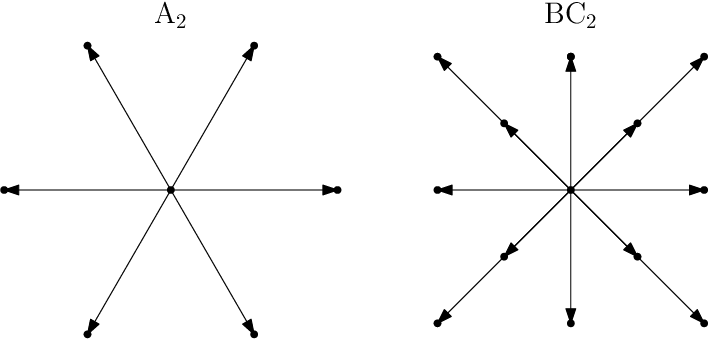}
		\caption{ Two examples of crystallographic root systems. Type $\operatorname{A}_2$ on the left, type $\operatorname{BC}_2$ on the right. While $\operatorname{A}_2$ is reduced, $\operatorname{BC}_2$ is not. }
		\label{fig:examples_root_systems}
	\end{figure}
	
	Figure \ref{fig:examples_root_systems} gives examples of two crystallographic root systems, one of them reduced the other not.	The \emph{spherical Weyl group $W_s$} of a root system is the subgroup of isometries of $V$ generated by reflections along the hyperplanes $M_\alpha$ for $\alpha \in \Phi$. A subset $\Delta \subseteq \Phi$ is a \emph{basis of $\Phi$}, if it is a vector space basis of $V$ such that all roots $\beta \in \Phi$ can be written as $\beta = \sum_{\delta \in \Delta} \lambda_\delta \delta$ with $\lambda_\delta \in \Z$ for $\delta \in \Delta$ and such that all $\lambda_{\delta}$ have the same sign (all are non-negative or non-positive). Every root system has a basis and the spherical Weyl group acts transitively on the set of bases of $\Phi$. The cardinality of a basis $\Delta$ is called the \emph{rank} of the root system $\Phi$ and coincides with $\dim(V)$.
	 
	The connected components of $V\setminus \bigcup_{\alpha \in \Phi} M_\alpha$ are called \emph{(open) chambers}. The spherical Weyl group acts simply transitively on the set of chambers. Given a basis $\Delta$, the \emph{fundamental Weyl chamber of $\Delta$} is
	$$
	C_0(\Delta) = \left\{ v\in V \colon \langle v, \alpha \rangle >0 \ \text{ for all } \alpha \in \Delta \right\}.
	$$ 
	Conversely, every chamber $C$ determines a basis $\Delta(C)$ formed by those $\alpha \in \Phi$, that satisfy $\langle \alpha, v \rangle >0$ for all $v \in C$ and that cannot be written as a sum of other such elements of $\Phi$, see \cite[VI\S1.5]{Bou08}. A basis $\Delta$ determines the set  of \emph{positive roots} $\Phi_{>0} \subseteq \Phi$ given by positive integer combinations of elements in the basis. A basis thus determines a partial order on $\Phi$ by $\alpha < \beta$ if $\beta-\alpha \in \Phi_{>0}$.  A total order on $\Phi$ with the same positive elements can be obtained by choosing an order on the basis and extending it lexicographically to $ \Phi \subseteq \operatorname{Span}_{\mathbb{Z}}(\Phi)$.


The scalar product defines an isomorphism of Euclidean vector spaces $V \cong V^\star$, $v \mapsto v^\star$ by the defining property that $v^\star(w) = \langle v, w \rangle$ for $v,w \in \Phi$, where $V^\star$ is equipped with the scalar product $\langle v^\star , w^\star  \rangle \coloneqq  \langle v,w \rangle$. The \emph{coroot\footnote{Sometimes $2\alpha/\langle \alpha , \alpha \rangle \in V$ is called a coroot instead of its dual. }  of} $\alpha \in \Phi$ is 
$$
 \alpha^\vee \coloneqq 2\frac{\alpha^\star}{\langle \alpha, \alpha \rangle} \in V^\star.
$$  
The \emph{coroot system} $(\Phi^\vee, V^\star)$, where $\Phi^\vee \coloneqq \{ \alpha^\vee \colon \alpha \in \Phi  \}$ is also a root system \cite[VI\S1.1]{Bou08}, but $\Phi$ and $\Phi^\vee$ are not isomorphic in general. When $\Phi$ is a crystallographic root system, there is a non-degenerate bilinear form
\begin{align*}
	b \colon \operatorname{Span}_{\mathbb{Z}}(\Phi) \times \operatorname{Span}_{\mathbb{Z}}( \Phi^\vee) & \to \mathbb{Z} \\
	(\alpha, \beta^\vee) & \mapsto 2\frac{\langle \alpha, \beta \rangle}{\langle \beta , \beta \rangle} = \beta^\vee(\alpha)
\end{align*}
taking values in $\mathbb{Z}$. Given a basis $\Delta = \{\delta_1, ,\ldots, \delta_r\}$, the matrix $(B_{ij})_{ij} \coloneqq b(\delta_i, \delta_j^\vee)$ is called the \emph{Cartan matrix}. Note that we have the following calculation rules, but in general $(\alpha + \beta)^\vee \neq \alpha^\vee + \beta^\vee$.

\begin{lemma}
	Let $\Phi$ be a crystallographic root system and $\alpha, \beta \in \Phi$. Then
	\begin{itemize}
		\item [(i)] $r_\beta(\alpha) = \alpha - b(\alpha, \beta^\vee) \cdot \beta 
		= \alpha - \beta^\vee(\alpha) \cdot \beta$,
		\item [(ii)]  $b(\alpha, \alpha^\vee) = 2$,
		\item [(iii)] $(\alpha^\vee)^\star = (\alpha^\star)^\vee$,
		\item [(iv)] $(\alpha^\vee)^\vee = \alpha$. 
	\end{itemize}
\end{lemma}
\begin{proof}
	Item (i) is immediate from the definition of the reflection $r_\beta$, and (ii) is the statement that $\Phi$ is crystallographic. For all $H \in \operatorname{Span}_{\mathbb{Z}}(\Phi^\vee)$ we have 
	$$
	(\alpha^\star)^\vee (H) = 2 \frac{(\alpha^\star)^\star(H)}{\langle \alpha^\star , \alpha^\star\rangle}
	= 2 \frac{\langle \alpha^\star, H \rangle }{\langle \alpha, \alpha \rangle} = \langle \alpha^\vee , H \rangle = (\alpha^\vee)^\star(H),
	$$
	so (iii) holds and then (iv) follows from
	\begin{equation*}
	(\alpha^\vee)^\vee = 2\frac{(\alpha^\vee)^\star}{\langle \alpha^\vee, \alpha^\vee\rangle} = 2 \frac{(\alpha^\star)^\vee}{\langle 2\frac{\alpha^\star}{\langle \alpha, \alpha \rangle }, 2\frac{\alpha^\star}{\langle \alpha, \alpha \rangle } \rangle}
	= 2\frac{ 2\frac{(\alpha^\star)^\star}{\langle \alpha^\star, \alpha^\star \rangle}}{4 \frac{1}{\langle \alpha ,\alpha \rangle^2}\langle \alpha, \alpha \rangle} = (\alpha^\star)^\star = \alpha. \qedhere
\end{equation*}
\end{proof}

	\subsection{Apartments}\label{sec:modelapartment}

	We introduce the model apartment $\A$. A more detailed introduction can be found in \cite{Ben1}, and including the non-crystallographic case in \cite{Sch09}. The root systems in our setting will arise from algebraic groups, and are therefore crystallographic, but possibly not reduced.
	
	Let $\Lambda$ be a non-trivial ordered abelian group. 
	 Let $(\Phi,V)$ be a crystallographic root system and $\operatorname{Span}_{\Z}(\Phi) \subseteq V$ its associated lattice. 
	Since both $\operatorname{Span}_{\Z}(\Phi)$ and $\Lambda$ are $\Z$-modules, we can define the \emph{model apartment} 
	$$
	\A \coloneqq \operatorname{Span}_{\Z}(\Phi) \otimes_{\Z} \Lambda .
	$$
	For a basis $\Delta = \{\delta_1 , \ldots , \delta_r \}$, the $\mathbb{Z}$-bilinear map
	\begin{align*}
		\operatorname{Span}_{\Z}(\Phi) \times \Lambda & \to \Lambda^{r} \\
		\left(\sum_{i=1}^r z_i \delta_i , \lambda \right)& \mapsto (z_1 \lambda, \ldots , z_r \lambda)
	\end{align*}
    extends to a $\mathbb{Z}$-module isomorphism
	$
	\A 
	\cong \Lambda^{r}
	$, giving rise to the notation
	$$
	\A \cong \left\{ \sum_{\delta \in \Delta} \lambda_\delta \delta \colon \lambda_\delta \in \Lambda \right\}.
	$$
	Note that if $\Lambda$ is a $\Q$-vector space, then we also have
	$$
	\A = \operatorname{Span}_\Q(\Phi) \otimes_{\Q} \Lambda.
	$$
	 The action of $W_s$ on $\Phi$ extends to an action on $\A$, fixing $0 \in \A$. The apartment $\A$ itself acts by translation on $\A$. Let $T\subseteq \A$ be a subgroup of the translation group normalized by $W_s$, meaning that $wTw^{-1} = T$ for all $w\in W_s$. Then $W_a = T \rtimes W $ is called the \emph{affine Weyl group} and acts on $\A$. If $T=\A$, $W_a$ is called the \emph{full affine Weyl group}. We denote the data of a model apartment together with an affine Weyl group as $\A = \A(\Phi,\Lambda,T)$ since the root system $\Phi$, the ordered abelian group $\Lambda$ and the subgroup of translations $T$ fully determine $\A$ and the affine Weyl group. 
	 
	 
	 For any root $\alpha \in \Phi$, the reflection $r_\alpha \colon \Phi \to \Phi$ extends to a \emph{reflection}
	 \begin{align*}
	 	r_{\alpha} \colon \A & \to \A \\
	 	\sum_{\delta \in \Delta} \lambda_\delta \delta & \mapsto \sum_{\delta \in \Delta} \lambda_\delta r_\alpha (\delta) 
	 \end{align*}
	 and for any $w \in W_a$, the conjugates $w \circ r_\alpha \circ w^{-1} \in W_a$ are called \emph{affine reflections}. 
	 The bilinear form $b$ 
	 extends naturally to the bilinear form
	 \begin{align*}
	 b \colon 	\A \times \operatorname{Span}_{\mathbb{Z}}(\Phi^\vee)  & \to \Lambda \\
	 	\left( \sum_{\delta \in \Delta} \lambda_\delta \delta , \sum_{\varepsilon \in \Delta } \mu_{\varepsilon} \varepsilon^\vee \right) & \mapsto \sum_{\delta \in \Delta} \sum_{ \varepsilon\in \Delta} \mu_{\varepsilon}   b( \delta, \varepsilon^\vee)  \lambda_{\delta}
	 \end{align*}
	 but we notice that $b$ can not be extended to all of $\A \times \operatorname{Span}_{\Z}(\Phi^\vee) \otimes_{\Z} \Lambda$, since $\Lambda$ may not have a multiplication. For $\alpha\in \Phi$ and $\lambda \in \Lambda$ we define the \emph{affine wall}
	 $$
	 M_{\alpha,\lambda} \coloneqq
	 \left\{x \in \A \colon  b(x, \alpha^\vee ) = \lambda  \right\}
	 $$
	 as well as the two \emph{affine half-spaces} 
	 \begin{align*}
	 	H_{\alpha, \lambda}^+ &= \left\{ x\in \A \colon    b( x, \alpha^\vee ) \geq \lambda  \right\} \\
	 	H_{\alpha, \lambda}^- &= \left\{ x \in \A \colon  b( x, \alpha^\vee ) \leq \lambda  \right\}
	 \end{align*}
	 defined by $M_{\alpha,\lambda}$. The intersection of the positive half-spaces for all $\delta \in \Delta$ is called the \emph{fundamental Weyl chamber}
	 $$
	 C_0 \coloneqq \left\{ x \in \A \colon b\left(x, \delta^\vee \right) \geq 0 \text{ for all } \delta \in \Delta \right\} = \bigcap_{\delta \in \Delta} H_{\delta,0}^+.
	 $$
	 The images of the fundamental Weyl chamber under the action of the affine Weyl group are called \emph{sectors or chambers}.

     The model apartment $\A$ can be endowed with a $\Lambda$-metric. There are multiple ways to do this and we will describe one. Since $\A$ may not admit a $\Lambda$-valued scalar product, we instead use the $W_s$-invariant $\Lambda$-valued norm $N \colon \A \to \Lambda$ defined by
     $$
     N( x ) = \sum_{\alpha \in \Phi_{>0}} \left| b( x, \alpha^\vee ) \right| ,
     $$ 
     where $|\lambda| = \max\{\lambda, -\lambda\}$ is the absolute value. Then
     $$
     d(x,y) = N(x-y)
     $$
     turns $\A$ into a $\Lambda$-metric space. Note that when $x-y \in C_0$, 
     $$
     d(x,y) =   b\left( x-y , \sum_{\alpha \in \Phi_{>0} } \alpha^\vee \right).
     $$
     Note that the $\Lambda$-metric space $\A$ may not be uniquely geodesic: there may be two or more distinct isometric embeddings of a $\Lambda$-interval with coinciding endpoints. Since $\Lambda$ is just a group, the classical notion of convexity using linear combinations does not make sense. Instead we say that a subset $B \subseteq \A$ is \emph{$W_a$-convex}, if it is a finite intersection of affine half-spaces. Figure \ref{fig:A2_segment} illustrates that $\A$ may not be uniquely geodesic and that the union of all geodesics between two points is a $W_a$-convex set.

	 	\begin{figure}[h]
		\centering
		\includegraphics[width=0.6\linewidth]{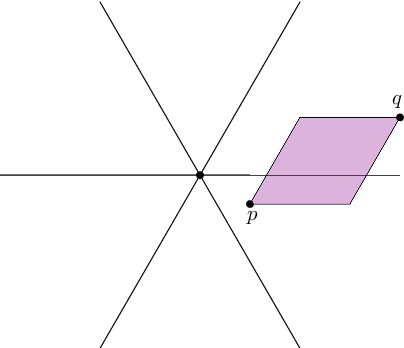}
		\caption{ Given two points $p,q \in \A$, the interval $\{ r\in~\A \colon d(p,r)+d(r,q) = d(p,q)\}$ is a $W_a$-convex set defined by the intersection of four half-planes parallel to the walls in this example of type $\operatorname{A}_2$. }
		\label{fig:A2_segment}
	\end{figure} 
	
	\subsection{Affine \texorpdfstring{$\Lambda$}{Λ}-buildings}\label{sec:lambda_building}
	 
	Let $\A$ be an apartment of type $\A=\A(\Phi,\Lambda,T)$ with affine Weyl group $W_a$ as in the previous section. Let $X$ be a set and $\Fun$ a set of injective maps $\A \to X$. The set $\Fun$ is called the \emph{atlas} or \emph{apartment system} and its elements are called the \emph{charts}. The images $f(\A) $ for $f\in\Fun$ are called \emph{apartments}. The images of walls, half-spaces and chambers are called \emph{walls, half-apartments} and \emph{sectors}. 
	
	The pair $(X,\Fun)$ is called a \emph{generalized affine building} or a \emph{affine $\Lambda$-building} of type $\A = \A(\Phi,\Lambda,T)$ if the following six axioms are satisfied. Axioms (A4) and (A6) are illustrated in Figures \ref{fig:A4} and \ref{fig:A6}.
	
	\begin{enumerate}
		\item [(A1)] For all $f\in \Fun$, $w\in W_a$, $f \circ w \in \Fun$.
		\item [(A2)] For all $f,f'\in \Fun$, the set $B= f^{-1}\left( f(\A)\cap f'(\A)  \right)$ is $W_a$-convex and there is a $w\in W_a$ such that
		$
		f|_B = f'\circ w |_B.
		$
		\item [(A3)] For all $x,y\in X$, there is a $f\in \Fun$ such that $x,y\in f(\A)$.
		\item [(A4)] For any sectors $s_1,s_2 \subseteq X$ there are subsectors $s_1' \subseteq s_1, s_2' \subseteq s_2$ such that there is an $f\in \Fun$ with $s_1', s_2' \subseteq f(\A)$.
		\item [(A5)] For every $x\in X$ and $f\in \Fun$ with $x\in f(\A)$ there is a distance diminishing retraction $r_{x,f} \colon X \to f(\A)$ such that $f^{-1}(\{x\}) = \{x\}$.
		\item [(A6)] For $f_1,f_2,f_3 \in \Fun$, if $f_i(\A) \cap f_j(\A)$ are half-apartments for $i\neq j$, then $f_1(\A) \cap f_2(\A) \cap f_3(\A) \neq \emptyset$.
 	\end{enumerate}
 	
 	\begin{figure}[h]
 		\centering
 		\includegraphics[width=0.8\linewidth]{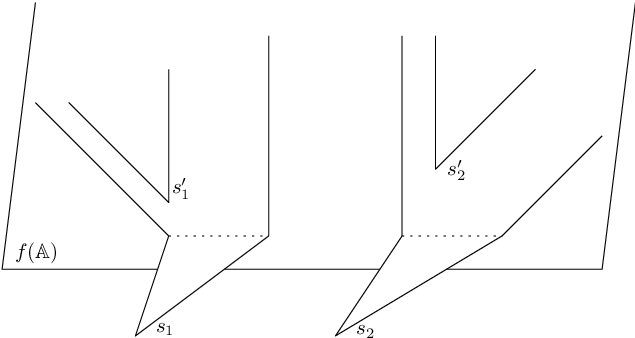}
 		\caption{Axiom (A4) states that while arbitrary sectors $s_1,s_2$ may not lie in a common flat, they contain subsectors $s_1',s_2'$ contained in a common flat $f(\A)$. }
 		\label{fig:A4}
 	\end{figure} 
 	
 	\begin{figure}[h]
 		\centering
 		\includegraphics[width=0.8\linewidth]{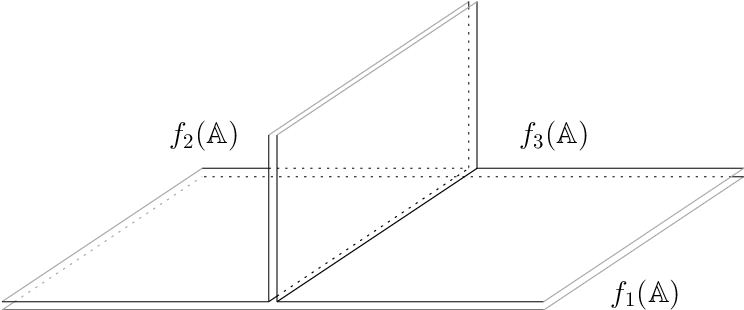}
 		\caption{Axiom (A6) states that whenever three apartments intersect pairwise in a half-apartment, then there is at least one point in the common intersection of all three.}
 		\label{fig:A6}
 	\end{figure}
 	
 	
 	The distance in axiom (A5) is the function $d\colon X\times X \to \Lambda$ induced from the distance on the model apartment $\A$, whose existence follows from axioms (A1), (A2) and (A3): for any two points $x,y\in X$ we use axiom (A3) to find $f\in \Fun$ such that $x,y \in f(\A)$ and define $d(x,y) = d( f^{-1}(x) , f^{-1}(y) ) \in \Lambda $. This is well defined, by axioms (A1) and (A2). A $d$-diminishing retraction is then a function $r \colon X \to f(\A)$ satisfying $r(y)=y$ for all $y\in f(\A)$ and $$
 	d(r(x),r(y)) \leq d(x,y)
 	$$
 	for all $x,y\in X$. Note that while axioms (A1) - (A3) can be used to define a symmetric positive definite function $d$, axiom (A5) can be used to show that $d$ satisfies the triangle inequality. As shown in \cite{BeScSt}, in the presence of the other five axioms, axiom (A5) is equivalent to the triangle inequality. In fact, \cite{BeScSt} contains a number of collections of axioms that characterize affine $\Lambda$-buildings. We will use the following characterization.
 	
 	\begin{theorem}[Theorem 3.1, \cite{BeScSt}] \label{thm:equivalent_axioms}
 		Let $(X,\Fun)$ be a set with an atlas such that the axioms (A1), (A2), (A3) and (A4), as well as
 		\begin{enumerate}
 			\item [(TI)] The function $d$ induced from the distance in apartments satisfies the triangle inequality.
 		\end{enumerate}
 		and the exchange condition
 		\begin{enumerate}
 			\item [(EC)] For $f_1,f_2 \in \Fun$, if $f_1(\A)\cap f_2(\A)$ is a half-apartment, then there exists $f_3 \in \Fun$ such that $f_i(\A)\cap f_3(\A)$ are half-apartments for $i\in \{1,2\}$. Moreover $f_3(\A)$ is the symmetric difference of $f_1(\A)$ and $f_2(\A)$ together with the boundary wall of $f_1(\A) \cap f_2(\A)$.
 		\end{enumerate}
 		are satisfied. Then $(X,\Fun)$ is an affine $\Lambda$-building. 
 	\end{theorem}

 \section{ Valued real closed fields}\label{sec:valued_real_closed_fields}
 
 For an introduction to real closed fields, valued fields and real algebraic geometry we recommend  \cite{BCR}, see also Section \ref{I-sec:real_closed} of \cite{AppAGRCF}. A \emph{real closed field} is an ordered field with the properties that 
 \begin{enumerate}
 	\item [(a)] every positive element has a square root, and
 	\item [(b)] every odd-degree polynomial has a zero.
 \end{enumerate}
 An ordered field is called \emph{Archimedean} if every element is bounded by a natural number. A major tool when working with real closed fields is the following transfer principle from model theory. 
 \begin{theorem}[Transfer principle, \cite{BCR}]\label{thm:logic}
 	Let $\F$ and $\F'$ be real closed fields. Let $\varphi$ be a sentence in first-order logic for ordered fields with parameters in $\F \cap \F'$. Then $\varphi$ is true for $\F$ if and only if $\varphi$ is true for $\F'$, formally $\F \models \varphi \iff \F' \models \varphi$.
 \end{theorem}
 
 \subsection{Valuations}
 
 A subring $O \subseteq \F$ of an ordered field $\F$ is an \emph{order convex subring}, if for all $a,b \in O,\  c \in \F, \ a \leq c \leq b $ implies $c \in O$. Note that every order convex subring is in particular a valuation ring: for all $a \in \F$, we have $a \in O$ or $a^{-1} \in O $. Let $(\Lambda,+)$ be an ordered abelian group. A \emph{valuation} on an ordered field $\F$ is a map $v \colon \F \to \Lambda \cup \{\infty\}$ which satisfies for all $a,b \in \F$
 \begin{itemize}
 	\item [(1)] $v(a) = \infty$ if and only if $a = 0$.
 	\item [(2)] $v(ab) = v(a) + v(b)$.
 	\item [(3)] $v(a+b) \geq \min \{ v(a), v(b)\}$. 
 \end{itemize}
 We say that the valuation is \emph{order compatible}, if $(-v)(a) \geq (-v)(b)$ whenever $a \geq b \geq 0$. We will often use the same letter for $v$ and $v|_{\F_{>0}} \colon \F_{>0} \to \Lambda$. We will often be more interested in $(-v)$ than in $v$, as $(-v)$ is order preserving (when restricted to $\F_{>0}$). There is a correspondence between order convex valuation rings and order compatible valuations.
 
 \begin{theorem}[\cite{Kru}]
 	Every order convex valuation ring $O\subseteq \F$ gives rise to an order compatible valuation
 	\begin{align*}
 		v \colon \F & \to \Lambda \cup \{\infty\},
 	\end{align*}
 	where $\Lambda \cong \F^\times/O^\times$ with the order given by $-v(a) \leq -v(b)$ for all $0 \leq a \leq b \in \F$. On the other hand, every order compatible valuation gives rise to an order convex valuation ring
 	$$
 	O \coloneqq \{ a \in \F \colon v(a)\geq 0\}.
 	$$
 \end{theorem}
 
 Real closed fields always admit order compatible valuation rings (sometimes many), but the only order compatible valuation on an Archimedean real closed field is the trivial valuation. The ordered abelian group $\Lambda$ of a valuation of a real closed field is always divisible, and hence a $\Q$-vector space.

 \subsection{Residue field}\label{sec:residue_field}
 
 If $O$ is a valuation ring with maximal ideal 
 $$
 J\coloneqq \{a \in O \colon v(a) > 0\} = O \setminus O^{\times}
 $$
 then $k \coloneqq  O/J$ is called the \emph{residue field}. When $\mathbb{F}$ is real closed, so is $k$ \cite[Theorem 4.3.7]{EnPr05}. To an algebraic group $G$ and a real closed field $\F$ we will associate an affine $\Lambda$-building $\mathcal{B}$. In Section \ref{sec:assoc-buildings} we will show that the building at infinity is the spherical building associated to $G(\F)$ and the residue building is the spherical building associated to $G(k)$.

 \subsection{Examples}\label{sec:real_closed_fields_examples}
 The real numbers $\R$ and the real algebraic numbers $\overline{\Q}^{\operatorname{rc}}$ are Archimedean real closed fields.
 The field of \emph{Puiseux series} over some real closed field $\F$
 $$
 \overline{\F(X)}^{\operatorname{rc}}  \coloneqq \left\{  \sum_{k=-\infty}^{k_0} c_k X^{\frac{k}{m}} \, \colon \, k_0, m \in \mathbb{Z}, \, m > 0 , \, c_k \in \F, \, c_{k_0}\neq 0\right\},
 $$ 
 is a non-Archimedean real closed field, where the usual order on $\F$ is extended by $X>r$ for all $r \in \F$ \cite{BCR}. An order compatible valuation $v \colon \overline{\F(X)}^{\operatorname{rc}} \to \Lambda\cup \{\infty\} =  \Q\cup \{\infty\}$ is given by the degree
 \begin{align*}
 	v\left(  \sum_{k=-\infty}^{k_0} c_k X^{k/m}\right)  =-\frac{k_0}{m}.
 \end{align*} 
 
Analogous to the algebraic closure of fields, every ordered field $\K$ admits a \emph{real closure} $\overline{\K}^{\operatorname{rc}}$, the intersection of all real closed fields containing $\K$. The real algebraic numbers $\overline{\Q}^{\operatorname{rc}}$ are the real closure of the rationals and the Puiseux-series are the real closure of the ordered field of rational functions $\F(X)$. 

Generalizing slightly, we may consider the field of rational functions $\F(X,Y)$ with two variables, with the order where $X> Y^n> r$ for every $n\in \N$ and $r\in \F$. The multi-degree gives a valuation $ v \colon \F(X,Y) \to \Z \times \Z \cup \{\infty\}$ with lexicographical ordering, and the real closure $\overline{\F(X,Y)}^{\operatorname{rc}}$ is a real closed field with valuation group $\Lambda \cong \Q \times \Q$ with lexicographical ordering. 
 
 A \emph{non-principal ultrafilter on $\Z$} is a function $\omega \colon \mathcal{P}(\Z) \to \{0,1\}$ that satisfies
 \begin{enumerate}
 	\item [(1)] $\omega(\emptyset) = 0$, $\omega(\Z) = 1$,
 	\item [(2)] if $A,B \subseteq \Z$ satisfy $A\cap B = \emptyset$, then $\omega(A \cup B) = \omega(A) + \omega(B)$,
 	\item [(3)] all finite subsets $A \subseteq \Z$ satisfy $\omega(A) = 0$.
 \end{enumerate}
 Ultrafilters can be thought of as finitely-additive probability measures that only take values in $0$ and $1$. The existence of non-principal ultrafilters is equivalent to the axiom of choice \cite{Hal}. For a given ultrafilter $\omega$, we define the \emph{hyperreal numbers} $\R_{\omega}$ to be the equivalence classes of infinite sequences $\R_{\omega} = \R^{\N}/ \!\! \sim$, where $x = (x_i)_{i \in \N} \sim y = (y_i)_{i \in \N}$ if $\omega(\{ i \in \N \colon x_i \neq y_i\}) = 0$ or $\omega(\{ i \in \N \colon x_i = y_i\}) = 1$. We define addition and multiplication componentwise, the multiplicative inverse is obtained by taking the inverses of all non-zero entries, turning $\F_\omega$ into a field. Considering constant sequences, the real numbers are a subfield of $\R_\omega$. The hyperreals are an ordered field with respect to the order defined by $[(x_i)_{i\in \N}] \leq [(y_i)_{i \in \N} ]$ if and only if $\omega(\{i \in \N \colon x_i \leq y_i\}) = 1$. The hyperreals are real closed, since $\R$ is. One can check that the hyperreals do not admit a non-trivial valuation to a subgroup of $\R$: any valuation group for the hyperreals has infinite rank. The hyperreals are non-Archimedean, since the equivalence class containing $(1, 2, 3, \ldots)$ is an \emph{infinite} element, meaning it is larger than any natural number. 
 
 Let $b \in \R_{\omega}$ be an infinite element. Then
 $$
 O_b \coloneqq \{ x \in \R_{\omega} \colon |x| < b^m \text{ for some } m \in \Z \}
 $$
 is an order convex subring of $\R_\omega$ with maximal ideal
 $$
 J_b \coloneqq \{ x \in \R_\omega \colon |x| <  b^m \text{ for all } m \in \Z \}.
 $$
 The \emph{Robinson field} associated to the non-principal ultrafilter $\omega$ and the infinite element $b$ is the quotient $\R_{\omega,b} \coloneqq O_b/J_b$ \cite{Rob96}. The Robinson field is a non-Archimedean real closed field. Note that $[b] \in \R_{\omega,b} $ is a \emph{big} element, meaning that for all $a \in \R_{\omega,b}$ there is an $n \in \N$ such that $a < b^n$. 
 
 Non-Archimedean ordered fields $\F$ with big elements admit an order compatible valuation $v \colon \F \to \R \cup \{\infty\}$ by letting $v(a)$ be the real number defined by the Dedekind cut
 \begin{align*}
 	A_a &\coloneqq \left\{ \frac{p}{q} \in \Q \colon b^p  \leq a^q , q \in \Z_{>0}, p \in \Z\right\} \\
 	B_a &\coloneqq \left\{ \frac{p}{q} \in \Q \colon b^p  \geq a^q , q \in \Z_{>0}, p \in \Z\right\} 
 \end{align*} 
 for $a \in \F$ and $b\in \F$ a big element \cite{Bru2}. Note that when $\F$ is Archimedean, every element $b>1$ is big and we can still define as above $v(a) = -\log_b|a|$, which then is the usual logarithm with base $b$. However $v$ is not a valuation in our sense, since it does not satisfy the strong triangle inequality, condition (3) in the definition.

%% file: definition_B.tex
\section{Definition of the building \texorpdfstring{$\B$}{B}}\label{sec:building_def} 
For background on algebraic groups over real closed fields, see \cite{AppAGRCF}. Let $\K$ and $\F$ be real closed fields such that $\K \subseteq \R \cap \F$. Often we assume $\F$ to be non-Archimedean with order compatible valuation $v \colon \F\to \Lambda \cup \{\infty\}$. Let $G<\operatorname{SL}_n$ be a semisimple connected self-adjoint algebraic $\K$-group and $S$ a maximal $\K$-split torus that satisfies $s=s\tran$ for all $s\in S$. Let $A_{\F}$ be the semialgebraic extension of the semialgebraic connected component of $S_{\K}$ containing the identity and let $K\coloneqq G\cap \operatorname{SO}_n$. 

For $\fraka = \operatorname{Lie}(A_\R)$, let $\Sigma \subseteq \fraka^\star$ be the root system whose elements $\alpha \in \Sigma$ correspond to $\K$-roots $\chi_\alpha \in \KPhi \subseteq \hat{S}$, see Section \ref{I-sec:compatibility} in \cite{AppAGRCF}. Then $W_s=\KW$ is its spherical Weyl group. After choosing a basis $\Delta \subseteq \Sigma$ we let $U$ be the unipotent group associated to the positive root spaces and
$$
A_\F^+ \coloneqq \left\{ a\in A_\F \colon \chi_\alpha(a) \geq 1 \text{ for all } \alpha \in \Delta \right\}.
$$ 

\subsection{Non-standard symmetric spaces}

In the theory of symmetric spaces, 
$$
P_\R = \left\{ x \in \R^{n\times n} \colon x=x\tran,\, \det(x)=1,\, x \text{ is positive definite}  \right\} 
$$
is a model for the symmetric space of non-compact type associated to $\operatorname{SL}(n,\R)$. The group $\operatorname{SL}(n,\R)$ acts transitively on $P_\R$ by
$$
g.x = g x g^{T}.
$$
for $g\in \operatorname{SL}(n,\R)$, $x \in P_\R$. The orbit $\X_\R = G_\R.\operatorname{Id} \subseteq P_\R$ is a closed subset and a model for the symmetric space associated to $G_\R$. We note that $P$ is a semialgebraic set defined over $\Q$ and consider its $\F$-extension $P_\F$. The action is algebraic, so the orbit can be semialgebraically extended to
$$
\X_\F = G_\F . P_\F.
$$
When $\F$ is non-Archimedean, we call $\X_\F$ the \emph{non-standard symmetric space associated to $G_\F$}. 
\begin{proposition}\label{prop:trans}
	\begin{itemize}
		\item [(a)] The group $G_\F$ acts transitively on $\X_\F$.
		\item [(b)] The stabilizer of $\operatorname{Id}\in \X_\F$ is $K_\F$.
		\item  [(c)] For any pair $x,y \in \X_\F$ there is a $g\in G_\F$ such that $g.x=\operatorname{Id}$ and $g.y$ lies in the non-standard closed Weyl chamber
		$$
		\ApF. \operatorname{Id} = \left\{ a.\operatorname{Id} \in \X_\F \colon \chi_\alpha(a) \geq 1 \text{ for all } \alpha \in \Sigma \right\}.
		$$
	\end{itemize}
\end{proposition}
\begin{proof}
	Transitivity and stabilizer of $\operatorname{Id}$ follow directly from the definitions. Use transitivity obtain $h \in G_\F$ with $h.x  =\operatorname{Id}$. Use transitivity again to obtain $h'\in G_\F$ with $h'.h.y =\operatorname{Id}$. Now decompose $h'=kak' \in G_\F = K_\F A_\F K_\F$ using the Cartan decomposition \cite[Theorem \ref{I-thm:KAK}]{AppAGRCF}, where we may assume $a^{-1}\in \ApF$ after applying an element of the spherical Weyl group $W_s$. Setting $g=k'h$ now results in the claimed
	\begin{align*}
		g.x &= k'h.x = k'.\operatorname{Id} = \operatorname{Id}, \\
		g.y &= k'h.y = k'(h')^{-1}.\operatorname{Id} = a^{-1}k^{-1}.\operatorname{Id} = a^{-1}.\operatorname{Id} \in \ApF . \operatorname{Id}. \qedhere
	\end{align*}
\end{proof}

\subsection{The pseudo-distance}\label{sec:pseudodistance}

The symmetric space $\X_\R$ admits an explicit distance formula. Here we mimic this process to define a pseudo-distance on $\X_\F$. Let $x,y \in \X_\F$ be two points. We will first send $x$ and $y$ to a common flat, on which we define a multiplicative norm $N_\F$. In the real case, the logarithm would then be applied to obtain an additive distance. For the non-standard symmetric space we instead use the valuation $v \colon \F_{>0} \to \Lambda$.

By Proposition \ref{prop:trans}(a), for $x,y \in \X_\F$ there are $g,h\in G_\F$ with $x=g.\Id$, $y=h.\Id$. We use the Cartan decomposition \cite[Theorem \ref{I-thm:KAK}]{AppAGRCF} to write $g^{-1}h = kak' \in G_\F = K_\F A_\F K_\F$, where $a$ can be chosen to lie in $A_\F^+$ and is then unique. This gives a Weyl-chamber valued distance on $\X_\F$, which we call the \emph{Cartan projection}. 


\begin{lemma}\label{lem:Cartan_projection}
	The Cartan-projection 
	\begin{align*}
		\delta_\F \colon \X_\F \times \X_\F & \to \ApF\\
		(x,y) &\mapsto a
	\end{align*}
	is well-defined and invariant under the action of $G_\F$. For all $x,y \in \X_\F$, $\delta_\F(y,x)$ is in the Weyl-group orbit of $\delta(x,y)^{-1}$.
\end{lemma} 
\begin{proof}
	Assume $x=g.\Id = \overline{g}.\Id$ and $y=h.\Id = \overline{h}.\Id$ for some $g,\overline{g}, h, \overline{h} \in G_\F$. By Proposition \ref{prop:trans}(b), $g^{-1}\overline{g} \in K_\F$ and $h^{-1}\overline{h} \in K_\F$. Let $a,\overline{a}\in A_\F^+$ be the unique elements satisfying $g^{-1}h \in K_\F a K_\F$ and $\overline{g}^{-1}\overline{h} = \overline{k}\overline{a}\overline{k}' \in K_\F \overline{a} K_\F$ from the Cartan decomposition \cite[Theorem \ref{I-thm:KAK}]{AppAGRCF}. Then
	$$
	\overline{k}\overline{a}\overline{k}= \overline{g}^{-1}\overline{h} \in K_\F g^{-1}  h K_\F = K_\F a K_\F,
	$$
	so $\overline{a} = a$ and $\delta_\F(x,y) = a$ is well-defined. The $G_\F$-invariance holds because for $t \in G_\F$, the distance between $t.g.\Id$ and $t.h.\Id$ is calculated from $(t g)^{-1}(t h) = g^{-1}h$ and thus agrees with the distance between $x$ and $y$. If $\delta_\F(x,y) = a$, then $\delta_\F(y,x)$ is defined by $h^{-1}g \in K_\F \delta_\F(y,x) K_\F$. Thus 
	$$
	K_\F \delta_\F(y,x) K_\F \ni h^{-1}g = (g^{-1}h)^{-1} \in K_\F \delta_\F(x,y)^{-1} K_\F \subseteq K_\F A_\F K_\F,
	$$
	and since the Cartan decomposition is unique up to the action of the spherical Weyl group, $\delta_\F(y,x)$ is in the orbit of $\delta(x,y)^{-1}$.
\end{proof}

We use a basis $\Delta$ of the root system $\Sigma$ to define a notion of positive roots $\Sigma_{>0}$. For $\alpha \in \Sigma$, let $\chi_\alpha \colon A_\R \to \R^\times$ be the corresponding ($\R$-points of the) algebraic characters. We define a continuous, semialgebraic $W_s$-invariant map
\begin{align*}
	N_{\R} \colon A_{\R} &\to \R^{\times} \\
	a & \mapsto  \prod_{\alpha \in \Sigma} \max \left\{\chi_\alpha(a), \chi_\alpha (a)^{-1}\right\}
\end{align*}
which is a multiplicative norm, meaning that for all $a,b \in A_{\R}$
\begin{itemize}
	\item [(1)] $N_{\R}(a)\geq 1$ and $N_{\R}(a) = 1$ if and only if $a=\operatorname{Id}$,
	\item [(2)] $N_{\R}(ab) \leq N_{\R}(a) N_{\R}(b)$.
\end{itemize}
We call $N_\R$ the \emph{semialgebraic norm}\footnote{There are many continuous semialgebraic multiplicative norms satisfying (1) and (2), but as norms on finite dimensional vector spaces, they are equivalent and it suffices for our purposes to fix $N_\R$.}.
Since $N_{\R}$ is semialgebraic, we can extend it to a map $N_{\F} \colon A_{\F} \to \F^{\times}$ which is still a $W_s$-invariant multiplicative norm satisfying (1) and (2) by the transfer principle and $N_\F$ is given by the same formula involving the characters.
For $\F$ non-Archimedean, we now use the Cartan projection $\delta_\F$ together with the semialgebraic norm $N_\F$ and the valuation $v \colon \F_{>0} \to \Lambda $ to define
\begin{align*}
	d \colon \X_\F \times \X_\F &\to \Lambda \\
	(x,y) &\mapsto (-v)(N_\F (\delta_{\F}(x,y))).
\end{align*}

We will show in Theorem \ref{thm:pseudodistance} that $d$ is a pseudo-distance on $\X_\F$. The pseudo-distance $d$ fails to be positive definite due to the fact that $v$ is not injective. The proof of the triangle inequality uses Kostant's convexity theorem and the Iwasawa retraction
\begin{align*}
	\rho \colon \X_\F &\to A_\F.\operatorname{Id} \\
	g.\operatorname{Id} = uak.\operatorname{Id} & \mapsto a.\operatorname{Id}.
\end{align*}
using the Iwasawa-decomposition $G_\F = U_\F A_\F K_\F$ \cite[Theorem \ref{I-thm:KAU}]{AppAGRCF}.

\begin{lemma}\label{lem:rhoA}
	For all $a\in A_\F$, $x\in \X_\F$, $\rho(a.x) = a.\rho(x)$.
\end{lemma}
\begin{proof}
	Let $g=ua'k\in G_\F=U_\F A_\F K_\F$ such that $g.\operatorname{Id} = x$. By \cite[Proposition \ref{I-prop:anainN}]{AppAGRCF}, $aua^{-1} \in U_\F$, so
	\begin{align*}
		\rho(a.x) &= \rho(aua'k.\operatorname{Id}) = \rho((aua^{-1}aa'.\operatorname{Id}) = aa'.\operatorname{Id} = a .\rho(ua'k.\operatorname{Id}) = a .\rho(x). 
	\end{align*}
\end{proof}

We use Kostant's convexity theorem to prove that $\rho$ is a $d$-diminishing retraction. 
\begin{theorem}\label{thm:UAKret}
	The map $\rho \colon \X_\F \to A_\F.\operatorname{Id}$ is a $d$-diminishing
	$$
	\forall x,y \in \X_\F \colon d(\rho(x),\rho(y)) \leq d(x,y),
	$$ 
	retraction to $A_\F.\operatorname{Id}$.
\end{theorem} 
\begin{proof}
	It is clear that $\rho$ is a retraction, meaning that $\forall a \in A_\F$, $\rho(a.\operatorname{Id})=a.\operatorname{Id}$. Let $a_\F(g) = a_\F(uak)=a$ denote the $A_\F$-component of $g\in G_\F$ as in Section \ref{I-sec:kostant} of \cite{AppAGRCF} on Kostant convexity. To show that $\rho$ is $d$-diminishing, we first claim that for all
	$
	b\in \ApF
	$
	and for all $k \in K_\F$ 
	$$
	d(\operatorname{Id}, \rho(kb.\operatorname{Id})) \leq d(\operatorname{Id}, b.\operatorname{Id}).
	$$
	For $b\in \ApF$ the set
	$$
	S_\F^b= \left\{ a\in A_\F \colon a = a_{\F}(kb) \text{ for some } k \in K_\F \right\}
	$$
	is semialgebraic and since over $\R$, $S_\R^b$ is closed under the action of the spherical Weyl group (this is a consequence of the real Kostant convexity theorem \cite[Theorem 4.1]{Kos}). The statement $W_s(S_\R^b) \subseteq S_\R^b$ can be formulated as a first-order formula, so $S_\F^b$ is also closed under the action of $W_s$. Note that $\rho(kb.\operatorname{Id}) = a_\F(kb).\operatorname{Id}$. While $a_\F(kb)$ may not lie in $\ApF$, there is a $w\in W$ such that $w(a_\F(kb)) \in \ApF$. We apply \cite[Theorem \ref{I-thm:kostant_F}]{AppAGRCF} to $w(a_\F(kb)) \in \ApF \cap S_\F^b$ to get 
	$$
	\chi_{\gamma_i} (w(a_\F(kb))) \leq \chi_{\gamma_i}(b)
	$$
	for all $\gamma_i$ described in \cite[Section \ref{I-sec:kostant}]{AppAGRCF}. By \cite[Lemma \ref{I-lem:kostant_gammaalpha}]{AppAGRCF} for every $\alpha \in \Sigma_{>0}$ there are non-negative rational numbers $n_{\alpha i} \in \Q_{\geq 0}$ such that
	$$
	\alpha = \sum_{i=1}^r n_{\alpha i}\gamma_i.
	$$
	We can now prove
	\begin{align*}
		(N_\F  \circ \delta_\F)(\operatorname{Id}, \rho(kb.\operatorname{Id})) 
		&=  \prod_{\alpha \in \Sigma} \max \left\{ \chi_\alpha (a_\F(kb)), \chi_\alpha (a_\F(kb))^{-1}\right\} \\
		&=  \prod_{\alpha \in \Sigma_{>0}} \max \left\{ \chi_\alpha (a_\F(kb))^{\pm 2} \right\}  
		=  \prod_{\alpha \in \Sigma_{>0}}  \chi_\alpha (w(a_\F(kb)))^2  \\
		&=  \prod_{\alpha \in \Sigma_{>0}} \prod_{i=1}^r  \chi_{ \gamma_i } (w(a_\F(kb)))^{2n_{\alpha i}} 
		\leq \prod_{\alpha \in \Sigma_{>0}} \prod_{i=1}^r \chi_{\gamma_i} (b)^{2n_{\alpha i}} \\
		&= \prod_{\alpha \in \Sigma}  \chi_{\alpha}(b) = (N_\F \circ \delta_\F)(\operatorname{Id},b.\operatorname{Id}),
	\end{align*}
	where we used that $n_{\alpha i}\geq 0$, so after applying $(-v)$ we proved the claim.
	
	Now let $x,y \in \X_\F$ arbitrary. By Proposition \ref{prop:trans}(c) and \cite[Theorem \ref{I-thm:KAU}]{AppAGRCF}, we can find $g=uak\in G_\F = U_\F A_\F K_\F$ such that $x=g.\operatorname{Id}$ and $y=g.b.\operatorname{Id}$ for some $b \in A_\F$. Now we use Lemma \ref{lem:rhoA} and the above to conclude 
	\begin{align*}
		d(\rho(x),\rho(y)) &= d(\rho(uak.\operatorname{Id}), \rho(uakb.\operatorname{Id})) 
		= d(a.\operatorname{Id}, \rho(akb.\operatorname{Id})) \\
		&= d(\operatorname{Id}, a^{-1}.\rho(akb.\operatorname{Id})) 
		= d(\operatorname{Id}, \rho(kb.\operatorname{Id})) \\
		&\leq d(\operatorname{Id}, b.\operatorname{Id}) = d(g.\operatorname{Id}, g.b.\operatorname{Id}) = d(x,y). \qedhere
	\end{align*}
\end{proof}

\begin{theorem}\label{thm:pseudodistance}
	The function $d \colon \X_\F \times \X_\F \to \Lambda$ is a pseudo-distance.
\end{theorem}
\begin{proof}
	We note that by definition, $N_\F(a) = N_\F(a^{-1})$ for all $a\in A_\F$. By Weyl group invariance and the last part of Lemma \ref{lem:Cartan_projection} we then obtain $N_\F(\delta_\F(x,y)) = N_\F(\delta_\F(y,x))$ for all $x,y\in \X_\F$, whence $d$ is symmetric.
	Since $N_\F(A_\F) \subseteq \F_{\geq 1}$, $(-v)(1)=0$ and $-v$ is order-preserving, $d$ is positive. It is also clear that $d(x,x) = 0$ for all $x\in \X_\F$. It remains to show that the triangle inequality holds. We start by analyzing the distance on the non-standard maximal flats $A_\F.\operatorname{Id}$. Let $a,b,c \in A_\F$, then we can use property (2) of the semialgebraic norm $N_\F \colon \A_\F \to \F_{>0}$ to deduce
	\begin{align*}
		d(a.\operatorname{Id},b.\operatorname{Id}) & = d(\operatorname{Id}, a^{-1}b.\operatorname{Id}) = (-v)(N_\F(a^{-1}b)) \\
		&= (-v)(N_\F(a^{-1}c c^{-1}b)) 
		\leq -(v)(N_\F(a^{-1}c) N_\F(c^{-1}b)) \\
		&= -v(N_\F(a^{-1}c)) + (-v)(N_\F(c^{-1}b)) \\
		&= d(a.\operatorname{Id},c.\operatorname{Id}) + d(c.\operatorname{Id},b.\operatorname{Id}),
	\end{align*}
	which settles the triangle inequality for points in $ A_\F.\operatorname{Id}$. For the general case we use the Iwasawa retraction from Theorem \ref{thm:UAKret}, as suggested in \cite[Lemma 1.2]{KrTe}. Let $x,y,z \in \X_\F$. By Proposition \ref{prop:trans}(c) there is a $g\in G_\F$ with $g.x ,g.y \in A_\F.\operatorname{Id}$. Then
	\begin{align*}
		d(x,y) & = d(g.x,g.y) = d(\rho(g.x),\rho(g.y)) \\
		& \leq d(\rho(g.x),\rho(g.z)) + d (\rho(g.z),\rho(g.y)) \\
		&\leq d(g.x,g.z) + d(g.z,g.y) = d(x,z) + d(z,y) 
	\end{align*}
	concludes the proof.
\end{proof}

\subsection{The building} \label{sec:building}

By Theorem \ref{thm:pseudodistance}, the non-standard symmetric space $\X_\F$ admits a $\Lambda$-pseudometric. We consider the quotient
$$
\B = \X_\F/\!\sim
$$
where $x\sim y \in \B$ when $d(x,y) = 0 \in \Lambda$. We denote the induced $\Lambda$-metric on $\B$ by the same letter $d$. We note that $G_\F$ acts by isometries on $\B$. In Section \ref{sec:Bisbuilding} we will show that the $\Lambda$-metric space $\B$ admits the structure of an affine $\Lambda$-building in certain cases, see Theorem \ref{thm:B_is_building}.
Before we start checking the axioms of affine $\Lambda$-buildings, we investigate the apartment structure in Section \ref{sec:Xapartment} and stabilizers in Section \ref{sec:stabilizer}.

\subsection{The model apartment}\label{sec:Xapartment}

Over the reals, the orbit $A_\R.\operatorname{Id} \subseteq \X_\R$ is a maximal flat in the symmetric space $\X$. We take a closer look at the group $A_\F$ and its orbit $A_\F.\operatorname{Id} \subseteq \X_\F$ to define a space $A_\Lambda$ which will play the role of the model apartment $\A$.
Let $O$ be an order convex valuation ring of the non-Archimedean real closed field $\F$ and $(-v) \colon \F_{>0} \to \Lambda$ the associated order preserving valuation. 
We define the group
$$
A_{\Lambda} \coloneqq A_{\F}/\{ a \in A_{\F} \colon N_{\F}(a) \in O \}.
$$
The goal of this section is to prove Theorem \ref{thm:apt} which states that $A_{\Lambda}$ can be given the structure of a model apartment $\A = \A(\Sigma^\vee,\Lambda,A_{\Lambda})$, as defined in Section \ref{sec:modelapartment}. We also describe walls, half-apartments, sectors and the distance function in terms of $A_\F$. 

From \cite[Section \ref{I-sec:compatibility}]{AppAGRCF} we recall that the root system $\Sigma$ of the real Lie group $G_\R$ can be identified with the root system $\KPhi$ of characters, where associated to every root $\alpha \in \Sigma$ there is a character $\chi_\alpha \colon A_\F \to \F_{>0}$. Similarly, the coroot system $\Sigma^\vee$ can be identified with the coroots $X_{\star}(S)$, where associated to every root $\alpha$ we obtain a one-parameter subgroup $t_\alpha \colon \F_{>0} \to A_\F$. In fact there are algebraic characters $\chi_\eta$ and one-parameter subgroups $t_\eta$ for every $\eta \in \operatorname{Span}_{\mathbb{Z}}(\Sigma)$ and semialgebraic characters and one-parameter subgroups for $\eta \in \operatorname{Span}_{\mathbb{Q}}(\Sigma)$. 

Formally $\Sigma \subseteq \fraka^\star$ where $\fraka$ is a maximal abelian subalgebra of $\frakp$ in a Cartan decomposition $\frakg = \frakk \oplus \frakp$. Thus $\Sigma^\vee \subseteq \fraka$, using the canonical identification $(\fraka^\star)^\star \cong \fraka$. Recall that every crystallographic root system $\Phi$ comes with a non-degenerate bilinear form
\begin{align*}
	b \colon \operatorname{Span}_{\Z}(\Phi) \times \operatorname{Span}_\Z(\Phi^\vee) & \to  \Z \\
	(\alpha, \beta^\vee) &\mapsto 
	2\frac{\langle \alpha ,\beta\rangle}{\langle \beta, \beta\rangle} = \beta^\vee (\alpha).
\end{align*}
The bilinear forms $b$ of $\Sigma$ and $b^\vee$ of $\Sigma^\vee$ are not symmetric, but they are the transposes of each other in the following sense.
\begin{lemma}\label{lem:coroot_transpose}
	For all $\alpha, \beta \in \Sigma $,
	$$
	\alpha(\beta^\vee) = \beta^\vee (\alpha) = b(\alpha, \beta^\vee) = b^\vee (\beta^\vee, \alpha),
	$$
	where we identified $(\alpha^\vee)^\vee = \alpha \in \Sigma$.
\end{lemma}
\begin{proof}
	This is a direct computation using $\langle \beta^\star, \alpha^\star \rangle = \langle \beta, \alpha \rangle$,
	\begin{align*}
		b^\vee \left( \beta^\vee, (\alpha^\vee)^\vee \right) 
		&= 2 \frac{ \langle \beta^\vee, \alpha^\vee \rangle }{\langle \alpha^\vee, \alpha^\vee \rangle} = 2 \frac{\left\langle 2 \frac{\beta^\star}{\langle \beta, \beta \rangle}, 2 \frac{\alpha^\star}{\langle \alpha, \alpha \rangle} \right\rangle }{ \left\langle 2 \frac{\alpha^\star}{\langle \alpha, \alpha \rangle}, 2 \frac{\alpha^\star}{\langle \alpha, \alpha \rangle} \right\rangle } 
		= 2 \frac{\langle \beta, \alpha \rangle}{\langle \beta, \beta \rangle} = b(\alpha, \beta^\vee). \qedhere
	\end{align*}
\end{proof}

\begin{proposition}
	For all $\alpha, \beta \in \operatorname{Span}_{\Z}(\Sigma)$ and $x \in \F_{>0}$, we have 
	$$
	\chi_\alpha(t_\beta(x)) = x^{b(\alpha, \beta^\vee)}.
	$$
\end{proposition}
\begin{proof}
	Over the reals, we have by \cite[Lemma \ref{I-lem:oneparam_compatibility}]{AppAGRCF} that for all $x = e^s$,
	\begin{align*}
		\chi_\alpha\left(t_\beta\left(e^s\right)\right) &= \chi_\alpha \left( \exp \left( s x_\beta \right)\right) 
		= e^{\alpha(s x_\beta)} = \left(e^s\right)^{\alpha(x_\beta)} 
		= x^{b(\alpha, \beta^\vee)}
	\end{align*}
    where $x_\beta := \beta^\vee := 2\beta^\star /\langle \beta, \beta \rangle  \in \Sigma^\vee$. The statement is a first-order formula, hence by the transfer principle the statement also holds for $\F$.
\end{proof}

The characters and one-parameter subgroups descend to $A_\Lambda$. 

\begin{proposition}\label{prop:char_cochar_Lambda}
	For all $\alpha, \beta \in \operatorname{Span}_{
		\mathbb{Z}}(\Sigma)$, the characters $\chi_\alpha$ and the one-parameter subgroups $t_\beta$ descend to group homomorphisms $\overline{\chi}_\alpha \colon A_\Lambda \to \Lambda$ and $\overline{t}_\beta \colon \Lambda \to A_\Lambda$ such that the diagram
	$$
	\begin{tikzcd}
		\mathbb{F}_{>0} \arrow[d, "-v"] \arrow[r, "t_{\beta}"] & A_{\mathbb{F}} \arrow[d,two heads] \arrow[r, "\chi_\alpha"] & \mathbb{F}_{>0} \arrow[d, "-v"]  \\
		\Lambda \arrow[r, "\overline{t}_\beta"]                             & A_\Lambda \arrow[r, "\overline{\chi}_\alpha"]                     & \Lambda                                 
	\end{tikzcd}
	$$
	commutes and such that 
	$
	\chi_\alpha \circ t_\beta (\lambda ) =b( \alpha, \beta^\vee ) \cdot \lambda
	$ for every $\lambda \in \Lambda$.
\end{proposition}
\begin{proof}
	We denote by $\pi \colon A_\F \to A_\Lambda = A_\F / \{ a\in A_\F \colon N_\F(a) \in O \}$ the projection. We first show that $\overline{t}_\beta \colon \Lambda \to A_\Lambda, \, (-v)(x) \mapsto \pi(t_\beta (x))$ is well defined: we have to show that if $(-v)(x) = 0$, then $N(t_\beta(x)) \in O$. Indeed, let $x\in \F_{>0}$ such that $(-v)(x)=0$, so $x \in O^\times$. Then in particular $ \chi_\gamma (t_\beta (x)) = x^{b( \gamma, \beta^\vee )} \in O^\times$ for all $\gamma \in \Sigma$ and thus
	\begin{align*}
		N_\F(t_\beta(x)) & = \prod_{\gamma \in \Sigma} \max \left\{ \chi_\gamma( t_\beta(x)) , \chi_\gamma(t_\beta(x))^{-1} \right\} \in O.
	\end{align*}
	
	To show that $\overline{\chi}_\alpha \colon A_{\Lambda} \to \Lambda, \, \pi(a) \mapsto -v(\chi_\alpha(a))$ is well-defined, we have to show that if $N_\F(a) \in O$, then $(-v)(\chi_\alpha(a)) =0$. Indeed, if $a\in A_\F$ with $N_\F(a) \in O$, then $\chi_\gamma(a) \in O$ for all $\gamma \in \Sigma$ since
	\begin{align*}
		N_\F(a) & = \prod_{\gamma \in \Sigma} \max \left\{ \chi_\gamma( a) , \chi_\gamma(a)^{-1} \right\} \in O
	\end{align*}
	is a product of elements that are larger than $1$. Since $\alpha \in \operatorname{Span}_\Z(\Sigma)$, $\chi_\alpha(a) $ is a product of $\chi_\gamma(a) \in O$ and hence in $O$.
	
	The maps $\overline{t}_\beta$ and $\overline{\chi}_\alpha$ were defined so that the diagrams commute. For $\lambda \in \Lambda$, we can find $x \in \F_{>0}$ such that $(-v)(x)=\lambda $. Then
	\begin{align*}
		\overline{\chi}_\alpha \circ \overline{t}_\beta (\lambda) &= \overline{\chi}_\alpha (\pi(t_\beta(x ) ) )  
		= (-v)\left( \chi_\alpha \circ t_\beta (x) \right) \\
		&= -v\left(x^{b( \alpha, \beta^\vee )} \right)
		= b( \alpha , \beta^\vee) \cdot (-v)(x)  . \qedhere
	\end{align*}
\end{proof}

Our goal is to use the characters and one-parameter subgroups to identify $A_\F$ with the model apartment $\operatorname{Span}_{\Q}(\Sigma^\vee)$. As $\F$ is real closed, we can view $\F_{>0}$ as a $\Q$-vector space with neutral element $1$ and scalar multiplication given by potentiation.
\begin{proposition}\label{prop:apt0}
	Let $L \coloneqq \operatorname{Span}_{\mathbb{Q}}(\Sigma)$. The map $\psi \colon A_\F \to \operatorname{Hom}_{\mathbb{Q}}(L,\F_{>0})$ that sends $a\in A_\F$ to $\gamma \mapsto \chi_\gamma(a)$ is a group isomorphism. For all $\delta, \delta' \in \Delta$, there are rational numbers $q_{\delta,\delta'} \in \Q$ such that 
	$$
	a = \prod_{\delta, \delta' \in \Delta} t_\delta (\chi_{\delta'}(a))^{q_{\delta,\delta'}}
	$$
	for all $a \in A_\F$.
\end{proposition}
\begin{proof}
	It is clear that $\psi$ is a well defined group homomorphism. From dimension reasons over the reals $\R$, it follows that if $\chi_\alpha(a)=1$ for all $\alpha \in \Sigma$, then $a=\operatorname{Id}$. This is a first-order property and hence also true over $\F$ by the transfer principle, so $\psi$ is injective. For surjectivity, we construct an explicit preimage of $f\in \operatorname{Hom}_{\Q}(L,\F_{>0})$ using the Cartan matrix $B_{ij} = b(\delta_i, \delta_j)$ for a basis $\Delta = \{\delta_1, \ldots, \delta_r\}$ of $\Sigma$, see Section \ref{sec:root_system}. Since $B\in \Z^{r \times r}$ is non-degenerate, there exists an inverse $B^{-1} \in \Q^{r \times r}$ with entries $B_{jk}^{-1}$. Note that rational powers of elements in $A_\F$ exist and are uniquely defined, as they are over $\R$ and this is a first-order property. Now consider 
	$$
	a \coloneqq \prod_{j,k = 1}^r t_{\delta_j}(f(\delta_k ))^{B_{jk}^{-1}} \in A_\F.
	$$
    For arbitrary $\gamma = \sum_{i=1}^r y_i \delta_i \in L$ we have
    \begin{align*}
    	\chi_{\gamma}\left(a \right) &= \prod_{i=1}^r  \chi_{\delta_i}\left( \prod_{j,k=1}^r t_{\delta_j}(f(\delta_k))^{B_{jk}^{-1}} \right)^{y_i} \\
    	&= \prod_{i,j,k=1}^r \chi_{\delta_i}\left(t_{\delta_j}(f(\delta_k)) \right)^{ B_{jk}^{-1} y_i }  
    	= \prod_{i,k=1}^r f(\delta_k)^{\sum_{j=1}^r B_{ij}B_{jk}^{-1} y_i} \\
    	&= \prod_{i=1}^r f(\delta_i)^{y_i} =   f\left( \sum_{i=1}^r y_i \delta_i \right) = f(\gamma),
    \end{align*}
    so $\psi$ is an isomorphism. Taking $f= \gamma \mapsto \chi_{\gamma}(a)$ gives the stated formula for $q_{\delta_j,\delta_k} \coloneq B_{jk}^{-1}$.
\end{proof}

We now prove the analogue of Proposition \ref{prop:apt0} in the setting of $\Lambda$.

\begin{proposition}\label{prop:apt1}
	Let $L:= \operatorname{Span}_{\Q}(\Sigma)$. The map $\overline{\psi} \colon A_\Lambda \to \operatorname{Hom}_{\Q}(L,\Lambda)$ that sends $\xi \in A_\Lambda$ to $\gamma \mapsto \overline{\chi}_\gamma(\xi)$ is a group isomorphism.
\end{proposition}
\begin{proof}
	It is clear that $\overline{\psi}$ is a well defined group homomorphism. By the isomorphism theorem applied to $f$ in the commuting diagram
	$$
	\begin{tikzcd}
		A_{\mathbb{F}} \arrow[r, "\psi"] \arrow[d, two heads] \arrow[rd, "f" description, two heads] & {\operatorname{Hom}_{\mathbb{Q}}(L,\mathbb{F}_{>0})} \arrow[d, two heads] \\
		A_{\Lambda} \arrow[r, "\overline{\psi}"]                                                     & {\operatorname{Hom}_{\mathbb{Q}}(L,\Lambda)}                            
	\end{tikzcd}
$$
	it suffices to show that $\operatorname{ker}(f) = \left\{ a \in A_\F \colon N_\F(a) \in O \right\}$. Indeed $a\in A_\F$ satisfies $N_\F(a)\in O$ if and only if $\chi_\gamma(a) = \psi(a)(\gamma)\in O$ for all $\gamma \in L$, if and only if $f(a) = 0 \in \Lambda$.
\end{proof}

\begin{proposition}\label{prop:apt2}
	Let $L:= \operatorname{Span}_{\Q}(\Sigma)$ and $L^\vee := \operatorname{Span}_{\Q}(\Sigma^\vee)$. The $\Q$-linear map $\overline{\varphi} \colon L^\vee \otimes_\Q \Lambda \to \operatorname{Hom}_{\Q}(L,\Lambda)$ that sends the basis element $ \alpha^\vee \otimes \lambda $ to $\gamma \mapsto \gamma(\alpha^\vee) \cdot \lambda$ is a $\Q$-vector space isomorphism.
\end{proposition}
\begin{proof}
	 Let as before $\Delta= \{\delta_1, \ldots, \delta_r\}$ be a basis of $\Sigma$. It is clear that 
	 $$
	 \overline{\varphi} \colon \quad  \sum_{i=1}^r \delta_i^\vee \otimes \lambda_i \mapsto \left( \gamma \mapsto \sum_{i=1}^r \gamma(\delta_i^\vee) \cdot \lambda_i \right)
	 $$
	 is well defined and $\Q$-linear. We will now use the Cartan matrix $B$ of the root system $\Sigma^\vee$ analogous to the previous proofs to show
	 \begin{align*}
	 	\overline{\varphi}^{-1} \colon \operatorname{Hom}_{\Q}(L,\Lambda) & \to L^\vee \otimes_\Q \Lambda \\
	 	f & \mapsto \sum_{j,k=1}^r  B_{jk}^{-1} \delta_j^\vee \otimes f(\delta_k)
	 \end{align*}
 is an inverse of $\overline{\varphi}$. Indeed, for $f\in \operatorname{Hom}_{\Q}(L,\Lambda)$,
 \begin{align*}
 	[\overline{\varphi} \circ \overline{\varphi}^{-1} (f)] \left( \sum_{i=1}^r{z_i \delta_i}\right) &= \left[\overline{\varphi} \left( \sum_{j,k=1}^r  B_{jk}^{-1} \delta_j^\vee \otimes f(\delta_k) \right)\right]\left( \sum_{i=1}^r{z_i \delta_i}\right) \\
 	&= \sum_{j,k = 1}^r  B_{j,k}^{-1} \sum_{i=1}^r z_i \delta_i(\delta_j^\vee)  f(\delta_k) = \sum_{j,k,i = 1}^r z_i  B_{ij}  B_{jk}^{-1} f(\delta_k) \\
 	&= \sum_{i,k = 1}^r z_i \operatorname{Id}_{ik} f(\delta_k) = \sum_{i=1}^r z_i f(\delta_i) = f\left(\sum_{i=1}^r z_i\delta_i\right) 
 \end{align*}
and for $\lambda_i \in \Lambda$,
\begin{align*}
	\overline{\varphi}^{-1}\left( \overline{\varphi} \left( \sum_{i=1}^r \delta_i^\vee \otimes \lambda_i \right)\right) 
	&= \sum_{j,k=1}^r  B_{jk}^{-1} \left(\delta_i^\vee \otimes \overline{\varphi}\left( \sum_{i=1}^r \delta_i^\vee \otimes \lambda_i \right)(\delta_k)\right)  \\
&= \sum_{j,k=1}^r  B_{jk}^{-1} \delta_j \otimes \left(\sum_{i=1}^r \delta_k(\delta_i^\vee) \lambda_i \right) \\ 
&= \sum_{i,j,k = 1}^r  B_{jk}^{-1} B_{ki} \delta_j^\vee \otimes \lambda_i \\
&= \sum_{j,i=1}^r \operatorname{Id}_{ji} \delta_{j}^\vee \otimes \lambda_i 
= \sum_{i=1}^r \delta_i^\vee \otimes \lambda_i.
\end{align*}
which shows that $\overline{\varphi}$ is an isomorphism.
\end{proof}
Propositions \ref{prop:apt1} and \ref{prop:apt2} together give an isomorphism $f =  \overline{\psi}^{-1} \circ \overline{\varphi} \colon \A \xrightarrow{\sim} A_{\Lambda} $, where $\A= L^\vee \otimes_\Q \Lambda$ is the model apartment associated to the coroot system $\Sigma^\vee$. We take the full translation group $T=\A$.
\begin{theorem}\label{thm:apt}
	There is a group isomorphism $f \colon \mathbb{A} \xrightarrow{\sim} A_\Lambda$ by which $A_\Lambda$ is given the structure of a model apartment of type $\A=\A(\Sigma^\vee, \Lambda, A_\Lambda)$.
	$$
\begin{tikzcd}
	\mathbb{A} \arrow[r, "\overline{\varphi}"] \arrow[rr, "f" description, bend right] & {\operatorname{Hom}_{\mathbb{Q}}(L,\Lambda)} & A_{\Lambda} \arrow[l, "\overline{\psi}"']
\end{tikzcd}
	$$
\end{theorem}

 The following allows us to describe walls and halfapartments in terms of the roots.
\begin{proposition}\label{prop:chi_f}
	For all $x \in \A$, $a\in A_\F$, $f(x) = [a]$ if and only if for all $\alpha \in \Sigma$ we have
	$$
	(-v)(\chi_{\alpha}(a)) = b^\vee(x, \alpha) = \alpha(x) \in \Lambda,
	$$
	where $b$ is the non-degenerate bilinear form associated to $\Sigma$.
\end{proposition}
\begin{proof}
	Let as before $B$ be the Cartan matrix of $\Sigma$. We use Lemma \ref{lem:coroot_transpose} to see that $b^\vee(\delta_j^\vee, \delta_\ell) = b(\delta_\ell, \delta_j^\vee) = B_{\ell j}$. According to Propositions \ref{prop:apt1} and \ref{prop:apt2},
	$$
	x = \sum_{j,k = 1}^r B_{jk}^{-1} \delta_j^\vee \otimes \overline{\chi}_{\delta_k}(f(x)) = \sum_{j=1}^r \delta_j^\vee \otimes \sum_{k=1}^r B_{kj}^{-1} \overline{\chi}_{\delta_k}(f(x)).
	$$ 
	Let $\eta = \sum_{\ell =1}^r z_\ell \delta_\ell$, then
	\begin{align*}
        b^\vee(x, \eta) 
        &= \sum_{j,k=1}^r B_{jk}^{-1} \overline{\chi}_{\delta_k}(f(x)) b^\vee(\delta_j^\vee, \eta) 
        =  \sum_{j,k,\ell =1}^r B_{jk}^{-1} z_\ell \overline{\chi}_{\delta_k}(f(x)) b^\vee(\delta_j^\vee,\delta_\ell) \\
        &=  \sum_{j,k,\ell =1}^r B_{\ell j} B_{jk}^{-1} z_\ell \overline{\chi}_{\delta_k}(f(x)) 
        = \sum_{\ell , k =1}^r \operatorname{Id}_{\ell,k} z_\ell \overline{\chi}_{\delta_k}(f(x)) 
        = \sum_{\ell = 1}^r z_\ell \overline{\chi}_{\delta_\ell}(f(x)) \\
        &= \overline{\chi}_{\sum_{\ell=1}^r z_\ell\delta_\ell }(f(x))
        = \overline{\chi}_\eta (f(x)).
	\end{align*}
If $f(x)=[a]$ for some $a\in \A_\F$, then $(-v)(\chi_\alpha(a)) = \overline{\chi}_\alpha(f(x)) = \alpha(x)$ for all $\alpha \in \Sigma$. On the other hand, $f(x)$ is fully determined by the values $\overline{\psi}(f(x))(\gamma) = \overline{\chi}_\gamma(f(x))$ for $\gamma \in L$, which are determined by linear combination of the given equations. 
\end{proof}
\begin{corollary}\label{cor:f_additive}
	For all $x,y \in \A$, if $f(x)= [a]$ and $f(y)=[b]$ for some $a , b\in A_\F$, then $f(x+y) = [ab]$.
\end{corollary}
\begin{proof}
	By Proposition \ref{prop:chi_f}, for all $(-v)(\chi_\alpha(a)) = \alpha(x)$ and $(-v)(\chi_\alpha(b))=\alpha(y)$, so
	$$
(-v)(\chi_\alpha(ab)) = (-v)(\chi_\alpha(a)\chi_\alpha(b)) = (-v)(\chi_\alpha(a)) + (-v)(\chi_\alpha(b)) = \alpha(x+y),
    $$
    so by Proposition \ref{prop:chi_f} $f(x+y) = [ab]$.
\end{proof}
We see from the definition of walls and halfapartments in Section \ref{sec:modelapartment} that for $\lambda \in \Lambda$ and $\alpha \in \Sigma,$
\begin{align*}
	M_{\alpha^\vee, \lambda}& = \left\{ x\in \A \colon b^\vee(x,\alpha) = \lambda \right\} 
	= \left\{ \xi \in A_\Lambda \colon \overline{\chi}_\alpha(\xi) = \lambda \right\} , \\
		H_{\alpha^\vee, \lambda}^+& = \left\{ x\in \A \colon b^\vee(x,\alpha) \geq \lambda \right\} 
	= \left\{ \xi \in A_\Lambda \colon \overline{\chi}_\alpha(\xi) \geq \lambda \right\}
\end{align*}
under the identification $\A \cong L^{\vee} \otimes_\Q \Lambda \cong A_\Lambda$. We will abbreviate $M_{\alpha,\lambda} := M_{\alpha^\vee, \lambda}$ and  $H_{\alpha,\lambda}^+ := H_{\alpha^\vee, \lambda}^+$.

We note that $\Sigma^\vee$ and hence $\A$ come with actions of the spherical Weyl group $W_s$, the translations $T\coloneqq \A$ and thus by the affine Weyl group $W_a \coloneqq  T \rtimes W_s$. The spherical Weyl group can be identified with $N_\F/M_\F$, \cite[Proposition \ref{I-prop:weylgroups}]{AppAGRCF}, which acts by conjugation on $A_\F$ and on $A_\Lambda$. We now verify that the identification $f \colon \A \to A_\Lambda$ is compatible with these actions. For roots $\alpha,\beta \in \Sigma$, the reflection $s_\alpha$ along the hyperplane defined by $\alpha$ is characterized by 
$$
\beta(s_\alpha(H)) =  \beta\left(H - \alpha(H)\alpha^\vee\right) = \beta(H)- \beta(\alpha^\vee) \cdot \alpha(H)
$$
for all $H\in \fraka$. Using \cite[Lemma \ref{I-lem:char}]{AppAGRCF} this can be translated to the multiplicative setting, where we can say $n\in N_\F \coloneqq \operatorname{Nor}_{K_\F}(A_\F)$ \emph{acts like the reflection} $s_\alpha$ if for all $a \in A_\F$,
$$
\chi_\beta (nan^{-1}) = \frac{\chi_\beta(a)}{\chi_\alpha(a)^{\beta(\alpha^\vee)}}.
$$
\begin{lemma}\label{lem:actionWs}
  For every reflection $s_\alpha$ where $\alpha \in \Sigma$ there is some $n \in N_\F$ which acts like the reflection $s_\alpha$. Moreover, for all $x\in \A$, if $[a]=f(x)$ for some $a\in A_\F$, then $f(s_\alpha(x)) = [nan^{-1}] \in A_\Lambda$.	
\end{lemma}
\begin{proof}
	There is some $n\in N_\F$ that acts like the reflection $s_\alpha$ by \cite[Proposition \ref{I-prop:weylgroups}]{AppAGRCF}.
	Let $x\in \A$ and $a\in A_\F$ such that $f(x) = [a] \in A_\Lambda$. By Lemma \ref{prop:chi_f}, for all $\gamma \in L \coloneqq \operatorname{Span}_\Q(\Sigma)$, $\gamma(x)= (-v)(\chi_\gamma(a))$. We verify that for all $\gamma\in L$
	\begin{align*}
		\overline{\psi}([nan^{-1}])(\gamma) &= \overline{\chi}_\gamma([nan^{-1}]) 
		= (-v)(\chi_\gamma(nan^{-1}))
		= (-v)\left( \frac{\chi_\gamma(a)}{\chi_\alpha(a)^{\gamma(\alpha^\vee)}} \right) \\
		&= \overline{\chi}_\gamma(f(x)) - \gamma(\alpha^\vee) \overline{\chi}_\alpha(f(x))
		= \gamma(x) - \gamma(\alpha^\vee) \alpha(x) \\
		&= \gamma( x-\alpha(x) \alpha^\vee )
		= \gamma(s_\alpha(x))
		= \overline{\varphi}(s_\alpha(x))(\gamma), 
	\end{align*}
confirming $f(s_\alpha(x)) = [nan^{-1}]$.
\end{proof}

\subsection{Valuation properties of \texorpdfstring{$A_\F$}{A(F)}} \label{sec:matrixvaluation}

In this subsection we investigate the connection of the matrix entries of elements in $a \in A_\F$ with their characters $\chi_\alpha(a)$. For any matrix $a \in \F^{n\times n}$ with matrix entries $(a_{ij})$ let
$$
(-v)(a) \coloneqq \max_{i,j} \{ (-v)(a_{ij}) \}.
$$

\begin{lemma}\label{lem:matrixvaluation}
	For all $a, a_i \in A_\F$ for $i= 1, \ldots , k$ and $k\in \operatorname{SO}_n(\F)$ we have
	\begin{itemize}
		\item [(i)] $(-v)\left( \prod_{i=1}^k a_i\right)  \leq \sum_{i=1}^k (-v)(a_i)$,
		\item [(ii)] $(-v)\left(\sum_{i=1}^k a_i\right) \leq \max_i \{ (-v)(a_i) \}$,
		\item [(iii)] $(-v)(k) = 0$,
		\item [(iv)] $(-v)(kak^{-1}) = (-v)(a)$,
		\item [(v)] $(-v)(a^{q}) = q \cdot (-v)(a)$ for all $q\in \Q$. 
	\end{itemize}
\end{lemma}
\begin{proof}
	For any matrices $a,b \in \F^{n\times n}$, we note that the valuation of the matrix entry $(ab)_{ik}$ is bounded by $(-v)(a)$ and $(-v)(b)$
	\begin{align*}
		(-v)((ab)_{ik}) &= (-v)\left(\sum_{j=1}^n a_{ij}b_{jk} \right)
		\leq \max_{j=1}^n \left\{ (-v)(a_{ij}) + (-v)(b_{jk}) \right\}
		\\ &\leq \max_{ij} \left\{ (-v)(a_{ij}) \right\} + \max_{ij}\left\{ (-v)(b_{ij})\right\}.
	\end{align*}
	This estimate extends to finite products of matrices by induction, proving (i).
	
	For (ii), we observe
	\begin{align*}
		(-v)\left(\sum_{i=1}^k a_i\right) &= \max_{j,\ell}\left\{  (-v)\left(\sum_{i=1}^k (a_i)_{j\ell} \right) \right\} 
		\leq \max_{j,\ell, i} \left\{ (-v)\left((a_i)_{j\ell}\right)  \right\} \\
		&= \max_i \{ (-v)(a_i) \}.
	\end{align*}
	
	For $k\in \operatorname{SO}_n(\F)$, we have $(-v)(k) \leq 0$, since all entries satisfy $|k_{ij}|\leq 1$. Statement (iii) holds, since there is at least one matrix entry of $k$ with valuation $0$, since otherwise the determinant of $k$ would have to have negative valuation, since the set of elements in $\F$ with negative valuation is an ideal.
	
	By (i) and (iii) we have that $(-v)(kak^{-1}) \leq  (-v)(k) + (-v)(a) + (-v)(k^{-1}) = (-v)(a)$, but replacing $a$ by $k^{-1}ak$ we also obtain the other inequality of (iv).

	Recall that $a\tran = a$ for all $a\in A_\F$, so by the spectral theorem there exists $k\in \operatorname{SO}_n$ such that $kak^{-1}$ is a diagonal matrix. Note that when $q \in \Q \setminus \Z$, then $a^q$ is defined as $k^{-1}(kak^{-1})^q k$. We then have 
	$$
	(-v)(a^q) = (-v)\left(ka^qk^{-1}\right) 
	= (-v)\left(\left(kak^{-1}\right)^q\right)
	= q \cdot (-v)\left( kak^{-1} \right)
	= q \cdot (-v)(a)
	$$
	for all $q\in \Q$, concluding the proof of (v).  
\end{proof}

The following description of $A_\F$ in terms of matrix entries comes from the semialgebraic definition of $A_\F$.

\begin{lemma}\label{lem:matrixval-bound}
	There exists a constant $C\in \Q_{\geq 0}$, such that for all $\varepsilon \in \F_{\geq 1}$ and $a\in A_\F$, if for all $\delta\in \Delta$, $\chi_\delta(a) \leq \varepsilon $, then $(-v)(a) \leq C \cdot  (-v)(\varepsilon)$. 
	
	There is also a constant $C' \in \Q_{\geq 0}$, such that for all $\lambda \in \Lambda_{\geq 0}$ and $a\in A_\F$, if for all $\delta\in \Delta$, $(-v)(\chi_\delta(a)) \leq \lambda$, then $(-v)(a) \leq C' \cdot \lambda$.
\end{lemma}
\begin{proof}
	Since both $t_\alpha \colon \F_{>0} \to A_\F$ and the projection $\pi_{ij} \colon A_\F \to \F$ to the $ij$-matrix entry are semi-algebraic defined over $\K$, the their concatenation $\F_{>0} \to \F$ is semi-algebraic and it's growth is bounded by a polynomial $p\in \K[x]$ in the sense that for all $\alpha \in \Sigma$ and all $x\in \F$ we have
	$$
	- p(x) \leq t_\alpha(x)_{ij} \leq p(x).
	$$
	
	Let $\varepsilon \in \F_{\geq 1}$ and $a\in A_\F$ such that $\chi_\delta(a) \leq \varepsilon $ for all $\delta \in \Delta$. By Proposition \ref{prop:apt0} there exist $q_{\delta,\delta'} \in \Q$ such that
	$$
	a = \prod_{\delta,\delta' \in \Delta} t_\delta(\chi_{\delta'}(a))^{q_{\delta,\delta'}}.
	$$ 
	We apply the above estimate and Lemma \ref{lem:matrixvaluation} to obtain
	\begin{align*}
		(-v)\left( a \right) & \leq \sum_{\delta,\delta'\in \Delta} (-v)\left(  t_\delta (\chi_{\delta'}(a))^{q_{\delta,\delta'}}  \right)
		= \sum_{\delta,\delta'\in \Delta} q_{\delta,\delta'} (-v)\left(  t_\delta (\chi_{\delta'}(a))\right) \\
		&\leq  \sum_{\delta,\delta'\in \Delta} q_{\delta,\delta'} (-v)\left(  p (\chi_{\delta'}(a))\right)
	\end{align*}
and if $p(x) = \sum_{n=0}^{\deg(p)} b_n x^n \in \K[x]$, then
\begin{align*}
	(-v)(p(\chi_{\delta'}(a))) & \leq \max_{n} \left\{ (-v)\left( b_n \chi_{\delta'}(a)^n\right) \right\} 
	\\
	&\leq \max \left\{ \deg(p) \cdot (-v)(\chi_{\delta'}(a)) , 0 \right\} \leq  \deg(p) \cdot (-v)(\varepsilon) ,
\end{align*}
so
\begin{align*}
	(-v)(a) \leq \left( \deg(p)\sum_{\delta,\delta'\in\Delta} q_{\delta,\delta'}  \right) (-v)(\varepsilon)
	\leq   \left\lvert \deg(p)\sum_{\delta,\delta'\in\Delta} q_{\delta,\delta'}  \right\rvert (-v)(\varepsilon),
\end{align*}
so we can define
$$
C \coloneqq \left\lvert \deg(p)\sum_{\delta,\delta'\in\Delta} q_{\delta,\delta'} \right\rvert 
$$
to conclude the first part of the proof. For the second statement, assume we have $\lambda \in \Lambda_{>0}$, $a\in A_\F$ with $(-v)(\chi_\delta(a))\leq \lambda$ for all $\delta \in \Delta$. Choose $\varepsilon \in \F_{>1}$ such that $(-v)(\varepsilon) = 2\lambda$. Then $(-v)(\chi_\delta(a))\leq \lambda < (-v)(\varepsilon)$, so $\chi_\delta(a) < \varepsilon$ and the first part can be applied to obtain $(-v)(a) \leq C \cdot (-v)(\varepsilon) = C \cdot 2\lambda \eqcolon C' \cdot \lambda$ for $C' = 2C$.
\end{proof}

\subsection{Stabilizer of a point} \label{sec:stabilizer}

We denote the equivalence class of $\operatorname{Id}$ by $o\in \B$. The stabilizer of $o$ has been calculated by \cite{Tho} when $\F$ is a Robinson field and by \cite{KrTe} more generally. We first describe the special case of the action of $A_\F$ on $\B$. For any semi-algebraic subset $H_\F\subseteq G_\F$ we write $H_\F(O) \coloneqq H_\F \cap O^{n\times n}$, where $O$ is the valuation ring.

\begin{proposition}\label{prop:stab_A}
	The following are equivalent for $a\in A_\F$.
	\begin{itemize}
		\item [(i)] $a \in \operatorname{Stab}_{G_\F}(o)$.
		\item [(ii)]  $N_\F(a) \in O$.
		\item [(iii)] $\chi_\alpha(a) \in O \ \forall \alpha \in \Sigma $.
		\item [(iv)]  $\chi_\alpha(a) \in O^\times \ \forall \alpha \in \Delta $.
		\item [(v)] $a \in A_\F(O)$.
	\end{itemize}
Moreover, the apartment can be identified with $\A \cong A_\Lambda \cong A_\F/A_\F(O) \cong A_\F. o $.
\end{proposition}
\begin{proof}
	If $a\in \operatorname{Stab}_{G_\F}(o)$, then $d(\operatorname{Id},a.\operatorname{Id}) = (-v)(N_\F(a)) = 0$, so $N_\F(a) \in O$. This implies by the definition of $N_\F$ that $\chi_\alpha(a) \in O \cap O^{-1} = O^\times$ for all $\alpha \in \Sigma$, in particular for all $\alpha \in \Delta$. There is some $\varepsilon \in O \cap \F_{\geq 1}$ such that $\chi_\alpha(a) \leq \varepsilon$ for all $\alpha \in \Delta$, so by Lemma \ref{lem:matrixval-bound}, the valuations of the matrix entries of $a$ are all bounded by $(-v)(\varepsilon) = 0$, so $a \in A_\F(O)$.
On the other hand, if $a\in A_\F(O)$, then the linear map $\operatorname{Ad}(a) \colon \frakg \to \frakg, X \mapsto aXa^{-1}$ restricts to multiplication by $\chi_\alpha(a)$ on $\frakg_\alpha$, from which can be concluded that $\chi_\alpha(a) \in O$.	Then $N_\F(a) \in O$ and $d(o,a.o) =0$, so $a \in \operatorname{Stab}_{G_\F}(o)$.

By Section \ref{sec:modelapartment}, $\A \cong A_\Lambda \cong A_\F/\{a \in A_\F \colon N(a) \in O \} \cong A_\F / \operatorname{Stab}_{A_\F}(o) = A_\F / A_\F(O)$, explicitly $x\in\A$ corresponds to $a.o=aa\tran = a^2$ if and only if $f(x)=[a^2]$ in the language of Theorem \ref{thm:apt} and by the orbit stabilizer theorem for the action of $A_\F$ on $\B$, we have $A_\F.o \cong A_\F/\operatorname{Stab}_{A_\F}(o) \cong \A$.
\end{proof}

\begin{theorem}\label{thm:stab} 
	The stabilizer of $o\in\B$ in $G_\F$ is $G_\F(O) $.
\end{theorem}
\begin{proof}	
	Let $g=kak' \in G_\F = K_\F \Ap_\F K_\F$. If $g\in \operatorname{Stab}_{G_\F}(o)$, then $d(\operatorname{Id},g.\operatorname{Id}) = (-v)(N_\F(a)) = 0$. By Proposition \ref{prop:stab_A}, $a \in A_\F(O)$. Now since $K_\F = K_\F(O)$, $g = kak' \in G_\F(O)$.
	
	If on the other hand we start with a $g\in G_\F(O)$, then $a\in \A_\F(O)$, and by Proposition \ref{prop:stab_A}, $N_\F(a) \in O$ and thus also $d(\operatorname{Id},g.\operatorname{Id}) =d(\operatorname{Id},a.\operatorname{Id}) =0 $, hence $g \in \operatorname{Stab}_{G_\F}(o)$.
\end{proof}

As an application of the Iwasawa retraction, Theorem \ref{thm:UAKret}, we can give a group decomposition for the stabilizer of $o$ in $\B$.

\begin{corollary}\label{cor:UAKO}
	There is an Iwasawa decomposition $G_\F(O) = U_\F(O) A_\F(O) K_\F$, meaning that for every $g \in G_\F(O)$ there are unique $u \in U_\F(O)$, $a\in A_\F(O)$, $k\in K_\F=K_\F(O)$ with $g=uak$.
\end{corollary}
\begin{proof}
	Let $g=uak \in G_\F(O) \subseteq  U_\F A_\F K_\F$. We have $\rho(g.\operatorname{Id}) = a.\operatorname{Id}$. Since $g\in G_\F(O)$, we have by Theorem \ref{thm:UAKret}
	\begin{align*}
		d(\operatorname{Id},a.\operatorname{Id}) &= d(\operatorname{Id}, \rho(g.\operatorname{Id}))  \leq  d(\operatorname{Id}, g.\operatorname{Id}) = 0.
	\end{align*}
	This means that $a \in A_\F \cap G_\F(O) = A_\F(O)$. Note that since $K_\F$ stabilizes $\operatorname{Id}\in \X_\F$, $K_\F=K_\F(O)$. Since $G_\F(O)$ is a subgroup of $G_\F$, $u = gk^{-1}a^{-1} \in G_\F(O)$, so $u \in U_{\F}(O)$.
\end{proof}


%% file: proof_B.tex
\section{Verification of the axioms for \texorpdfstring{$\B$}{B}}\label{sec:Bisbuilding} 

We continue in the setting of Section \ref{sec:building_def}: $\K$ and $\F$ are real closed fields such that $\K \subseteq \R \cap \F$ and $\F$ is non-Archimedean with an order compatible valuation $v$. Let $G<\operatorname{SL}_n$ be a semisimple connected self-adjoint algebraic $\K$-group. Let $\B = \X_\F/\!\sim$ as in Section \ref{sec:building} and $\A \cong A_\F.o  \subseteq \B$ as in Proposition \ref{prop:stab_A}. Denote the inclusion by $f_0 \colon \A \to \B$. Explicitly, in terms of Theorem \ref{thm:apt}, we have for all $x\in \A$, $f_0(x) = a.o$ if and only if $f(x)=[a^2]$. We define a set of charts
$$
\Fun = \left\{ g.f_0 \colon \A \to \B \colon g \in G_\F \right\}.
$$
The goal of this section is to show that $(\B,\Fun)$ is an affine $\Lambda$-building in the sense of Section \ref{sec:lambda_building}. Recall that a root system $\Phi$ is called \emph{reduced} if $\alpha \in \Phi$ implies $2\alpha \notin \Phi$. Recall that the root system $\Sigma$ of the Lie group $G_\R$ coincides with the root system $\KPhi$ of the algebraic group $G$ relative to an $\F$-split torus.

\begin{theorem}\label{thm:B_is_building}
	If the root system $\Sigma$ of $G_\R$ is reduced, then the pair $(\B,\Fun)$ is an affine $\Lambda$-building of type $\A = \A(\Sigma^\vee,\Lambda,T)$, where $T=\A$ is the full translation group.
\end{theorem}

We will show the theorem by checking the set of axioms (A1), (A2), (A3), (A4), (TI), (EC), as described in Theorem \ref{thm:equivalent_axioms}. Axioms (A1), (A3) and (TI) are treated in Section \ref{sec:axiomsA1A3TI} and follow easily from what we have developed so far.
Before proving the other axioms, we will develop some theory. We will introduce explicit root group valuations to investigate the action of the root groups $(U_\alpha)_\F$. This results in partial results about $W_a$-convexity and allows us to describe the pointwise stabilizers of the fundamental Weyl chamber $C_0$, the entire apartment $\A$ and half-apartments $H_\alpha$. We will then follow some ideas of \cite{BrTi, Lan96}, that allow us to describe the stabilizers of arbitrary finite subsets of $\A$. In Section \ref{sec:A2_subsubsection} we first show the change of charts condition of axiom (A2), which we then use to describe the stabilizers of arbitrary (not necessarily finite) subsets of $\A$, which in turn allows us to conclude the proof of axiom (A2). Finally, axioms (A4) and (EC) are proven in Sections \ref{sec:A4} and \ref{sec:EC}.
 
 We develop much of the theory also for the case where $\Sigma$ may not be reduced, and explicitly mention when reduced $\Sigma$ is needed.
 The assumption that $\Sigma$ is reduced is directly used in the proof of axioms (A2) and (EC). Axiom (A4) uses the assumption indirectly, as it relies on the statement of (A2) in its proof. The remaining axioms (A1), (A3) and (TI) do not need the assumption. Further discussion of this assumption can be found in Section \ref{sec:beyond}.
 An alternative, much simpler proof of axiom (A2) in the case of $G_\F = \operatorname{SL}(n,\F)$ is given in Appendix \ref{sec:appendixSLn}.

\subsection{Axioms (A1), (A3) and (TI)}\label{sec:axiomsA1A3TI}

Three of the axioms follow directly from what we have done.

\begin{lemma}\label{lem:A1}
	The pair $(\B,\Fun)$ satisfies axiom
	\begin{enumerate}
		\item [(A1)] For all $f\in \Fun$, $w\in W_a$, $f \circ w \in \Fun$.
	\end{enumerate}
\end{lemma}
\begin{proof} 
	Let  $w= (t^x,w_s) \in W_a = T \rtimes W_s$, where $t$ is a translation by $x\in \A$ and $w_s=s_1 \cdots s_k$ is a product of reflections. Let $a\in A_\F$ such that $f_0(x)=a.o$ as in Theorem \ref{thm:apt}. Let $n_i \in N_\F$ that act like $s_i$ as in Lemma \ref{lem:actionWs} and let $n = n_1 \cdots n_k$. Let $g\in G_\F$ with $g.f_0 = f$.
	
	Now for all $y \in \A$, let $b_i\in A_\F$ such that $f_0(s_i \cdots s_k(y))= b_i.o$. Then by Corollary \ref{cor:f_additive} and Lemma \ref{lem:actionWs}
	\begin{align*}
		f\circ w (y) &= g.f_0(t^x(w_s(y))) 
		= g.f_0(x + w_s(y)) 
		= g.ab_1.o = ga.f_0(w_s(y)) \\
		&= ga.f_0(s_1 \cdots s_k(y)) = ga.n_1.b_2.o = gan_1.f_0(s_2 \cdots s_k(y))
		= \cdots \\
		&= gan_1 \cdots n_k.f_0(y) = gan.f_0(y),
	\end{align*} 
so for $\tilde{g} \coloneqq gan \in G_\F$ we have $f \circ w = \tilde{g}.f_0 \in \Fun$, proving axiom (A1).
\end{proof}
\begin{lemma}\label{lem:A3}
	The pair $(\B,\Fun)$ satisfies axiom 
	\begin{enumerate}
		\item [(A3)] For all $p,q\in \B$, there exists $f\in \Fun$ such that $p,q\in f(\A)$.
	\end{enumerate}
\end{lemma}
\begin{proof}
	This follows from Proposition \ref{prop:trans}(c): if $[x],[y]\in \B$, for points $x,y \in \X_\F$, then there is a $g \in G_\F$ such that $g.x = \operatorname{Id} \in A_\F.\operatorname{Id} $ and $g.y \in A_\F.\operatorname{Id}$. Then $[x],[y] \in g^{-1}.f_0(\A)$ for the chart $g^{-1}.f_0 \in \Fun$.
\end{proof}

Axiom (TI) just states that $\B$ satisfies the triangle inequality. The triangle inequality was proven in Theorem \ref{thm:pseudodistance} using Kostant convexity and Iwasawa-retractions.

\subsection{Valuation properties of \texorpdfstring{$U_\F$}{U(F)}} \label{sec:root_group_valuation}

Recall from \cite[Section \ref{I-sec:U}]{AppAGRCF}, that $U_\F$ has subgroups $(U_\alpha)_\F$ for $\alpha \in \Sigma_{>0}$ defined as $(U_\alpha)_\F = \exp((\frakg_\alpha)_\F \oplus (\frakg_{2\alpha})_\F)$. Since $(U_\alpha)_\F$ is unipotent, the matrix exponential $\exp \colon (\frakg_\alpha)_\F \oplus (\frakg_{2\alpha})_\F \to (U_\alpha)_\F$ and the matrix logarithm 
\begin{align*}
	\log \colon (U_\alpha)_\F &\to (\frakg_\alpha)_\F \oplus (\frakg_{2\alpha})_\F \\
	u & \mapsto \sum_{k=1}^\infty (-1)^{k+1} \frac{(u-\operatorname{Id})^k}{k}
\end{align*}
are just polynomials on $(U_\alpha)_\F$. Viewing $\frakg_\F \subseteq \mathfrak{sl}(n ,\F) \subseteq \F^{n\times n}$, we can speak about the matrix entries $Z_{ij}$ of $Z \in \frakg_\F$. Let as in Section \ref{sec:matrixvaluation}
$$
(-v)(Z) := \max_{ij} \{ (-v)(Z_{ij}) \}.
$$
Inspired by \cite[6.2]{BrTi} and in order to avoid talking about matrix entries too much we introduce the \emph{root group valuations} $\varphi_\alpha$
\begin{align*}
	\varphi_\alpha \colon (U_{\alpha})_\F &\to \Lambda  \cup \{ -\infty \}\\
	\exp(X+X') & \mapsto \max \left\{ \max_{i,j} \left\{(-v)(X_{ij}) \right\}, \frac{1}{2} \max_{i,j} \left\{ (-v)(X'_{ij}) \right\}\right\}
\end{align*}
where $X \in (\frakg_\alpha)_\F$ and $X' \in (\frakg_{2\alpha})_\F$. The following Lemma justifies their name. Readers familiar with Bruhat-Tits theory may notice that we have chosen the opposite sign convention to simplify the notation.

\begin{lemma}\label{lem:BTgroup_valuation}
	For all $\alpha \in \Sigma$ and $u,v \in (U_{\alpha})_\F$
	$$
	\varphi_\alpha(uv) \leq \max \left\{ \varphi_\alpha (u) , \varphi_\alpha(v) \right\}
	$$
	and if $\varphi_\alpha(u) \neq \varphi_\alpha(v)$, then equality holds.
\end{lemma}
\begin{proof} We first claim that for all $X,Y \in \frakg_{\alpha}$,
	$$
	\frac{1}{2} (-v)\left(\left[X,Y\right]\right) \leq \max \{ (-v)(X), (-v)(Y) \}.
	$$
	Indeed in matrix entries, $[X,Y]_{ij} = \sum_{k} X_{ik}Y_{kj} - Y_{ik}X_{kj}$, and hence
	\begin{align*}
		\frac{1}{2}(-v)\left(\left[X,Y\right]\right)  &= \frac{1}{2} \max_{ij} \left\{ (-v)\left( \sum_{k} X_{ik}Y_{kj} - Y_{ik}X_{kj}  \right)  \right\} \\
		&\leq \frac{1}{2} \max_{ijk}\left\{ (-v)(X_{ik} Y_{kj}) , (-v)(Y_{ik}X_{kj}) \right\} \\
		& \leq \frac{1}{2} \max_{ijk} \{ (-v)(X_{ik}) + (-v)(Y_{kj}), (-v)(Y_{ik}) + (-v)(X_{kj})  \} \\
		&\leq \frac{1}{2} \max \{ (-v)(X) + (-v)(Y) ,  
		(-v)(Y) + (-v)(X)  \} \\
		&\leq \frac{1}{2} \max \{ 2 (-v)(X), 2(-v)(Y) \} = \max \{ (-v)(X), (-v)(Y) \}.
	\end{align*}
	
	Now let $u = \exp(X+X')$ and $v=\exp(Y+Y')$ for  $X,Y \in (\frakg_\alpha)_\F$ and $X', Y' \in (\frakg_{2\alpha})_\F$. Then
	$$
	u\cdot v = \exp \left( X+Y + X'+Y'+ \frac{1}{2}[X,Y] \right)
	$$
	by the Baker-Campbell-Hausdorff-formula \cite[Proposition \ref{I-prop:BCH}]{AppAGRCF}. By the claim
	\begin{align*}
		\varphi_\alpha(uv) &= \max \left\{ (-v)(X+Y) , \frac{1}{2} (-v)\left(X'+Y' + \frac{1}{2}[X,Y]\right) \right\} \\
		& \leq \max\left\{    (-v)(X), (-v)(Y) , \frac{1}{2}(-v)(X'), \frac{1}{2}(-v)(Y')    \right\} \\
		&= \max \{  \varphi_\alpha(u), \varphi_{\alpha}(v)  \}.
	\end{align*}
	
	Without loss of generality $\varphi_\alpha(u) < \varphi_\alpha(v)$ and we distinguish two cases. If $\varphi_\alpha(v) = (-v)(Y) \geq \frac{1}{2}(-v)(Y')$, then $(-v)(X) < (-v)(Y)$ and $\frac{1}{2}(-v)(X') < (-v)(Y)$. Thus using the claim
	\begin{align*}
		\varphi_\alpha (uv) & = \max \left\{ (-v)(Y), \frac{1}{2} (-v) \left( X' + Y' + \frac{1}{2}[X,Y] \right)\right\} \\
		&= \max \left\{  (-v)(Y), \frac{1}{2}(-v)(X'), \frac{1}{2}(-v)(Y'), (-v)(X), (-v)(Y) \right\} \\
		&		= (-v)(Y) = \varphi_\alpha(v) .
	\end{align*}
	If on the other hand $\varphi_\alpha(v) = \frac{1}{2}(-v)(Y') > (-v)(Y)$, then $(-v)(X)<\frac{1}{2}(-v)(Y')$ and $(-v)(X') <  (-v)(Y')$. Thus using the claim
	\begin{align*}
		(-v)\left(X' + \frac{1}{2}[X,Y]\right) & \leq \max\left\{ (-v)(X') , 2(-v)(X), 2(-v)(Y) \right\} < (-v)(Y')
	\end{align*}
	and
	\begin{align*}
		\varphi_\alpha(uv) &=\max \left\{ (-v)(X+Y) ,\frac{1}{2}(-v)\left(Y'+ X' + \frac{1}{2}[X,Y]    \right)  \right\} \\
		&= \max \left\{ (-v)(X+Y) , \frac{1}{2} (-v)(Y')  \right\} = \frac{1}{2}(-v)(Y') = \varphi_\alpha(v). \qedhere
	\end{align*}
\end{proof} 

\begin{lemma}\label{lem:orthogonal_valuation}
	Let $B_\theta$ be the scalar product defined on $\frakg_\F$ in \cite[Section \ref{I-sec:killing_involutions_decompositions}]{AppAGRCF}, then
	$$
	(-v)(X) = (-v)\left(\sqrt{B_\theta(X,X)}\right)
	$$
	for all $X \in \frakg_\F$. If $X,Y \in \frakg_\F$ are orthogonal with respect to $B_\theta$, then 
	$$
	(-v)(X+Y) = \max\{ (-v)(X), (-v)(Y) \}.
	$$ 
\end{lemma}
\begin{proof}
	We note that $(-v)(X) = (-v)(\lVert X\rVert_\infty)$ with the supremum norm $\lVert \cdot \rVert_\infty$. Since on $\frakg_\R$, all norms are equivalent, we can use the transfer principle to obtain a constant $k \in \N$ such that for all $X \in \frakg_\F$
	$$
	\frac{1}{k} \sqrt{B_\theta(X,X)} \leq \lVert X \rVert_\infty \leq k \sqrt{B_{\theta}(X,X)},
	$$
	from which $(-v)(X) = (-v)\left(\lVert X\rVert_\infty\right) = (-v)\left(\sqrt{B_\theta(X,X)}\right)$ follows.
	
	Now if $X,Y \in \frakg_\F$ are orthogonal, $B_\theta(X,Y) = 0$, then
	\begin{align*}
		(-v)(X+Y) &= (-v)\left(\sqrt{B_\theta (X+Y, X+Y)} \right)= \frac{1}{2}(-v)\left(B_\theta (X,X)+ B_{\theta}(Y,Y)\right) \\ 
		&= \frac{1}{2}\max\left\{ (-v)(B_{\theta}(X,X)), (-v)(B_{\theta}(Y,Y)) \right\} \\
		& = \max\{(-v)(X), (-v)(Y) \}
	\end{align*}
	where we used positive definiteness to see that $B_\theta(X,X)$ and $B_{\theta}(Y,Y)$ do not cancel.
\end{proof}

The root group valuation allows us to describe when an element of $(U_\alpha)_\F$ fixes the base point $o \in \B$.

\begin{lemma}\label{lem:stab_Ualpha}
	Let $\alpha \in \Sigma$ and $u \in (U_\alpha)_\F$. The following are equivalent
	\begin{enumerate}
		\item [(i)] $u.o = o$,
		\item [(ii)] $\varphi_\alpha(u) \leq 0$,
		\item [(iii)] $u = \exp(X + X')$ for some $X\in (\frakg_\alpha)_\F(O), \ X'\in (\frakg_{2\alpha})_\F(O)$,
		\item [(iv)] $\log(u) \in \frakg_\F(O)$.
	\end{enumerate}
Moreover $\varphi_\alpha(u) < 0$ if and only if $(-v)\left((u - \Id)_{ij}\right) < 0$ for all $i,j$.
\end{lemma}
\begin{proof}
	Let $u = \exp(X+X')$ with $X\in (\frakg_\alpha)_\F, X'\in (\frakg_{2\alpha})_\F$. By Theorem \ref{thm:stab}, (i) $u.o = o$ is equivalent to $u \in G_\F(O)$ which by applying the logarithm, which is a polynomial, is equivalent to $X+X' \in \frakg_\F (O)$, (iv). By \cite[Proposition \ref{I-prop:root_decomp}]{AppAGRCF} the root space decomposition is orthogonal with respect to the Killing form. Since $X$ and $X'$ lie orthogonal to each other, $X,X' \in \frakg_\F(O)$ individually (iii), by Lemma \ref{lem:orthogonal_valuation}. It is then clear that all the matrix entries of $X$ and $X'$ lie in $O$, hence (ii) $\varphi_\alpha(u) \leq 0$. All these implications are equivalences.
	
    If $\varphi_\alpha(u) < 0$, then $(-v)(X_{ij})< 0$ and $(-v)(X_{ij}') < 0$. Then
    $$
    (-v)\left((u-\Id)_{ij} \right)= (-v)\left(\left( \sum_{k=1}^\infty \frac{(X+X')^k}{k!} \right)_{ij} \right)< 0.
    $$  
    If $(-v)((u-\Id)_{ij})<0$, then 
    $$
    (-v)((X+X')_{ij}) =(-v)( \log(u)_{ij} ) = (-v)\left( \left(\sum_{k=1}^{\infty} (-1)^{k+1}\frac{(u-\Id)^k}{k} \right)_{ij} \right) < 0,
    $$
    so with Lemma \ref{lem:orthogonal_valuation}, $\varphi_\alpha(u) < 0$.
\end{proof}

\subsection{\texorpdfstring{$W_a$}{W}-convexity for \texorpdfstring{$U_\F$}{U(F)}}\label{sec:Wconvexity_for_U}

We first use the Iwasawa retraction to show that if $p \in \A$ is sent to $q \in \A$ by an element $u \in U_\F$, then $p=q$.

\begin{proposition} \label{prop:UonA}
	For all $u\in U_\F$ and $a \in A_\F$
	$$
	ua.o \in \mathbb{A} \iff ua.o = a.o. 
	$$
\end{proposition}
\begin{proof}
	For all $u \in U_\F$ and $a \in A_\F$. If $ua.o = a.o$, then clearly $ua.o \in \mathbb{A}$. We now have to show the converse. So let $ua.o\in \mathbb{A}$, meaning that there exists $b\in  A_\F$ such that $d(ua.o , b.o) =0$. Since the Iwasawa retraction $\rho$ is distance diminishing by Theorem \ref{thm:UAKret}, we have
	\begin{align*}
		d(b.o, a.o) &= d(b.\operatorname{Id},a.\operatorname{Id}) = d(\rho(b.\operatorname{Id}),\rho(u.a.\operatorname{Id})) \\
		&\leq d(b.\operatorname{Id},ua.\operatorname{Id}) = d(b.o, ua.o) = 0,
	\end{align*}
	and thus $a.o = b.o = ua.o \in \A$, as claimed.
\end{proof}
We now show that the set of points in $\A$ that are fixed by an element $u\in (U_\alpha)_\F$ is a half-apartment and in particular $W_a$-convex.

\begin{proposition}\label{prop:Ualphaconv}
	Let $\alpha \in \Sigma$. For $u \in (U_\alpha)_\F$ we have
	$$
	\left\{ p \in \mathbb{A} \colon u.p  \in \mathbb{A}  \right\} = \left\{ a.o \in \A \colon \varphi_\alpha(u) \leq (-v)\left(\chi_\alpha\left(a\right)\right)  \right\}
	$$
	and therefore this set is a half-apartment (with a wall parallel to the wall defined by $(-v)(\chi_\alpha) = 0$) when $u \neq \operatorname{Id}$. 
\end{proposition}
\begin{proof}
	Let $p=a.o \in \A$ and $u = \exp(X+X')$ with $X\in (\frakg_\alpha)_\F, X'\in (\frakg_{2\alpha})_\F$. We have
	\begin{align*}
		ua.o \in \A \stackrel{\text{Prop. }\ref{prop:UonA}}{\iff} & ua.o = a.o 
		\stackrel{\text{Thm. }\ref{thm:stab}}{\iff} a^{-1}ua \in G_\F(O) \\
		\iff &a^{-1} \exp\left(X + X'\right) a \in G_\F(O) \\
		\stackrel{ 
			\text{
				 \cite[Lemma \ref{I-lem:aexpXa}]{AppAGRCF}
			}
		 }{\iff} &\exp\left( \chi_\alpha\left(a\right)^{-1} X + \chi_\alpha(a)^{-2} X' \right) \in G_\F(O).
	\end{align*}
	Denoting $u' := \exp\left( \chi_\alpha\left(a\right)^{-1} X + \chi_\alpha(a)^{-2} X' \right)  $, by Lemma \ref{lem:stab_Ualpha}, the above are equivalent to ${\varphi_\alpha(u') \leq 0}$. Using the abbreviation $(-v)(Z) := \max_{ij}\{ (-v)(Z_{ij}) \}$ for $Z \in \frakg_\F$, we have
	\begin{align*}
		\varphi_\alpha(u') &= \max \left\{ (-v)\left( \chi_\alpha(a)^{-1}X  \right), \frac{1}{2} (-v)\left( \chi_\alpha(a)^{-2} X'  \right) \right\} \\
		&= \max \left\{  (-v)\left( X  \right), \frac{1}{2} (-v)\left(  X'  \right)  \right\} - (-v)(\chi_\alpha(a)) \\
		&= \varphi_\alpha(u) - (-v)(\chi_\alpha(a)) \leq 0 
	\end{align*}
	Thus we see that $ua.o \in \A$ is equivalent to $\varphi_\alpha(u) \leq (-v)(\chi_\alpha(a))$.
	%
	%
	%
\end{proof}
Before upgrading the previous result to all of $U_\F$, we consider what happens to products in $U_\F$.

\begin{lemma} \label{lem:nn'}
	Let $\eta \in \Sigma_{> 0}, u\in \exp((\frakg_\eta)_\F)$ and $u' \in \exp(\bigoplus_{\alpha > \eta}(\frakg_\alpha)_\F)$. If $uu' \in G_\F(O)$, then $u\in G_\F(O)$ and $u' \in G_\F(O)$.
\end{lemma}
\begin{proof}
	For $X= \log(u)$ and $Y= \log(u')$, we consider the BCH-formula from \cite[Proposition \ref{I-prop:BCH}]{AppAGRCF}
	$$
	\exp(X)\exp(Y)= \exp\left(X+Y+\frac{1}{2}[X,Y] + \frac{1}{12}\left(\left[X,\left[X,Y\right]\right]+\left[Y,\left[Y,X\right]\right]\right) + \ldots \right)
	$$
	Now if $uu' = \exp(X)\exp(Y) \in G_\F(O)$, then
	$$
	\log(uu') = X+Y+\frac{1}{2}[X,Y] + \frac{1}{12}([X,[X,Y]]+[Y,[Y,X]]) + \ldots \in \frakn_\F(O),
	$$
	where $\frakn_\F = \bigoplus_{\lambda \in \Sigma_{ > 0}} (\frakg_\lambda)_\F$ is an orthogonal direct sum. We note that $X\in (\frakg_{\eta})_\F$ is orthogonal to the remaining terms of $\log(uu')$, hence $X \in (\frakg_\eta)_\F(O)$, see Lemma \ref{lem:orthogonal_valuation}. Thus $u= \exp(X) \in G_\F(O)$ and hence also $u' = u^{-1}uu' \in G_\F(O)$ since $G_\F(O)$ is a group.
\end{proof}

\begin{proposition}\label{prop:Uconv}
	For $u\in U_\F$ there are $k_\alpha \in \Lambda \cup \{-\infty \}$ for $\alpha \in \Sigma_{>0}$ such that
	$$
	\{p \in \A \colon u.p \in \A\} = \left\{ a.o \in \A \colon  k_\alpha \leq (-v)\left(\chi_\alpha \left(a\right)\right) \text{ for all } \alpha \in \Sigma_{>0}   \right\}
	$$
	and therefore the set of fixed points is a finite intersection of half-apartments. If $u=u_1 \cdot \ldots \cdot u_k$ for $u_i \in (U_{\alpha_i})_\F$ with $\Sigma_{>0} = \{ \alpha_1, \ldots , \alpha_k \}$ such that $\alpha_1 > \ldots > \alpha_k$, then $k_{\alpha_i} = \varphi_{\alpha_i}(u_i)$. If $u$ fixes all of $\A$, then $u = \operatorname{Id}$. 
\end{proposition}
\begin{proof}
	We use \cite[Lemma \ref{I-lem:BCH_consequence}]{AppAGRCF} to write
	$$
	u = u_1 \cdot \ldots \cdot u_k
	$$
	for some $u_i \in (U_{\alpha_i})_\F$ where $\Sigma_{>0}= \{\alpha_1, \ldots ,\alpha_k\}$ with $\alpha_1 > \ldots > \alpha_k$. By Proposition \ref{prop:UonA}, $u.p \in \A$ for $p=a.o\in \A$ if and only if $u.p=p$ and by Theorem \ref{thm:stab} and \cite[Proposition \ref{I-prop:anainN} ]{AppAGRCF}, $a^{-1}ua \in U_\F(O)$. Then
	$$
	a^{-1}ua = a^{-1}u_1a \cdot \ldots \cdot a^{-1}u_ka \in U_\F(O)
	$$ 
	and we can apply Lemma \ref{lem:nn'} repeatedly to obtain $a^{-1}u_1a \in U_\F(O)$, \ldots, $a^{-1}u_ka \in U_\F(O)$. By Proposition \ref{prop:Ualphaconv}, this implies
	$$
	k_{\alpha_i}:=\varphi_{\alpha_i}(u_i) \leq (-v)(\chi_{\alpha_i}(a))
	$$
	for all $\alpha_i \in \Sigma_{>0}$. All the previous implications are equivalences, concluding the description of the fixed point set of $u$. If $u$ fixes all of $\A$, then $a^{-1}ua \in U_\F(O)$ and $a^{-1}u_ia \in (U_{\alpha_i})_\F(O)$ for all $a \in A_\F$. This is only possible if $\varphi_{\alpha_i}(u_i) \leq (-v)(\chi_{\alpha_i}(a))$, so $u_i = \operatorname{Id}$ for all $i$ and thus $u = \operatorname{Id}$.
\end{proof}
As an application of the above, we can conclude that elements of $U_\F(O)$ fix the fundamental Weyl chamber, which is defined as 
$$
C_0 := \{a.o \in \A \colon 0 \leq (-v)(\chi_\alpha (a)) \text{ for all }\alpha \in \Sigma_{>0} \}.
$$

\begin{corollary}\label{cor:NO_fixes_s0}
	Let $u\in U_\F(O)$. Then $u.p = p$ for all $p \in C_0$.
\end{corollary}
\begin{proof}
	Elements $u\in U_\F(O)$ fix $o \in \B$, so by Proposition \ref{prop:Uconv},
	$$
	o \in \{p \in \A \colon u.p \in \A\} = \{ a.o \in \A \colon   k_\alpha \leq (-v)\left(\chi_\alpha \left(a\right)\right) \text{ for all } \alpha \in \Sigma_{>0}  \}
	$$
	from which we conclude that $k_\alpha \leq 0$ for all $\alpha \in \Sigma_{>0}$. If now $p=a.o \in C_0$, then $(-v)(\chi_\alpha(a)) \geq 0 \geq k_\alpha$, so applying Proposition \ref{prop:Uconv} again results in $u.p = p$.
\end{proof}

We also obtain that any $u\in U_\F$ can be conjugated by $A_\F$ into $G_\F(O)$. This statement will be useful in the proof of (A4).

\begin{proposition}\label{prop:exists_a_NO}
	For every $u \in U_\F$ there is an $a\in A_\F$ such that $a^{-1}ua \in U_\F(O)$.
\end{proposition}
\begin{proof}
	Use Proposition \ref{prop:Uconv} to obtain $k_\alpha \in \Lambda \cup \{-\infty\}$ for $\alpha \in \Sigma_{>0}$ such that $u$ fixes all the points in
	$$
	C:= \bigcap_{\alpha \in \Sigma_{ > 0}} H_{\alpha, k_\alpha}^+.
	$$ 
	We claim that $C$ is non-empty. Indeed, let $c\in \F_{>0}$ with $\lambda \coloneqq (-v)(c) > k_\alpha$ for all $\alpha >0$.
	Consider the fundamental Weyl chamber $D$ as a subset of $\operatorname{Span}_{\Q}(\Sigma^\vee)$ and let $H$ be an element in the interior of $D$, in fact we may assume that for all $\alpha \in \Sigma_{>0}$, $\alpha(H) \geq 1$. If we write $H=\sum_{\delta \in \Delta} q_\delta \delta^\vee \in \operatorname{Span}_\Q(\Sigma^\vee)$ and let
	$
	x \coloneqq \sum_{\delta \in \Delta} q_{\delta} \delta^\vee \otimes \lambda \in \A,
	$
	then 
	\begin{align*}
	\alpha(x) &= \sum_{\delta \in \Delta} q_\delta \alpha(\delta^\vee) \lambda = \lambda \alpha\left( \sum_{\delta \in \Delta} q_\delta \delta^\vee \right)
	= \lambda \alpha(H) \geq \lambda,
    \end{align*}
	see Sections \ref{sec:root_system} and \ref{sec:modelapartment} for definitions. Using the identification $\A \cong A_\Lambda$ from Theorem \ref{thm:apt}, $x\in \A$ corresponds to $[a] \in A_\Lambda$ with $\overline{\chi}_\alpha([a]) \geq \lambda$ for all $\alpha >0$, see Theorem \ref{prop:chi_f}. Any representative $a\in A_\F$ of $[a]\in A_\Lambda$ satisfies $(-v)\left(\chi_\alpha(a) \right) \geq \lambda  > k_\alpha$, so $a.o \in C$.	By \cite[Proposition \ref{I-prop:anainN}]{AppAGRCF}, $a^{-1}ua \in U_\F$ and by Theorem \ref{thm:stab}, $a^{-1}ua.o = a^{-1}a.o=o$ implies $a^{-1}ua \in U_\F(O)$.
\end{proof}

\subsection{Stabilizers of apartment, half-apartments and Weyl-chambers} \label{sec:stab_apt_halfapt_chambers}

In Corollary \ref{cor:NO_fixes_s0} we proved that $U_\F(O)$ fixes the fundamental Weyl chamber $C_0 = \{ a.o \in \A  \colon \chi_\alpha(a) \geq 1 \text{ for all } \alpha \in \Delta \}$, where $\Delta$ is a basis of $\Sigma$. The goal of this subsection is to describe the whole stabilizer of $C_0$ using the group $B_\F=U_\F A_\F M_\F$ from \cite[Section \ref{I-sec:BWB}]{AppAGRCF}. So far, we considered the action of $U_\F$ and $A_\F$, now we continue by investigating the action of $K_\F$. 

To determine which elements of $K_\F$ fix the standard apartment $\A \subseteq \B$, we will use some CAT(0) geometry on the symmetric space $\X_\R$, equipped with the right metric. 

\begin{theorem}[Proposition II.2.2 \cite{BH99}]\label{thm:CAT0_convex}
	If $\gamma, \gamma'$ are two unit-speed geodesics in a CAT(0)-space, then the function 
	$$t \mapsto d(\gamma(t),\gamma'(t))$$ is convex.
\end{theorem}

In the non-standard symmetric space $\X_\F$, the elements of $K_\F$ that fix all points of the standard maximal flat $A_\F.\operatorname{Id}$ are exactly $\operatorname{Cen}_{K_\F}(A_\F) =: M_\F$. In the next Proposition we show that these are also exactly the elements in $K_\F$ that fix $\A \subseteq \B$ pointwise.

\begin{proposition}\label{prop:K_fixing_A_is_M}
	Elements $k\in K_\F$ fix $\A \subseteq \B$ pointwise if and only if $k \in M_\F$. 
\end{proposition}
\begin{proof}
	Recall that the distance on the non-standard symmetric space $\X_\F$ is given by $d  = (-v) \circ N \circ \delta \colon \X_\F \times \X_\F \to \Lambda$. We claim that the first-order formula
	\begin{align*}
		\varphi \colon \quad \forall k \in K \colon &  (\forall a \in A \colon a.\operatorname{Id}=k.a.\operatorname{Id}) \\
		& \vee ( \forall c >0 \ \exists a \in A \colon N(\delta(a'.\operatorname{Id},k.a.\operatorname{Id}))>c )
	\end{align*}
	holds over $\F$. The situation is illustrated in Figure \ref{fig:rotation_statement}. 
	Informally, $\varphi$ states that $k\in K_\F$ either fixes all points in the maximal flat $A_\F.\operatorname{Id}$, or there are points $a.\operatorname{Id}$ that are sent arbitrarily far away by $k$. To prove that $\varphi$ holds over $\F$, it suffices to show $\varphi$ over $\R$ by the transfer principle.
	
	\begin{figure}[h]
		\centering
		\includegraphics[width=0.6\linewidth]{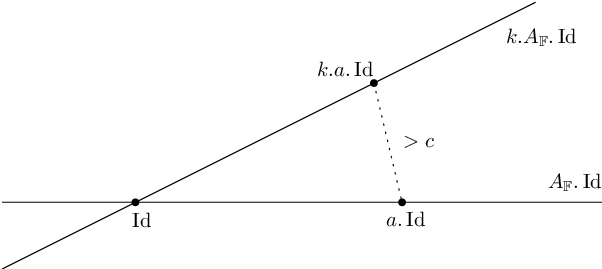}
		\caption{The first-order formula $\varphi$ states that either $k$ fixes all points in $A_\F$ or there are points $a.\operatorname{Id}, k.a.\operatorname{Id}$ whose distance is arbitrarily large. }
		\label{fig:rotation_statement}
	\end{figure}
	
	We note that changing the $W_s$-invariant multiplicative norm $N$ only changes $\X_\R$ up to quasi-isometry, so we may choose $N_\R$ coming from a scalar product, even if $N_\R$ is then not semialgebraic, as the truth of $\varphi$ only depends on $N_\R$ up to equivalency. So consider 
	\begin{align*}
		N_\R \colon A_\R.\operatorname{Id} &\to \R_{\geq 1} \\
		a.\operatorname{Id} &\mapsto \exp\left( \sqrt{B_\theta(\log(a),\log(a))}\right)
	\end{align*}
	for the scalar product $B_\theta$ on $\fraka$. Then by the general theory of symmetric spaces of non-compact type, $\X_\R$ with the distance $d = \log \circ N_\R \circ \delta_\R$ is a complete CAT(0)-space.
	Consider a unit-speed geodesic $\gamma \colon \R \to A_\R.\operatorname{Id} \subseteq \X_\R$ passing through $\gamma(0)=\operatorname{Id}$. Then $k.\gamma$ is also a unit-speed geodesic passing through $\operatorname{Id}$. By Theorem \ref{thm:CAT0_convex}, the function $f \colon t \mapsto d(\gamma(t),k.\gamma(t))$ is convex. Since $f$ is non-negative and $f(0)=0$, $f$ then has to be constant (hence $a.\operatorname{Id}=k.a.\operatorname{Id}$ for all $a\in A_\F$) or eventually be larger than $\log(c)$ for every $c\in \R_{>0}$ (hence there is some $a\in A_\F$ such that $N_\R(\delta_\R(a.\operatorname{Id},k.a.\operatorname{Id}))>c$).
	
	Now that $\varphi$ is established over $\F$, we consider some $k \in K_\F$ that fixes all points of $\A \subseteq \B$. Choosing $c \in \F_{>0}$ with $c\notin O$, we see that the second option in $\varphi$ cannot be true, whence $k.a.\operatorname{Id} = a.\operatorname{Id}$ for all $a\in A_\F$, or equivalently $k\tilde{a}k\tran = \tilde{a}$ for all $\tilde{a}= a^2 \in A_\F$. This means $k \in \operatorname{Cen}_{K_\F}(A_\F)$ using $k\tran = k^{-1}$.
\end{proof}

In Proposition \ref{prop:K_fixes_C0_in_M}, we strengthen the previous result by only requiring $k$ to fix a chamber of $\A$. We first need a preliminary result.

\begin{lemma}\label{lem:K_fixing_pm_mean}
	Let $k\in K_\F$ and $a,b\in A_\F$. If $k.a.o = a.o$, then $k.a^{-1}.o = a^{-1}.o$. If moreover $k.b.o = b.o$, then $k.\sqrt{ab}.o = \sqrt{ab}.o$. 
\end{lemma}
\begin{proof}
	We first assume that $A_\F$ consists of diagonal matrices, so we may write $a = \operatorname{Diag}(a_1, \ldots, a_n)$ and $b = \operatorname{Diag}(b_1,\ldots, b_n)$. Then
	\begin{align*}
		k.a.o = a.o & \iff a^{-1}ka \in G_\F(O) \iff \forall i,j\colon  k_{ij}\frac{a_j}{a_i} \in O \\
		&\iff \forall j,i \colon k\tran_{ij}\frac{a_i}{a_j} \in O \iff ak\tran a^{-1 } \in G_\F(O) \\
		&\iff k\tran.a^{-1}.o = a^{-1}.o \iff a^{-1}.o = k.a^{-1 }.o.
	\end{align*}
	Moreover, if $k.a.o = a.o$ and $k.b.o=b.o$, then $k_{ij}a_j/a_i \cdot k_{ij} b_j/b_i \in O$ for all $i,j$. Then also
	$$
	k_{ij} \frac{\sqrt{a_jb_j}}{\sqrt{a_ib_i}} \in O
	$$
	for all $i,j$, which translates to $k.\sqrt{ab}.o = \sqrt{ab}.o$.
	
	In general, the matrices in $A_\F$ may not be diagonal, but they are symmetric. By the spectral theorem for symmetric matrices, which holds over $\F$ by the transfer principle, $A_\F$ is orthogonally diagonalizable, meaning that there is some $Q \in \operatorname{SO}(n)$ such that $QA_\F Q\tran$ is diagonal. We can then apply the above arguments to the group $QG_\F Q\tran < \operatorname{SL}_n(\F)$. 
	For $k.a.o = a.o$ we obtain
	\begin{align*}
		a^{-1}ka \in G_\F(O) & \iff (QaQ\tran)^{-1} QkQ\tran QaQ\tran \in (QG_\F Q\tran)(O) \\
		&\iff QaQ\tran Qk\tran Q\tran (QaQ\tran)^{-1} \in (QG_\F Q\tran)(O) \\
		&\iff ak\tran a^{-1} \in G_\F(O)
	\end{align*}
	and complete the argument as above. When additionally $b^{-1}kb \in G_\F(O)$ we have
	\begin{align*}
		(QbQ\tran)^{-1} QkQ\tran QbQ\tran \in (QG_\F Q\tran)(O) 
	\end{align*}
	which by the above implies
	\begin{align*}
		(Q\sqrt{ab}Q\tran)^{-1} QkQ\tran Q\sqrt{ab}Q\tran \in (QG_\F Q\tran)(O)
	\end{align*}
	and thus$
	\sqrt{ab}^{-1}k\sqrt{ab} \in G_\F(O).
	$
\end{proof}

Let $C_0 = \{ a.o \in \A  \colon \chi_\alpha(a) \geq 1 \text{ for all } \alpha \in \Delta\}$ be the fundamental Weyl chamber associated to a basis $\Delta$ of $\Sigma$. 

\begin{proposition}\label{prop:K_fixes_C0_in_M}
	Let $k \in K_\F$ such that $k.p = p$ for all $p \in C_0$. Then $k \in M_\F$ and hence $k$ fixes all points in $\A$. In fact, if $k$ fixes all the points in $a.C_0$ for any $a \in A_\F$, then $k \in M_\F$.
\end{proposition}
\begin{proof}
	We first claim that every element $a\in A_\F$ is of the form $a=a_1a_2^{-1}$ for $a_1.o,a_2.o \in C_0$. To see this, we show that the first-order formula
	$$
	\varphi\colon \quad  \quad \forall a \in A \colon \exists a_1, a_2 \in A \colon a = a_1\cdot a_2^{-1} \wedge  \bigwedge_{\alpha \in \Sigma_{>0}} \chi_\alpha(a_1) \geq 1 \wedge \chi_\alpha(a_2) \geq 1 
	$$
	holds over $\R$ and then apply the transfer principle. Over $\R$ we can transfer the problem to the Lie algebra $\fraka_\R$ using the logarithm. We equip $\fraka$ with the distance defined by the scalar product $B_\theta$.
	Let $H:=\log(a)$ and $R := \sqrt{B_\theta(H,H)}$. 
	Since $\mathfrak{c}_0 := \{ H \in \fraka_\R \colon \alpha(H)>0 \text{ for all }\alpha \in \Sigma_{>0} \}$ contains an open cone, it contains a ball $B_{r}(H')$ for some $r>0$ and $H' \in \mathfrak{c}_0$. Scaling the ball by the factor $R/r$, we obtain that $B_{R}(R/r\cdot H') \subseteq \mathfrak{c}_0$. As in Figure \ref{fig:scaling_statement}, we define $H_1 = R/r\cdot H'$ and $H_2=H_1-H$ which lies on the boundary of $B_{R}(R/r\cdot H')$ and hence also in $\mathfrak{c}_0$. Then $a = \exp(H_1)\exp(H_2)^{-1}$, concluding the proof of $\varphi$ over $\R$ and hence over $\F$.
	
	\begin{figure}[h]
		\centering
		\includegraphics[width=0.45\linewidth]{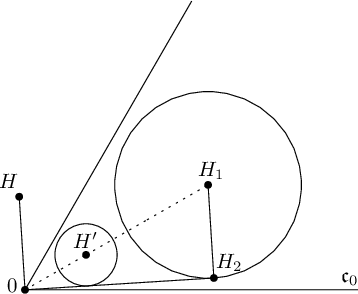}
		\caption{We use that the cone $\mathfrak{c}_0\subseteq \fraka_\mathbb{R}$ contains an open ball, which we can scale to obtain $H_1,H_2 \in \mathfrak{c}_0$ with $H=H_1-H_2$. }
		\label{fig:scaling_statement}
	\end{figure}
	
	If now $k\in K_\F$ fixes $C_0$ pointwise and $p=a.o \in \A$. Then $a=a_1a_2^{-1}$ as above with $a_1.o, a_2.o \in C_0$. Since $C_0$ is a cone, also $a_1^2.o, a_2^2.o \in C_0$, so $k.a_1^2.o = a_1^2.o$ and $k.a_2^2.o=a_2^2.o$. By Lemma \ref{lem:K_fixing_pm_mean}, then $k.a_2^{-2}.o = a_2^{-2}.o$ and 
	$$
	k.a.o = k.\sqrt{a_1^2 a_2^{-2}}.o = \sqrt{a_1^2 a_2^{-2}}.o = a.o,
	$$
	completing the first statement of the proof. If $k$ fixes $a.C_0$ for some $a\in A_\F$, a modification of the above argument similarly implies that $k$ fixes all of $\A$ pointwise.
\end{proof}
%
%

\begin{theorem}\label{thm:NO_fixes_s0}
	The pointwise stabilizer of $C_0$ in $G_\F$ is 
	$$B_\F(O) = U_\F(O)A_\F(O)M_\F.$$
\end{theorem}
\begin{proof}
	If $g=uak \in  B_\F(O) = U_\F(O)A_\F(O)M_\F$ and $p \in C_0$, then $g.p=uak.p=ua.p=u.p=p$, where the last equality follows from Corollary \ref{cor:NO_fixes_s0}.
	
	If $g$ fixes $C_0$ pointwise, in particular it fixes $o \in C_0$, hence $g\in G_\F(O)$, which can be decomposed to $G_\F(O) = U_\F(O)A_\F(O)K_\F$ by Corollary \ref{cor:UAKO}. Therefore $g=uak$ with $u\in U_\F(O)$ and $a \in  A_\F(O)$. Therefore we have that $k$ fixes $C_0$ pointwise. Proposition \ref{prop:K_fixes_C0_in_M} now concludes the proof by showing $k \in M_\F$.
\end{proof}

Recall that $S <G$ is a maximal $\K$-split torus and $A_\F < S_\F$ is the semialgebraically connected component of the identity. In the following we consider the groups $T_\F := \operatorname{Cen}_{G_\F}(A_\F)$ and $T_\F(O) := T_\F \cap G_\F(O) $. Sometimes (namely when $G$ is quasi-split), $T = \operatorname{Cen}_{G}(S)$ is a maximal algebraic torus, but we do not use this in what follows. 
\begin{lemma}\label{lem:BT_T_is_MA}
	We have $T_\F = \operatorname{Cen}_{G_\F}(A_\F) = M_\F \cdot A_\F$. 
\end{lemma}
\begin{proof}
	The inclusion $\supseteq$ is clear. For the other direction, let $g \in T_\F$ and choose an Iwasawa decomposition $g = nak$ with $n \in U_\F , a\in A_\F, k \in K_\F$, see \cite[Theorem \ref{I-thm:KAU}]{AppAGRCF}. For all $b \in A_\F$ we have $nak.b.o = b.nak.o = bna.o$, and thus $k.b.o = a^{-1}n^{-1}bna.o$. Denoting $\tilde{n} := (a^{-1}n^{-1}a)(a^{-1}bnb^{-1}a) \in U_\F$, where we made use of \cite[Proposition \ref{I-prop:anainN}]{AppAGRCF}. Then $k.b.o = \tilde{n}.b.o$ for all $b \in A_\F$. By Proposition \ref{prop:exists_a_NO} there is a $c\in A_\F$ such that $\tilde{\tilde{n}} := c\tilde{n}c^{-1} \in U_\F(O)$ and hence $\tilde{\tilde{n}}$ fixes the elements of $C_0$ by Theorem \ref{thm:NO_fixes_s0}. In particular, $\tilde{n}$ fixes $c.C_0$ and thus $k$ fixes $c.C_0$, since $k.p = \tilde{n}.p$ for all $p \in \A$. By Lemma \ref{prop:K_fixes_C0_in_M}, we then have $k \in M_\F$. This implies also $\tilde{n}.p=p$ for all $p\in \A$, hence $\tilde{n} = \operatorname{Id}$. Then $bnb^{-1} =n$ for all $b\in A_\F$, not only the ones with $(-v)(\chi_\alpha(b)) = 0$, hence by \cite[Lemma \ref{I-lem:aexpXa}]{AppAGRCF}, $n= \operatorname{Id}$. Thus $g = ak$ with $k \in M_\F$ and $a \in A_\F$.
\end{proof}

\begin{theorem}\label{thm:stab_A}
	The pointwise stabilizer of $\A$ in $G_\F$ is $T_\F(O) = A_\F(O)M_\F$.
\end{theorem}
\begin{proof}
	Elements of $A_\F(O)M_\F$ fix all points in $\A$. If $g\in G_\F$ fixes all points of $\A$, it fixes in particular the points in $C_0$, so $g = uak$ with $u \in U_\F(O), a \in A_\F(O), k \in M_\F$ by Theorem \ref{thm:NO_fixes_s0}. By the description of the stabilizer of $u$ in Proposition \ref{prop:Uconv}, $u$ can only fix all of $\A$ if $u=\Id$, so $g \in A_\F(O)M_\F$. By Lemma \ref{lem:BT_T_is_MA}, $T_\F(O) = A_\F(O)M_\F$.
\end{proof} 

\begin{theorem}\label{thm:NalphaO_fixes_H}
	Let $\alpha \in \Sigma$. The pointwise stabilizer of the half-apartment
	$$
	H_\alpha^+ = \{ a.o \in \A \colon (-v)(\chi_\alpha(a)) \geq 0\}
	$$ 
	in $G_\F$ is $(U_\alpha)_\F(O) A_\F(O) M_\F$.
\end{theorem}
\begin{proof}
	Without loss of generality, we may assume that $\Sigma$ is equipped with an order in which $\alpha>0$. Then $(U_\alpha)_\F(O) A_\F(O) M_\F \subseteq \operatorname{Stab}_{G_\F}(H_\alpha^+)$ by Proposition \ref{prop:Ualphaconv}. If $g \in \operatorname{Stab}_{G_\F}(H_\alpha^+)$, then $g = uak \in U_\alpha(O) A_\F(O) M_\F$ by Theorem \ref{thm:NO_fixes_s0} since $C_0 \subseteq H_\alpha$, so in fact $u$ fixes $H_\alpha^+$ pointwise. By \cite[Lemma \ref{I-lem:BCH_consequence}]{AppAGRCF}, $u = u_1\cdots u_k$ where $u_i \in U_{\alpha_i}$ such that $\alpha_1 > \ldots > \alpha_k > 0$. Then we can apply the refined version of Proposition \ref{prop:Uconv}, to obtain 
	$$
	\operatorname{Fix}_\A(u) = \left\{ a.o \in \A \colon k_{\alpha_i} \leq (-v)(\chi_{\alpha_i}(a)) \text{ for all } \alpha_i \in \Sigma_{>0} \right\} 
	$$
	where $k_{\alpha_i} = \varphi_{\alpha_i}(u_i)$. In our case, $H_\alpha^+ \subseteq \operatorname{Fix}_\A(u)$ and since when $k_{\alpha_i} \neq -\infty$
	$$
	H_{\alpha}^+ \subseteq \left\{ a.o \in \A \colon k_{\alpha_i} \leq (-v)(\chi_{\alpha_i}(a)) \right\}
	$$ 
	implies $\alpha = \alpha_i$, 
	$k_{\alpha_i} = -\infty$ for all $\alpha_i \neq \alpha$ and $k_\alpha = 0$, so $\varphi_{\alpha_i}(u_i) = -\infty$ implies $u_i = \operatorname{Id}$ whenever $\alpha_i \neq \alpha$, leading to $u \in (U_\alpha)_\F(O)$ concluding the proof.
\end{proof}

\begin{corollary}\label{cor:NalphaO_fixes_H_affine}
	Let $\alpha \in \Sigma$, $\ell \in \Lambda$. The pointwise stabilizer of the affine half-apartment
	$$
	H_{\alpha,\ell}^+ = \left\{ a.o \in \A \colon (-v)(\chi_\alpha(a) ) \geq \ell  \right\}
	$$
	in $G_\F$ is $U_{\alpha, \ell} A_\F(O) M_\F$, where $U_{\alpha, \ell} =  \{ u \in (U_\alpha)_\F \colon u \text{ stabilizes } H_{\alpha, \ell} \text{ pointwise} \}$.
\end{corollary}
\begin{proof}
	If $g\in G_\F$ stabilizes $H_{\alpha,\ell}^+$ pointwise and $a\in A_\F$ satisfies $(-v)(\chi_\alpha(a)) = \ell$, then $a^{-1} g a \in G_\F$ stabilizes $H_{\alpha, \ell}^+$ pointwise. By Theorem \ref{thm:NalphaO_fixes_H} and since $A_\F(O) M_\F \subseteq \operatorname{Cen}_{G_\F}(A_\F)$, we have $g \in a(U_\alpha)_\F(O) a^{-1} A_\F(O) M_\F$. While elements $u \in (U_\alpha)_\F(O)$ stabilize elements $b.o$ with $(-v)(\chi_\alpha(b)) = 0$ pointwise, $a ua^{-1}$ stabilize elements $ab.o$ with $(-v)(\chi_\alpha(ab)) = \ell + 0 $ pointwise, concluding the proof.
\end{proof}

\subsection{Bruhat-Tits theory for root groups} \label{sec:BT_root_groups}
The goal of Subsections \ref{sec:BT_root_groups}, \ref{sec:BT_rank_1} and \ref{sec:BT_higher_rank} is to describe the pointwise stabilizer of a finite subset $\Omega \subseteq \A$ in Theorem \ref{thm:BTstab_fin}. To obtain this, we first consider points fixed by $(U_\alpha)_\F$ in Section \ref{sec:BT_root_groups}, then points fixed by the rank one subgroups generated by $(U_\alpha)_\F$ and $(U_{-\alpha})_\F$ in Section \ref{sec:BT_rank_1}, before taking on the whole group $G_\F$ in \ref{sec:BT_higher_rank}. These subsections are inspired by the study of stabilizers in \cite[Sections 6 and 7]{BrTi}, for a good English reference see \cite{Lan96}. 

Let $\Omega \subseteq \A$ be any subset of the apartment $\A$. Let
\begin{align*}
	U_{\alpha, \Omega} &:= \left\{ u \in (U_\alpha)_{\F} \colon  u.p = p \text{ for all } p \in \Omega  \right\}
\end{align*}
denote the pointwise stabilizer of $\Omega$ in the group $(U_{\alpha})_\F$. The subscript $\F$ is no longer needed, since there is no corresponding group of $\R$-points. In view of Proposition \ref{prop:UonA}, Proposition \ref{prop:Ualphaconv} can be reformulated.
\begin{lemma}\label{lem:BTUalphaOmega}
	For roots $\alpha \in \Sigma$ and any subset $\Omega \subseteq \A$, we have
	\begin{align*}
		U_{\alpha,\Omega} 
		&= \left\{  u \in (U_\alpha)_\F \colon  \varphi_\alpha(u) \leq (-v)\left(\chi_\alpha\left(a\right)\right) \text{ for all } a.o \in \Omega  \right\} .
	\end{align*}
\end{lemma}
For any $\ell \in \Lambda$, denote 
$$
U_{\alpha, \ell} := \left\{ u \in (U_\alpha)_\F \colon \varphi_\alpha(u) \leq \ell \right\}.
$$
Note that if $\ell = \min_{a.o \in \Omega} \{ (-v)(\chi_\alpha(a)) \}$ and $\alpha \in \Sigma$, then Lemma \ref{lem:BTUalphaOmega} can be reformulated as
$$
U_{\alpha,\Omega} = U_{\alpha, \ell}.
$$
In the setting of \cite{BrTi}, $\ell = \inf_{a.o \in \Omega} \{ (-v)(\chi_\alpha(a)) \}$ with $U_{\alpha,\Omega} = U_{\alpha,\ell}$ always exists, since they work with $\Lambda=\R$. In our case we have to be more careful. If $\Omega$ is a finite set, this $\ell$ exists. In particular, when $|\Omega| = 1$, we have for any $a\in A_\F$
$$
U_{\alpha,\left\{ a.o \right\}} = U_{\alpha, (-v)(\chi_\alpha(a))}.
$$
\begin{lemma}\label{lem:BTaUa}
	Let $\alpha \in \Sigma, \ell \in \Lambda$ and $a \in A_\F$. Then
	$$
	a U_{\alpha , \ell} a^{-1} = U_{\alpha,\ell + (-v)(\chi_\alpha(a))}.
	$$
\end{lemma}
\begin{proof}
	Let $b \in A_\F$ with $(-v)(\chi_\alpha(b)) = \ell $, the existence of which can be concluded from \cite[Lemma \ref{I-lem:Jacobson_Morozov_oneparam}]{AppAGRCF}. We consider the single element set $\Omega = \{b.o\} $. Then $U_{\alpha, \ell} = U_{\alpha, \Omega}$. Let $u \in U_{\alpha, \ell}$. Then $aua^{-1} \in U_{\alpha, a.\Omega}$ since
	$$
	aua^{-1}.(a.b.o) = au.b.o = a.b.o  
	$$
	and since $aua^{-1} \in (U_\alpha)_\F$, see \cite[Proposition \ref{I-prop:anainN}]{AppAGRCF}. Now since 
	$$
	(-v)(\chi_{\alpha}(ab)) = (-v)(\chi_\alpha(a)) + (-v)(\chi_\alpha(b)) = \ell + (-v)(\chi_\alpha(a)) 
	$$
	we have $aU_{\alpha, \ell}a^{-1} = aU_{\alpha, \Omega}a^{-1} = U_{\alpha, a.\Omega} = U_{\alpha, \ell + (-v)(\chi_\alpha(a))}$.
\end{proof}
\begin{lemma}\label{lem:BTkUk}
	Let $\alpha \in \Sigma, \ell \in \Lambda$ and $k \in M_\F = \operatorname{Cen}_{K_\F}(A_\F)$. Then
	$$
	kU_{\alpha,\ell}k^{-1}= U_{\alpha,\ell}.
	$$
\end{lemma}
\begin{proof}
	Since $k \in M_\F$, it represents the trivial element 
	of the spherical Weyl group acting on the root system. In particular $kU_{\alpha}k^{-1} = 
	U_{\alpha}$. 
	
	Let $b\in A_\F$ with $(-v)(\chi_\alpha(b))= \ell$ and $\Omega = \{b.o\}$. Then $U_{\alpha,\ell} = U_{\alpha, \Omega}$. For any $u\in U_{\alpha,\ell}$ we have
	$$
	kuk^{-1}.b.o = ku.b.o = k.b.o = b.o,
	$$
	hence $kU_{\alpha,\Omega}k^{-1} = U_{\alpha, \Omega}$ concluding the proof.
\end{proof}

Putting the previous two results together shows that $U_{\alpha,\ell}$ is invariant under conjugation by elements in the pointwise stabilizer $T_\F(O)$.
\begin{lemma}\label{lem:BTtUt}
	Let $\alpha \in \Sigma, \ell \in \Lambda$ and $t \in T_\F(O)=M_\F A_\F(O)$. Then
	$$
	t U_{\alpha,\ell} t^{-1} = U_{\alpha,\ell}.
	$$
\end{lemma}

\subsection{Bruhat-Tits theory in rank 1} \label{sec:BT_rank_1}
Our goal in this section is to study the group generated by $U_{\alpha, \Omega}$ and $U_{-\alpha, \Omega}$. For this we will use Jacobson-Morozov in the form of \cite[Proposition \ref{I-prop:Jacobson_Morozov_real_closed}]{AppAGRCF}. Therefore we have to restrict ourselves to reduced root systems from now on. In this subsection we fix $\alpha \in \Sigma$ such that $(\frakg_{2\alpha})_\F = 0$ and $u \in (U_\alpha)_\F$.

When $u\neq \Id $, there is a $t \in \F$ and an $\mathfrak{sl}_2$-triplet $(X,Y,H)$ as in \cite[Proposition \ref{I-prop:Jacobson_Morozov_real_closed}]{AppAGRCF} $u = \exp(tX)$. Up to choosing a different $t\in \F$ we may assume $(-v)(X)=0$, in which case also $(-v)(Y)=0$ since $\exp(tY) = \exp(tX)^\tran$, and $(-v)(H)=0$, as $H=[X,Y]$. Moreover, there is an algebraic group homomorphism $\varphi_\F \colon \operatorname{SL}(2,\F) \to G_\F$ with finite kernel such that
\begin{align*}
	u = \exp(tX) = \varphi_\F \begin{pmatrix}
		1 & t \\ 0& 1
	\end{pmatrix}   \quad \text{and} \quad \exp(tY) = \varphi_\F \begin{pmatrix}
		1 & 0 \\ t & 1
	\end{pmatrix}.
\end{align*}
We note that $t \in O$ if and only if $u \in G_\F(O)$, using Lemma \ref{lem:stab_Ualpha}.
\begin{lemma}\label{lem:BTphiu_is_vt}
	For every $t\in \F$ we have
	$$
	\varphi_\alpha \left( \varphi_\F \begin{pmatrix}
		1 & t \\ 0 & 1
	\end{pmatrix} \right) = (-v)(t) = \varphi_{-\alpha} \left( \varphi_\F \begin{pmatrix}
		1 & 0 \\ t & 1
	\end{pmatrix} \right) .
	$$
\end{lemma}
\begin{proof}
	Since $(-v)(X) := \max_{ij}\{(-v)(X_{ij})\} = 0$, $\varphi_\alpha(u) = (-v)(tX) = (-v)(t)$. Similarly $\varphi_{-\alpha}(\exp(tY)) = (-v)(tY) = (-v)(t)$.
\end{proof}

	\begin{lemma}\label{lem:BTm_uuu}
		Let $\ell := (-v)(t)$. Then
		$$
		m(u) := \varphi_\F\begin{pmatrix}
			0 & t \\ -1/t & 0
		\end{pmatrix} = \varphi_\F \left(
		\begin{pmatrix}
			1 & 0 \\ -1/t & 1
		\end{pmatrix}\begin{pmatrix}
			1 & t \\ 0 & 1
		\end{pmatrix}\begin{pmatrix}
			1 & 0 \\ -1/t & 1
		\end{pmatrix},
		\right)
		$$
		in particular $m(u) \in U_{-\alpha,-\ell} U_{\alpha, \ell} U_{-\alpha,-\ell}$.
	\end{lemma}
	\begin{proof}
		Let 
		$$
		u' := \varphi_\F\begin{pmatrix}
			1 & 0 \\ -1/t & 1
		\end{pmatrix}  \in (U_{-\alpha})_\F.
		$$
		The matrix expression $m(u) = u' u u'$ is a direct calculation, showing $m(u) \in (U_{-\alpha})_\F(U_\alpha )_\F (U_{-\alpha})_\F$. By Lemma \ref{lem:BTphiu_is_vt}, $\varphi_\alpha(u) = \ell$ and $\varphi_{-\alpha}(u') = (-v)(-1/t) = -\ell$, concluding the proof.
	\end{proof}
	The element 
	$$
	m(u) := \varphi_\F \begin{pmatrix}
		0 & t \\ -1/t & 0
	\end{pmatrix} \in G_\F
	$$
	is contained in $ \operatorname{Nor}_{G_\F}(A_\F)$ by \cite[Lemma \ref{I-lem:Jacobson_Morozov_m}]{AppAGRCF} and thus a representative of an element of the affine Weyl group $W_a = \operatorname{Nor}_{G_\F}(A_\F)/\operatorname{Cen}_{G_\F}(A_\F) $. Recall that the affine Weyl group $W_a$ can be identified with $W_a = \A \rtimes W_s $, where $W_s = N_\F/M_\F= \operatorname{Nor}_{K_\F}(A_\F)/\operatorname{Cen}_{K_\F}(A_\F)$ is the spherical Weyl group.
	
	\begin{proposition}\label{prop:BTmu_in_Wa} The action of $m(u)$ decomposes as
		$$
		m(u) = \varphi_\F \begin{pmatrix}
			t & 0 \\ 0 & t^{-1}
		\end{pmatrix} \cdot \varphi_\F \begin{pmatrix}
			0 & 1 \\ -1 & 0 
		\end{pmatrix} =: a_t\cdot m \in A_\F \cdot N_\F 
		$$
		into an affine part represented by $a_t$ and a spherical part represented by $m$.
		The element $m$ represents the reflection $r_\alpha \in W_s$ and for $a_t$ we have $(-v)(\chi_\alpha(a_t)) = 2(-v)(t) = 2\varphi_\alpha(u)$. Thus $m(u)$ represents the affine reflection along the hyperplane
		$$
		\{ a.o \in \A \colon (-v)(\chi_\alpha(a)) = \varphi_\alpha(u) \}.
		$$
	\end{proposition}
	\begin{proof}
		The decomposition of $m(u)$ is a direct calculation. We investigate the action of $m$. We may decompose $A_\F = (A_{\pm \alpha})_\F \cdot A_{\perp}$ as a direct product, where
		$$
		(A_{\pm \alpha})_\F = \varphi_\F \left(\left\{ \begin{pmatrix}
			\lambda & 0 \\ 0 & \lambda^{-1}
		\end{pmatrix} \colon \lambda >0 \right\}\right)
		$$
		and
		$$
		A_\perp = \{ a\in A_\F \colon \chi_\alpha(a) = 1 \}.
		$$
		The reflection $r_\alpha \colon \A \to \A$ is defined by $r_\alpha(a_\alpha.o) = a_{\alpha}^{-1}.o$ for all $a_\alpha \in (A_{\pm \alpha})_\F$ and $r_\alpha(a_\perp.o) = a_\perp.o $ for all $a_\perp \in A_\perp$. For $a_\alpha \in (A_{\pm \alpha})_\F$ there is some $\lambda >0$ such that
		\begin{align*}
			m.a_\alpha.o &= m.a_\alpha.m^{-1} . o= \varphi_\F \left(\begin{pmatrix}
				0 & 1 \\ -1 & 0
			\end{pmatrix}  \begin{pmatrix}
				\lambda & 0 \\ 0 & \lambda^{-1}
			\end{pmatrix}\begin{pmatrix}
				0 & -1 \\ 1 & 0
			\end{pmatrix} \right). o \\
			&=  \varphi_\F\begin{pmatrix}
				\lambda^{-1} & 0 \\ 0 & \lambda
			\end{pmatrix}  .o = a_\alpha^{-1}.o.
		\end{align*}
		For $a_\perp \in A_{\perp}$ we use 
		\begin{align*}
			m &= \varphi_\F \begin{pmatrix}
				0 & 1 \\ -1 & 0
			\end{pmatrix} = \varphi_\F \left( \begin{pmatrix}
				1 & 0 \\ -1 & 1
			\end{pmatrix} \begin{pmatrix}
				1 & 1 \\ 0 & 1
			\end{pmatrix} \begin{pmatrix}
				1 & 0 \\ -1 & 1
			\end{pmatrix} \right) \\
			& \in \exp((\frakg_{-\alpha})_\F) \cdot \exp((\frakg_{\alpha})_\F) \cdot \exp((\frakg_{-\alpha})_\F)
		\end{align*}
		and \cite[Lemma \ref{I-lem:aexpXa}]{AppAGRCF} to obtain $m. a_\perp.o = a_\perp. m.o = a_\perp.o$. By \cite[Lemma \ref{I-lem:Jacobson_Morozov_oneparam}]{AppAGRCF} and Lemma \ref{lem:BTphiu_is_vt}, $(-v)(\chi_\alpha(a_t)) = 2(-v)(t) = 2\varphi_\alpha(u)$. A point $a.o = a_\perp a_\alpha.o \in \A$ is fixed by $m(u)$ if and only if
		$$
		a.o = m(u).a.o = a_t m a_\perp a_\alpha m^{-1} .o = a_t a_\perp a_\alpha^{-1}.o 
		$$
		which is the case exactly when $(-v)(\chi_\alpha(a_t a_\alpha^{-2})) =0$, i.e. when $(-v)(\chi_\alpha(a)) = (-v)(t) = \varphi_\alpha(u)$.
	\end{proof}
	We obtain three corollaries from the geometric description above. Recall from Theorem \ref{thm:stab_A} that $\operatorname{Stab}_{G_\F}(\A) = M_\F A_\F(O) = T_\F(O)$.
	
	\begin{lemma}\label{lem:BTphi_stab}
		For any $u_1,u_2 \in (U_\alpha)_\F$, 
		$$
		\varphi_\alpha(u_1) = \varphi_\alpha(u_2) \quad \iff \quad m(u_2)^{-1}m(u_1) \in T_\F(O).
		$$
		For any $u \in (U_\alpha)_\F, u' \in (U_{-\alpha})_\F$, 
		$$
		\varphi_\alpha(u) = -\varphi_{-\alpha}(u') \quad \iff \quad m(u)^{-1}m(u') \in T_\F(O).
		$$
	\end{lemma}
	\begin{proof}
		We use the description of the action in Proposition \ref{prop:BTmu_in_Wa}. Both $m(u_1)$ and $m(u_2)$ act on $\A$ by an affine reflection along a hyperplane. They reflect along the same hyperplane if and only if $\varphi_\alpha(u_1) = \varphi_\alpha(u_2)$, which is the case exactly when $m(u_2)^{-1}m(u_1) \in \operatorname{Stab}_{G_\F}(\A)$. 
		
		The fixed hyperplanes of $m(u)$ and $m(u')$ are given respectively by the conditions $(-v)(\chi_\alpha(a)) = \varphi_\alpha(u)$ and $(-v)(\chi_{-\alpha}(a)) = \varphi_{-\alpha}(u')$. The second condition can also be written as $(-v)(\chi_\alpha(a)) = -\varphi_{-\alpha}(u') = \varphi_\alpha(u)$ and hence agrees with the first.
	\end{proof}
	
	\begin{lemma}\label{lem:BTmMm_M}
		We have $m(u) \cdot M_\F \cdot m(u)^{-1} = M_\F$.
	\end{lemma}
	\begin{proof}
		Let $m(u)=a_t\cdot m \in A_\F \cdot N_\F$ as in Proposition \ref{prop:BTmu_in_Wa}. In the spherical Weyl group $W_s = N_\F/M_\F$ we have $mM_\F \cdot m^{-1}M_\F = M_\F$. Then 
		\begin{equation*}
				m(u) \cdot M_\F \cdot  m(u)^{-1} =  a_tm M_\F m^{-1}a_t^{-1} = a_t M_\F a_t^{-1} = M_\F. \qedhere
		\end{equation*}
	\end{proof}
	
	\begin{lemma}\label{lem:BTmUm}
		Let
			$a,b \in A_\F$. If $a.o = m(u).b.o$, then $$(-v)(\chi_{\alpha}(a)) = (-v)( t^2 \chi_{\alpha}(b)^{-1} ).$$
			Moreover,
			$$
			m(u)U_{\alpha,\ell}m(u)^{-1} = U_{-\alpha, \ell-2(-v)(t)} = U_{-\alpha, \ell-2\varphi_\alpha(u)}
			$$
			for any $\ell \in \Lambda$.
		\end{lemma}
		\begin{proof}
			Decomposing $m(u) =a_t \cdot m \in A_\F \cdot N_\F$ as in Proposition \ref{prop:BTmu_in_Wa}, we see $m(u)$ as consisting of a translational element $a_t$ and a representative of an element $w=[m]$ of the spherical Weyl group $W_s$. 
			
			If $a,b\in A_\F$ satisfy $a.o = m(u).b.o = a_t m b m^{-1}.o$, then
			\begin{align*}
				(-v)(\chi_\alpha(a)) &= (-v)(\chi_\alpha(a_t mbm^{-1})) 
				= (-v)(\chi_\alpha(a_t) \chi_\alpha(mbm^{-1})) \\
				&= (-v)(t^{2} \chi_\alpha(b)^{-1}).
			\end{align*}
			where the last equality comes from \cite[Lemma \ref{I-lem:Jacobson_Morozov_oneparam} and \ref{I-lem:Jacobson_Morozov_m}]{AppAGRCF}.
			
			As in the proof of Lemma \ref{lem:BTaUa}, we now consider $\Omega = \{b.o\}$ with $(-v)(\chi_\alpha(b)) = \ell$. Then $U_{\alpha,\ell} = U_{\alpha,\Omega}$. By \cite[Lemma \ref{I-lem:Jacobson_Morozov_m} and \ref{I-lem:aexpXa}]{AppAGRCF}, we have
			\begin{align*}
				m(u)\cdot (U_\alpha)_\F \cdot m(u)^{-1} & = a_t  m  (U_\alpha)_\F  m^{-1}  a_t^{-1} = a_t (U_{-\alpha})_\F a_t^{-1} = (U_{-\alpha})_\F.
			\end{align*}
			Since 
			$$	
			(m(u)\cdot u \cdot m(u)^{-1}).m(u).b.o = m(u).b.o,
			$$
			we have thus $m(u) \cdot U_{\alpha, \Omega} \cdot m(u)^{-1} = U_{-\alpha, m(u).\Omega}$. 
			Now by the comment after Lemma \ref{lem:BTUalphaOmega}, $U_{-\alpha, m(u).\Omega} = U_{-\alpha, \ell'}$ for $\ell' = (-v)(\chi_{-\alpha}(a))$ where $a.o = m(u).b.o$. We have
			$$
			\ell' = (-v)(\chi_{-\alpha}(a)) = -(-v)(\chi_\alpha(a)) = (-v)(\chi_\alpha(b)) - 2(-v)(t) = \ell -2(-v)(t)
			$$
			whence $m(u)U_{\alpha,\ell}m(u)^{-1} = U_{-\alpha,\ell-2(-v)(t)}$.
			%
		\end{proof}
		
		We know from Lemma \ref{lem:BTm_uuu}, that $m(u)\in (U_{-\alpha})_\F u (U_{-\alpha})_\F$. Next, we will show that $m(u)$ is the only element in $\operatorname{Nor}_{G_\F}(A_\F) \cap  (U_{-\alpha})_\F u (U_{-\alpha})_\F$.
		
		\begin{lemma}\label{lem:BTUUUm}
			Let 
			$u',u'' \in (U_{-\alpha})_\F$ such that $u'uu'' \in \operatorname{Nor}_{G_\F}(A_\F)$. Then $u'uu'' = m(u)$ and
			$$
			\varphi_{-\alpha}(u') = - \varphi_{\alpha}(u) = \varphi_{-\alpha}(u'').
			$$
		\end{lemma}
		\begin{proof}
			Recall that if
			$$
			u = \varphi_\F\begin{pmatrix}
				1 & t \\ 0 & 1
			\end{pmatrix}, \quad \text{and} \quad \overline{u} := \varphi_\F \begin{pmatrix}
				1 & 0 \\ 1/t & 1
			\end{pmatrix},
			$$
			then
			\begin{align*}
				m(u)& = \varphi_\F \left(\begin{pmatrix}
					1 & 0 \\ -1/t & 1
				\end{pmatrix}\begin{pmatrix}
					1 & t \\ 0 & 1
				\end{pmatrix}\begin{pmatrix}
					1 & 0 \\ -1/t & 1
				\end{pmatrix} \right) \\
				&= \overline{u}^{-1}u\overline{u}^{-1} \in (U_{-\alpha})_\F(U_{\alpha})_\F(U_{-\alpha})_\F.
			\end{align*}
			Then $u'uu'' = u'\overline{u} \cdot m(u) \cdot \overline{u} u'' \in \operatorname{Nor}_{G_\F}(A_\F)$ and $\overline{u}u'' \in (U_{-\alpha})_\F$. Now for any $\ell \leq -\varphi_{-\alpha}(\overline{u}u'')$, use \cite[Lemma \ref{I-lem:Jacobson_Morozov_oneparam}]{AppAGRCF} to obtain $a\in A_\F$ such that $(-v)(\chi_\alpha(a)) = \ell $ or equivalently $(-v)(\chi_{-\alpha}(a)) = -\ell$. By Lemma \ref{lem:BTUalphaOmega} this means that $\overline{u}u'' \in U_{-\alpha, -\ell} = U_{-\alpha, \{ a.o\}}$. Let $b \in A_\F$ be such that
			$$
			b.o = u'uu'' . a.o = u'\overline{u} \cdot m(u) \cdot \overline{u}u''.a.o = u'\overline{u} \cdot m(u) \cdot a.o
			$$ 
			and applying Proposition \ref{prop:UonA}, $b.o = m(u).a.o$. Since
			$$
			(-v)(\chi_{-\alpha}(b)) = (-v)(\chi_{\alpha}(a)) -2(-v)(t) = \ell - 2(-v)(t)  
			$$
			by Lemma \ref{lem:BTmUm}, we can use Lemma \ref{lem:BTUalphaOmega} to obtain $u'\overline{u} \in U_{-\alpha, \{ b.o\}} = U_{-\alpha, \ell - 2(-v)(t)}$ for all $\ell \leq -\varphi_{-\alpha}(\overline{u}u'')$. Therefore $\varphi_{-\alpha}(u'\overline{u}) \leq \lambda$ for all $\lambda \in \Lambda$. The only element of $(U_{-\alpha})_\F$ with this property is $u'\overline{u} = \operatorname{Id}$. 
			
			Now for all $a\in A_\F$,
			\begin{align*}
				u'uu''.a.o &= u'\overline{u} m(u) \overline{u}u''.a.o = m(u). \overline{u}u''.a.o \in \A 
			\end{align*}
			and thus $\overline{u}u''.a.o \in \A$, hence $\overline{u}u''.a.o = a.o$ by Proposition \ref{prop:UonA}. The only element of $(U_{-\alpha})_\F$ acting trivially on all of $\A$ is the identity, so $\overline{u}u'' = \operatorname{Id}$. Thus $u'uu'' = \overline{u}^{-1}u\overline{u}^{-1} = m(u)$. Moreover
			\begin{align*}
				\varphi_{-\alpha}(u')&=\varphi_{-\alpha}(u'') = \varphi_{-\alpha}(\overline{u}^{-1}) = (-v)(-1/t) = -(-v)(t) = -\varphi_{\alpha}(u). \qedhere
			\end{align*}	
		\end{proof}
		
		\begin{figure}[h]
			\centering
			\includegraphics[width=1\linewidth]{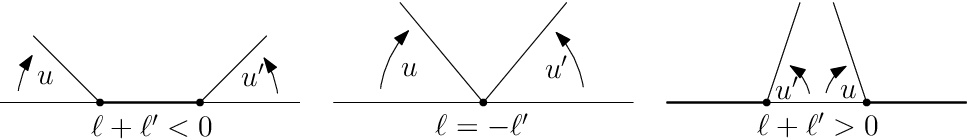}
			\caption{ Illustration of the action of elements $u\in (U_\alpha)_\F, u' \in (U_{-\alpha})_\F$ with $\ell := \varphi_{\alpha}(u) $ and $\ell' := \varphi_{-\alpha}(u')$ in the case where $\ell + \ell' < 0$ (left), $\ell = - \ell'$ (middle) and $\ell + \ell' > 0$. An element $m(u)$ of the affine Weyl group can only be generated by $u,u'$ in the case $\ell =-\ell'$. }
			\label{fig:valuationuu}
		\end{figure}
		
		The actions of $ u\in (U_\alpha)_\F, u' \in (U_{-\alpha})_\F$ depending on their root group valuations $\ell, \ell'$ are illustrated in Figure \ref{fig:valuationuu}. Only when $\ell + \ell' \leq 0$ do $u,u'$ fix a common point. To study the stabilizer we consider the cases $\ell + \ell' < 0$ and $\ell = -\ell'$ separately.
		\begin{lemma}\label{lem:BTuu_usu}
			Let 
			$u'\in (U_{-\alpha})_\F$ such that $\varphi_\alpha(u) + \varphi_{-\alpha}(u') < 0$. Then there is an element $(u_1, t , u_1') \in (U_{\alpha})_\F \times T_\F 
			\times (U_{-\alpha})_\F $ such that $u'u = u_1 t u_1'$. 
			
			Moreover $\varphi_\alpha(u) = \varphi_{\alpha}(u_1)$, $\varphi_{-\alpha}(u') = \varphi_{-\alpha}(u_1')$ and $t \in T_\F(O)$.
		\end{lemma}
		\begin{proof}
			If $u = \operatorname{Id}$ we can choose $u_1 = \operatorname{Id}, t=\operatorname{Id}$ and $u_1' = u'$. We may thus assume $u \neq \operatorname{Id}$. The element $u'u$ lies in the group $(L_{\pm \alpha})_\F$ defined in \cite[Section \ref{I-sec:rank1}]{AppAGRCF}. As $(L_{\pm \alpha})_\F$ is again a semisimple liear algebraic group, subgroups $(K_{\pm \alpha})_\F$, $(A_{\pm \alpha})_\F$ and $(M_{\pm \alpha})_\F$ can be identified as we did for $G_\F$ in Section \ref{sec:building_def}. By \cite[Corollary \ref{I-cor:levi_Bruhat_alternative}]{AppAGRCF},
			$$
			(L_{\pm \alpha})_\F = (B_{\alpha})_\F \cdot (B_{-\alpha})_\F  \ \amalg \   m \cdot (B_{ -\alpha})_\F,
			$$
			where $(B_{\alpha})_\F := (M_{\pm \alpha})_\F (A_{\pm \alpha})_\F (U_\alpha)_\F$ and $(B_{-\alpha})_\F := (M_{\pm \alpha})_\F (A_{\pm \alpha})_\F (U_{-\alpha})_\F$. We note that by \cite[Lemma \ref{I-lem:levi_fixes_A}]{AppAGRCF} and then Lemma \ref{lem:BT_T_is_MA} that $(M_{\pm \alpha})_\F(A_{\pm \alpha})_\F \subseteq M_\F \cdot A_\F = T_\F 
			$. 
			
			We have that $u'u \notin m \cdot (B_{-\alpha})_\F$, since otherwise there exists $u'' \in (U_{-\alpha})_\F$ such that
			$$
			u'uu'' \in m \cdot (M_{\pm \alpha})_\F (A_{\pm \alpha})_\F  \subseteq \operatorname{Nor}_{G_\F}(A_\F)
			$$
			and hence by Lemma \ref{lem:BTUUUm}, $-\varphi_\alpha(u) = \varphi_{-\alpha}(u')$ contradicting our assumption. Thus
			\begin{align*}
				u'u &\in (B_\alpha)_\F \cdot (B_{-\alpha})_\F = (U_{\alpha})_\F (A_{\pm \alpha})_\F (M_{\pm \alpha})_\F (A_{\pm \alpha})_\F (U_{-\alpha})_\F \\
				& \subseteq (U_{\alpha})_\F \cdot T_\F 
				\cdot (U_{-\alpha})_\F. 
			\end{align*}
			
			Let now $u_1 \in (U_\alpha)_\F,\ u_1' \in (U_{-\alpha})_\F$ and $ t \in (A_{\pm \alpha})_\F(M_{\pm \alpha})_\F (A_{\pm \alpha})_\F$ with $u'u = u_1 t u_1'$.
			If $u_1 = \Id$, $u=(u')^{-1} u_1 tu_1' \in (U_{-\alpha})_\F M_\F A_\F \cap U_{\alpha}$, but since $u$ can not send some halfapartment $H_{-\alpha,\ell}^+$ to a halfapartment $H_{-\alpha,\ell'}^+$ and at the same times fix a halfapartment $H_{\alpha,k}^+$, see Proposition \ref{prop:Ualphaconv} and Corollary \ref{cor:NalphaO_fixes_H_affine}, $u=\Id$ and we can argue as in the beginning of the proof. Thus we may assume that $u_1 \neq \operatorname{Id}$.			
			We now show $\varphi_\alpha(u)=\varphi_{\alpha}(u_1)$. Consider the element $m(u_1) \in \operatorname{Nor}_{G_\F}(A_\F)$. Then
			$$
			m(u_1) \in (U_{-\alpha})_\F u_1 (U_{-\alpha})_\F,
			$$
			so let $u_2', u_3' \in (U_{-\alpha})_\F$ such that $u_1 = u_2' m(u_1) u_3'$ and by Lemma \ref{lem:BTUUUm} $\varphi_{-\alpha}(u_2') = - \varphi_\alpha(u_1) = \varphi_{-\alpha}(u_3')$. Then
			\begin{align*}
				u&= (u')^{-1}u_1 t u_1' = (u')^{-1} u_2' m(u_1) u_3' t u_1' \\
				&= ((u')^{-1}u_2') m(u_1)t (t^{-1} u_3' t u_1') \in (U_{-\alpha})_\F m(u_1) t (U_{-\alpha})_\F 
			\end{align*}
			where we used that $t^{-1}u_3' t \in (U_{-\alpha})_\F$ by \cite[Proposition \ref{I-prop:anainN}]{AppAGRCF} and the fact that $(M_{\pm \alpha})_\F$ acts trivially on the rank 1 root system $\{\alpha, -\alpha\}$. Applying Lemma \ref{lem:BTUUUm} again gives $\varphi_{-\alpha}((u')^{-1}u_2') = - \varphi_\alpha(u) $. Now since $\varphi_\alpha(u) + \varphi_{-\alpha}(u') < 0$, we have 
			$$
			\varphi_{-\alpha}((u')^{-1}) = \varphi_{-\alpha}(u') < -\varphi_\alpha(u) = \varphi_{-\alpha}((u')^{-1}u_2') \leq \max\{ \varphi_{-\alpha}((u')^{-1}), \varphi_{-\alpha}(u_2') \}
			$$
			from which we can infer that $\varphi_{-\alpha}((u')^{-1}) \neq \varphi_{-\alpha}(u_2')$ and thus by the second part of Lemma \ref{lem:BTgroup_valuation}, $\varphi_{-\alpha}((u')^{-1}u_2') = \varphi_{-\alpha}(u_2')$, whence 
			$$
			\varphi_\alpha(u) = - \varphi_{-\alpha}((u')^{-1}u_2') =  -\varphi_{-\alpha}(u_2') = \varphi_\alpha (u_1).
			$$
			Proving $\varphi_{-\alpha}(u') = \varphi_{-\alpha}(u_1')$ works the same way, after replacing $u$ by $u'$, $u_1$ by $u_1'$ and $\alpha$ by $-\alpha$.
			
			From the above discussion, we obtain
			$$
			m(u_1)t = ((u_2')^{-1}u') u ((u_1')^{-1}t^{-1}(u_3')^{-1}t) \in (U_{-\alpha})_\F u (U_{-\alpha})_\F
			$$
			which by the uniqueness statement of Lemma \ref{lem:BTUUUm} implies $m(u_1)t = m(u)$. Since $\varphi_\alpha(u) = \varphi_\alpha(u_1)$, we know by Proposition \ref{prop:BTmu_in_Wa} that both $m(u_1)$ and $m(u)$ represent the same element of the affine Weyl group: the reflection $r_\alpha$ followed by a translation of length $\varphi_\alpha(u) = \varphi_\alpha(u_1)$. This means that $[t] = [m(u_1)^{-1}][m(u)] = [\operatorname{Id}] \in W_a$, in particular $t$ fixes $o$, so $t \in T_\F(O)$.
		\end{proof}
		
		\begin{proposition}\label{prop:BTL_l+l'<0} 
			Let 
			$\ell, \ell' \in \Lambda$ such that $\ell + \ell' <0$. Let 
			$$
			L_{\ell, \ell'} := \langle u \in U_{\alpha, \ell} \cup U_{-\alpha, \ell'} \rangle 
			$$
			and $H_{\ell, \ell'} := L_{\ell,\ell'} \cap T_\F(O) $. Then
			$$
			L_{\ell,\ell'} = U_{\alpha, \ell} \cdot U_{-\alpha, \ell'} \cdot H_{\ell,\ell'} .
			$$
		\end{proposition}
		\begin{proof}
			The inclusion $\supseteq$ is clear. We show the other direction by induction on the word length. If the length of a word in $L_{\ell, \ell'}$ is $1$, then it lies in $U_{\alpha, \ell} \cdot U_{-\alpha, \ell'} \cdot H_{\ell,\ell'}$. It remains to show that for every $u_{\pm \alpha} \in U_{\alpha, \ell} \cup U_{-\alpha,\ell'}$, and $v = u \cdot u' \cdot h \in U_{\alpha, \ell} \cdot U_{-\alpha, \ell'} \cdot H_{\ell,\ell'} $, we have $u_{\pm \alpha} v \in U_{\alpha, \ell} \cdot U_{-\alpha, \ell'} \cdot H_{\ell,\ell'}$. If $u_\alpha \in U_{\alpha,\ell}$, this is straightforward, so let $u_{-\alpha} \in U_{-\alpha,\ell'}$. By Lemma \ref{lem:BTuu_usu}, we can find $(u_1,t_1,u_1') \in U_{\alpha,\ell}\times T_\F(O) \times U_{-\alpha,\ell'}$ such that $u_{-\alpha}u = u_1t_1u_1'$. Then $u_{-\alpha}v = u_1 t_1 u_1' u' h$. By Lemma \ref{lem:BTtUt}, $t_1U_{-\alpha,\ell'} t_1^{-1} = U_{-\alpha,\ell'}$, so $u_{-\alpha}v \in U_{\alpha,\ell}\cdot U_{-\alpha,\ell'} \cdot H_{\ell,\ell'}$ as required.
		\end{proof}
		Next up, we consider the situation $\ell + \ell' = 0$. For this we introduce
		$$
		U_{-\alpha, -\ell+} := \bigcup_{t < -\ell} U_{-\alpha, t} \subseteq U_{-\alpha, -\ell}
		$$
		which is a group by Lemma \ref{lem:BTgroup_valuation}.
		
		\begin{proposition}\label{prop:BTL_l+l'=0} 
			Let 
			$\ell := \varphi_\alpha(u)$ and let $m(u)$ as in \cite[Lemma \ref{I-lem:Jacobson_Morozov_m}]{AppAGRCF}. Let 
			$$
			L_{\ell,- \ell} := \langle \tilde{u} \in U_{\alpha, \ell} \cup U_{-\alpha, -\ell} \rangle 
			$$
			and $H_{\ell, -\ell} := L_{\ell,-\ell} \cap T_\F(O) $. Then
			$$
			L_{\ell,-\ell} = (U_{\alpha,\ell}\cdot H_{\ell, -\ell} \cdot U_{-\alpha,-\ell+}) \cup (U_{\alpha,\ell} \cdot m(u) \cdot H_{\ell,-\ell} \cdot U_{\alpha,\ell}).
			$$
		\end{proposition}
		\begin{proof}
			By Lemma \ref{lem:BTm_uuu}, $m(u) \in U_{-\alpha,-\ell} U_{\alpha, \ell} U_{-\alpha,-\ell}$ and thus
			$$
			B:= (U_{\alpha,\ell}\cdot H_{\ell, -\ell} \cdot U_{-\alpha,-\ell+}) \cup (U_{\alpha,\ell} \cdot m(u) \cdot H_{\ell,-\ell} \cdot U_{\alpha,\ell}) \subseteq L_{\ell,-\ell}.
			$$
			By Lemma \ref{lem:BTmUm}, $m(u) \cdot U_{\alpha, \ell}\cdot m(u)^{-1} = U_{-\alpha, -\ell}$, which shows that $L_{\ell,-\ell}$ is generated by $U_{\alpha,\ell}$ and $m(u)$. Thus to see $L_{\ell,-\ell} \subseteq B$, it suffices to show that $B$ is a group.
			By Proposition \ref{prop:BTL_l+l'<0}, $B\cdot U_{\alpha, \ell} = B$ and $B\cdot U_{-\alpha,-\ell+} = B$. By Lemma \ref{lem:BTtUt}, $B\cdot H_{\ell,-\ell} = B$. It remains to show that $B \cdot m(u) = B$. 
			
			Since $m(u)\cdot U_{\alpha,\ell}\cdot m(u)^{-1} = U_{-\alpha,-\ell}$ and by Lemma \ref{lem:BTmMm_M}, $m(u)\cdot T_\F(O) \cdot m(u)^{-1} = T_\F(O)$, we have
			$$
			(U_{\alpha,\ell} \cdot H_{\ell,-\ell}\cdot U_{-\alpha,-\ell})m(u) = U_{\alpha, \ell} \cdot m(u) \cdot  H_{\ell, -\ell} \cdot U_{\alpha,\ell} \subseteq B
			$$
			and 
			\begin{align*}
				\left( U_{\alpha,\ell}\cdot m(u) \cdot H_{\ell,-\ell} \cdot U_{\alpha,\ell} \right)  m(u)
				= U_{\alpha, \ell} \cdot  H_{\ell, -\ell} \cdot U_{-\alpha,-\ell}\cdot m(u)^2.
			\end{align*}
			By Proposition \ref{prop:BTmu_in_Wa}, $m(u)^{2}$ acts trivially on $\A$, 
			hence it lies in $ \operatorname{Cen}_{G_\F}(A_\F) \cap G_\F(O) = T_\F(O)$, so
			$$
			\left( U_{\alpha,\ell}\cdot m(u) \cdot H_{\ell,-\ell} \cdot U_{\alpha,\ell} \right)  m(u)
			= U_{\alpha, \ell} \cdot  H_{\ell, -\ell} \cdot U_{-\alpha,-\ell}.
			$$
			If $u'\in (U_{-\alpha,-\ell})_\F$ satisfies $\varphi_{-\alpha}(u') = -\ell$, then by Lemma \ref{lem:BTphi_stab}, $m(u)^{-1}m(u') \in T_\F(O) 
			$, so $m(u') \in m(u)\cdot H_{\ell,-\ell}$ since $m(u), m(u') \in L_{\ell,-\ell}$. Thus
			$$
			u' \in U_{\alpha,\ell}\cdot m(u')\cdot U_{\alpha, \ell} \subseteq U_{\alpha,\ell}\cdot  m(u) \cdot H_{\ell,-\ell} \cdot U_{\alpha,\ell}.
			$$
			and
			$$
			U_{-\alpha, -\ell} \subseteq U_{-\alpha,-\ell+} \cup U_{\alpha,\ell}\cdot m(u) \cdot H_{\ell,-\ell} \cdot U_{\alpha,\ell}.
			$$
			This shows 
			\begin{align*}
				&\left( U_{\alpha,\ell}\cdot m(u) \cdot H_{\ell,-\ell} \cdot U_{\alpha,\ell} \right)  m(u)\\
				&\subseteq \left( U_{\alpha, \ell} \cdot  H_{\ell, -\ell} \cdot U_{-\alpha,-\ell + } \right) \cup \left(U_{\alpha, \ell} \cdot  H_{\ell, -\ell} \cdot U_{\alpha,\ell } \cdot m(u)\cdot H_{\ell ,-\ell} \cdot U_{\alpha,\ell}\right) \\
				&= B.
			\end{align*}
			Similarly, one can show $B\cdot m(u)^{-1} \subseteq B$, which concludes the proof of $B= L_{\ell,-\ell}$.
		\end{proof}
		
		\begin{proposition}\label{prop:BTL_l+l'}  
			Let $\Omega \subseteq \A$ be a non-empty finite subset and 
			$$\ell = \min_{a.o \in \Omega} \{ (-v)(\chi_\alpha(a)) \}, \quad \ell' = \min_{a.o \in \Omega} \{ (-v)(\chi_{-\alpha}(a)) \}. 
			$$
			Let 
			$
			L_{\ell,\ell'} := \langle u \in U_{\alpha, \ell} \cup U_{-\alpha, \ell'} \rangle 
			$
			and $N_{\ell, \ell'} := L_{\ell,\ell'} \cap \operatorname{Nor}_{G_\F}(A_\F) $. Then
			$$
			L_{\ell,\ell'} = U_{\alpha, \ell} \cdot U_{-\alpha, \ell'} \cdot N_{\ell,\ell'} .
			$$
		\end{proposition}
		\begin{proof}
			Since $\Omega \neq \emptyset$, we have $\ell + \ell' \leq 0$. If $\ell+\ell' <0$ we can apply Proposition \ref{prop:BTL_l+l'<0} and are done. Otherwise, we have 
			$$
			L_{\ell,-\ell} = (U_{\alpha,\ell}\cdot H_{\ell, -\ell} \cdot U_{-\alpha,-\ell+}) \cup (U_{\alpha,\ell} \cdot m(u) \cdot H_{\ell,-\ell} \cdot U_{\alpha,\ell})
			$$
			by Proposition \ref{prop:BTL_l+l'=0}. As $H_{\ell,\ell'} \subseteq T_\F(O)$, we can apply Lemma \ref{lem:BTtUt} to obtain $H_{\ell,\ell'}U_{\alpha,\ell} = U_{\alpha,\ell}H_{\ell,\ell'}$ and $H_{\ell,\ell'}U_{-\alpha,\ell'} = U_{-\alpha,\ell'} H_{\ell,\ell'}$. Finally, Lemma \ref{lem:BTmUm} gives $m(u) U_{\alpha, \ell} m(u)^{-1} = U_{-\alpha, \ell- 2 \varphi_\alpha(u)} = U_{-\alpha , -\ell} = U_{-\alpha, \ell'}$.
		\end{proof}
		
		
		We will prove a 'mixed Iwasawa'-decomposition for the rank one subgroup $(L_{\pm \alpha})_\F$ defined in \cite[Section \ref{I-sec:rank1}]{AppAGRCF}. The compact group $K_\F$ in the Iwasawa decomposition is replaced by a subgroup of $\operatorname{Stab}_{(L_{\pm \alpha})_\F}(o)$ and an element $m$ representing the non-trivial element of the spherical Weyl group $W_{\pm \alpha} = \{[\operatorname{Id}],[m]\}$ of $(L_{\pm \alpha})_\F$.
		\begin{proposition}\label{prop:BT_mixed_Iwasawa_rank_1}
			Let $L_o = \langle U_{\alpha,o} , U_{-\alpha, o} \rangle$. Then 
			\begin{align*}
				L_{\pm \alpha} &= (U_\alpha)_\F \cdot (\operatorname{Nor}_{G_\F}(A_\F) \cap (L_{\pm \alpha})_\F)\cdot L_o \\
				=\ & (U_\alpha)_\F \cdot (\operatorname{Cen}_{G_\F}(A_\F) \cap (L_{\pm \alpha})_\F)\cdot L_o  \ \amalg \ (U_\alpha)_\F \cdot (\operatorname{Cen}_{G_\F}(A_\F) \cap (L_{\pm \alpha})_\F)\cdot m \cdot L_o.
			\end{align*}
		\end{proposition}
		\begin{proof}
			It is clear that the inclusion $\supseteq$ holds. For the reverse, we use \cite[Corollary \ref{I-cor:levi_Bruhat}]{AppAGRCF} to write
			$$
			(L_{\pm \alpha})_\F = (B_\alpha)_\F  \ \amalg \ (B_\alpha)_\F m (B_\alpha)_\F
			$$
			where $m$ is a representative of the non-trivial element in the spherical Weyl group $W_{\pm\alpha}$ associated with $(L_{\pm \alpha})_\F$ and $(B_\alpha)_\F = (M_{\pm \alpha})_\F (A_{\pm \alpha})_\F (U_\alpha)_\F$ as in \cite[Section \ref{I-sec:rank1}]{AppAGRCF}. If $g \in (B_\alpha)_\F$, then $g \in U_\alpha (M_{\pm \alpha})_\F (A_{\pm \alpha})_\F \subseteq (U_\alpha)_\F \cdot (\operatorname{Nor}_{G_\F}(A_\F)\cap (L_{\pm\alpha})_\F) $. We claim that
			$$
			m(B_\alpha)_\F = m (M_{\pm\alpha})_\F (A_{\pm \alpha})_\F (U_{\alpha})_\F \subseteq (U_{\alpha})_\F \cdot (\operatorname{Nor}_{G_\F}(A_\F) \cap (L_{\pm \alpha})_\F)\cdot L_o.
			$$
			Assuming the claim we obtain
			\begin{align*}
				(B_\alpha)_\F m (B_\alpha)_\F &\subseteq (B_\alpha)_\F (U_\alpha)_\F \cdot (\operatorname{Nor}_{G_\F}(A_\F) \cap (L_{\pm \alpha})_\F)\cdot L_o \\
				&\subseteq (U_\alpha)_\F \cdot (\operatorname{Nor}_{G_\F}(A_\F) \cap (L_{\pm \alpha})_\F)\cdot L_o
			\end{align*}
			since $(M_{\pm\alpha})_\F$ and $(A_{\pm\alpha})$ normalize $(U_\alpha)_\F$. To prove the claim, we consider an element in $m(M_{\pm \alpha})_\F (A_{\pm \alpha})_\F u$ for $u \in (U_\alpha)_\F$. If $\varphi_\alpha(u)\leq 0$, then $m(M_{\pm \alpha})_\F (A_{\pm \alpha})_\F u \subseteq (U_\alpha)_\F (\operatorname{Nor}_{G_\F}(A_\F) \cap (L_{\pm \alpha})_\F) L_o$ and we are done. If on the other hand $\varphi_\alpha(u) > 0$, then there are $u',u'' \in (U_{-\alpha})_\F$ with $u = u' \cdot m(u) \cdot u''$ and  $\varphi_{-\alpha}(u') = \varphi_{-\alpha}(u'') = -\varphi_\alpha(u) <0$ by Lemma \ref{lem:BTm_uuu}. In particular $u'' \in L_o$ and
			\begin{align*}
				m(M_{\pm \alpha})_\F (A_{\pm \alpha})_\F u &= m(M_{\pm \alpha})_\F (A_{\pm \alpha})_\F u' \cdot m(u) \cdot u'' \\
				&= m\tilde{u}'m^{-1} \cdot m(M_{\pm \alpha})_\F (A_{\pm \alpha})_\F m(u) \cdot u'' \\
				&\subseteq (U_\alpha)_\F \cdot (\operatorname{Nor}_{G_\F}(A_\F) \cap (L_{\pm \alpha})_\F ) \cdot L_o,
			\end{align*}
			where $\tilde{u}' \in (U_{-\alpha})_\F$ is a $(M_{\pm \alpha})_\F (A_{\pm \alpha})_\F$-conjugate of $u'$.
		\end{proof}
		The result of Proposition \ref{prop:BTL_l+l'} still holds, when $\Omega =\emptyset$.
		\begin{lemma}\label{lem:BTL_l+l'=-infty}
			Let $\alpha \in \Sigma$, $L = \langle (U_{\alpha})_\F , (U_{-\alpha})_\F \rangle$ and 
			$N = \operatorname{Nor}_{G_\F}(A_\F) \cap L$. Then
			$$
			L = (U_{\alpha})_\F(U_{-\alpha})_\F N.
			$$
		\end{lemma}
		\begin{proof}
			Clearly $\supseteq$. We consider the word length of an element $g \in L$. If $g$ consists of at most two letters, it is clearly part of $(U_{\alpha})_\F(U_{-\alpha})_\F N$, unless $g = u'u$ for $u'\in (U_{-\alpha})_\F$ and $u\in (U_\alpha)_\F$. In this case, we write $u' = u'' m(u) u'''$ for $u'',u''' \in (U_{\alpha})_\F$. Then $\bar{u} := m(u) u''' u m(u)^{-1} \in (U_{-\alpha})_\F$ by Lemma \ref{lem:BTmUm} and we have $g = u'' \bar{u} m(u) \in (U_{\alpha})_\F(U_{-\alpha})_\F N$.
			
			If $g$ contains at least three letters, write $g = u_n u_{n-1} \ldots u_{2} u_1$, where $u_i$ are elements in $(U_{\alpha})_\F, (U_{-\alpha})_\F$ alternating. Write $u_2 = u_2' m(u_2) u_2''$ with $u_2',u_2'' \in  (U_{-\beta})_\F$, where $u_2\in (U_{\beta})_\F$ for $\beta \in \{\pm \alpha\}$. Using Lemma \ref{lem:BTmUm} again we obtain $\bar{u} := m(u_2) u_2'' u_1 m(u_2)^{-1} \in (U_{\beta})_\F $, and
			$g = u_n \ldots u_4 (u_3 u_2') \bar{u} m(u)$, so $g \in h N $ for some $h \in L$ with a smaller word length. Applying induction gives the result.
		\end{proof}

		The following description of the stabilizer of a subset $\Omega \subseteq \A$ under the action of the rank 1 subgroup $L$ will not be used later, but is included for completeness. 
		
		\begin{proposition}\label{prop:stab_L_Omega}
			Let $\alpha \in \Sigma$, $\Omega \subseteq \A$ non-empty finite, $L = \langle (U_{\alpha})_\F , (U_{-\alpha})_\F \rangle$, $N  = \operatorname{Nor}_{G_\F}(A_\F) \cap L$ and $N_\Omega = \{g \in N \colon n.p = p \text{ for all } p \in \Omega\}$. Then
			$$
			\operatorname{Stab}_{L} (\Omega) = \langle U_{\alpha,\Omega} , U_{-\alpha,\Omega} , N_\Omega \rangle =  U_{\alpha, \Omega} U_{-\alpha,\Omega} N_\Omega.
			$$
		\end{proposition}
		\begin{proof}
			The inclusions $\supseteq$ are clear. Let 
			$$
			\ell = \min_{a.o \in \Omega} \{ (-v)(\chi_\alpha(a)) \} \quad \text{and} \quad \ell' = \min_{a.o \in \Omega} \{ (-v)(\chi_{-\alpha}(a)) \}.
			$$
			Then $\ell \leq -\ell'$ since $\Omega \neq \emptyset$, $U_{\alpha,\ell} = U_{\alpha,\Omega}$ and $U_{-\alpha, \ell'} = U_{-\alpha,\Omega}$. Now if $g\in \operatorname{Stab}_L(\Omega)$, then there are $u\in (U_\alpha)_\F , u' \in (U_{-\alpha})_\F , n \in N$ such that $g.p = uu'n.p = p$ for all $p\in \Omega$, by Lemma \ref{lem:BTL_l+l'=-infty}. We distinguish three cases for $\varphi_\alpha(u)$. If $\varphi_{\alpha}(u) \leq \ell$, then $u \in U_{\alpha,\ell}$ and $u'n.p = u^{-1}.p = p$ for all $p \in \Omega$, so $u' \in U_{-\alpha, \ell'}$ and $n\in N_\Omega$ by Proposition \ref{prop:UonA}, so $g \in U_{\alpha,\ell} U_{-\alpha,\ell'} N_{\Omega}$.
			
			If $\varphi_\alpha(u) \geq -\ell'$, then we write $u = u'' m(u) u''' $ for some $u'',u''' \in (U_{-\alpha})_\F $ with $\varphi_{-\alpha}(u'') = \varphi_{-\alpha}(u''') = -\varphi_\alpha(u) \leq \ell'$. Abbreviating $\bar{u}: = m(u) u'''u' m(u)^{-1} \in (U_{\alpha})_\F $ we obtain
			$g = uu'n = u'' \bar{u} m(u) n \in U_{-\alpha,\ell'} \cdot (U_{\alpha})_\F  \cdot N$ by Lemma \ref{lem:BTmUm}. Then $\bar{u}m(u)n.p = p$ for all $p \in \Omega$, whence $\varphi_{\alpha}(\bar{u}) \leq \ell$ by Proposition \ref{prop:UonA} and then $m(u)n.p = p$ for all $p\in \Omega$, so $m(u)n \in N_\Omega$. Then $g \in U_{\alpha,\ell}  U_{-\alpha,\ell'} N_\Omega$.
			
			We claim that the final case $\ell < \varphi_\alpha(u) < -\ell'$ cannot happen. For this we first notice that for $p\in \Omega$ we have $u.p = p$ if and only if $u'n.p = p$, if and only if $u'.p = p$. Now let $a_1,a_2 \in A_\F$ with $(-v)(\chi_\alpha(a_1)) = \ell$ and $(-v)(\chi_{\alpha}(a_2)) = -\ell'$. Now $u.a_1.o \neq a_1.o$, so $u'.a_1.o \neq a_1.o$, so $\varphi_{-\alpha}(u') > -\ell$, but $u.a_2.o = a_2.o$, so $u'.a_2.o = a_2.o$, so $\varphi_{-\alpha}(u')\leq \ell'$. But this is not possible, since then $-\ell \leq \varphi_{-\alpha}(u') \leq \ell'$ contradicting $\ell < \ell'$.
			We have shown that $\operatorname{Stab}_{L}(\Omega) \subseteq U_{\alpha,\ell} U_{-\alpha,\ell'} N_{\Omega}$.  
		\end{proof}
		
\subsection{Bruhat-Tits theory in higher rank} \label{sec:BT_higher_rank}
		
		In this section, we continue to assume that $\Sigma$ is a reduced root system. Let $>$ be an order on $\Sigma$ and $\Omega \subseteq \A$. Now for any subset $\Theta \subseteq \Sigma_{>0}$ closed under addition, let
		$$
		U_{\Theta,\Omega} :=\left\{ u \in \exp\left( \bigoplus_{\alpha \in \Theta} (\frakg_\alpha)_\F \right) \colon  u.p=p \text{ for all } p \in \Omega \right\} 
		$$
		and in particular
		$$
		U_{\Omega}^+ := U_{\Sigma_{>0}, \Omega} = \left\{ u \in U_{\F} \colon u.p=p \text{ for all } p\in \Omega\right\}
		$$
		for $\Theta = \Sigma_{>0}$. Note that $U_{\emptyset}^+ = U_\F$. Analogously we define
		$$
		U_{\Omega}^- := \left\{ u \in \exp \left( \bigoplus_{\alpha <0} (\frakg_\alpha)_\F \right) \colon  u.p=p \text{ for all } p\in \Omega \right\}.
		$$
		Let
		$$
		N_\Omega : = \left\{ n \in \operatorname{Nor}_{G_\F}(A_\F) \colon n.p = p \text{ for all } p \in \Omega \right\}.
		$$
		When $\Omega = \{p\}$ consists of a single point, we abbreviate the notation by omitting the brackets such as in 
		$U_{p}^+ := U_{\{p\}}^+$, $U_{p}^- := U_{\{p\}}^-$ and $N_p := N_{\{p\}}$. The goal of this subsection is to prove in Theorem \ref{thm:BTstab_fin} that if $\Omega$ is finite, then the pointwise stabilizer of $\Omega$ satisfies
		$$
		\operatorname{Stab}_{G_\F}(\Omega) = U_{\Omega}^+ U_\Omega^- N_\Omega.
		$$
		\begin{proposition}\label{prop:BTUOmega}
			For any subset $\Omega \subseteq \A$ and subset $\Theta \subseteq \Sigma_{>0}$ closed under addition,
			$
			U_{\Theta, \Omega} = \left\langle u \in U_{\alpha,\Omega} \colon \alpha \in \Theta \right\rangle
			$
			and in particular, 
			$
			U_\Omega^+ = \langle  u \in U_{\alpha, \Omega} \colon  \alpha >0 \rangle.
			$
			More precisely, if $\Sigma_{>0} = \left\{ \alpha_1, \ldots , \alpha_r \right\}$ with $\alpha_1 < \ldots < \alpha_r$, then the product map
			\begin{align*}
				\prod_{\alpha \in \Sigma_{>0}} U_{\alpha, \Omega} & \to U_\Omega^+ \\
				(u_{\alpha_1}, \ldots , u_{\alpha_r}) & \mapsto u_{\alpha_1} \cdots u_{\alpha_r}
			\end{align*}
			is a bijection.
		\end{proposition}
		\begin{proof}
			For any $\Theta = \{\alpha_1, \ldots, \alpha_r\} \subseteq \Sigma_{>0}$ closed under addition with $\alpha_1 > \ldots > \alpha_k$, the image of $\prod_{\alpha \in \Theta} U_{\alpha,\Omega} $ under the product map is contained in $U_{\Theta,\Omega}$. On the other hand, if we start with $u \in U_{\Theta, \Omega}$, we can apply \cite[Lemma \ref{I-lem:BCH_consequence}]{AppAGRCF} to obtain a unique $(u_1, \cdots , u_k) \in (U_{\alpha_1}) \times \ldots (U_{\alpha_k})_\F $ with $u = u_1 \cdot \ldots \cdot u_k$. A point $a.o \in \Omega$ is fixed by $u$ if and only if $a^{-1}ua \in U_\F(O)$, so
			$$
			a^{-1}u_1a \cdot \ldots \cdot a^{-1}u_k a \in U_\F(O).
			$$
			We can apply Lemma \ref{lem:nn'} repeatedly to obtain $a^{-1}u_1a \in U_\F(O)$, \ldots, $a^{-1}u_ka \in U_\F(O)$ for all $a.o \in \Omega$. This exactly means that $u_i \in U_{\alpha,\Omega}$. This shows that the product map is surjective. Note that the order in the statement of the Proposition is the inverse of the one used in the proof, but one follows from the other by applying the inverse.
		\end{proof}
		We notice that the previous propositions hold for any chosen order on $\Sigma$. We may in particular invert the order to obtain
		\begin{align*}
			U_\Omega^- = \langle u \in U_{\alpha,\Omega} \colon \alpha < 0 \rangle 
		\end{align*}
		from Proposition \ref{prop:BTUOmega}. 
		We now define the groups
		$$
		P_\Omega := \langle  U_{\alpha,\Omega} \colon \alpha \in \Sigma \rangle \quad \text{ and } \quad 
		\hat{P}_\Omega := \langle N_\Omega, U_{\alpha,\Omega}\colon \alpha \in \Sigma \rangle
		$$
		generated by all elements of $U_{\alpha, \Omega}$, not just those with positive $\alpha$. The following is a generalization of Lemma \ref{prop:BTL_l+l'} to higher rank.
		\begin{proposition}\label{prop:BTPOmegaUUN} 
			Let $\Omega \subseteq \A$ finite. Then $P_\Omega$ and $\hat{P}_\Omega$ decompose as
			\begin{align*}
				P_\Omega &= U_\Omega^- \cdot U_\Omega^+ \cdot(\operatorname{Nor}_{G_\F}(A_\F)\cap P_{\Omega}) \\
				\hat{P}_\Omega &= U_\Omega^- \cdot U_\Omega^+ \cdot N_{\Omega}.
			\end{align*}
		\end{proposition}
		\begin{proof}
			The inclusions $\supseteq$ are clear. Recall that the set $\Sigma_{>0}$ is determined by a basis $\Delta \subseteq \Sigma$, or equivalently a chamber. After choosing a chamber, an ordering $>$ of $\Sigma$ can be obtained by choosing an order $\alpha_1 < \alpha_2 < \ldots < \alpha_r$ on $\Delta = \{ \alpha_1, \alpha_2, \ldots , \alpha_r \}$; the ordering on $\Sigma$ is then the lexicographical ordering \cite[VI.1.6, p. 174-175]{Bou08}. 
			
			For the $\subseteq$ direction of the description of $P_\Omega$, we will first show that 
			$$
			U_\Omega^- \cdot U_\Omega^+ \cdot(\operatorname{Nor}_{G_\F}(A_\F)\cap P_{\Omega})
			$$
			is independent of the chamber defining $\Sigma_{>0}$.
			Let $<_1, <_2$ be two orderings on $\Sigma$ whose chambers $C_1, C_2$ are related by a reflection determined by a simple root $\alpha \in \Delta_1 \subseteq \Sigma$, so $\Delta_2 = r_\alpha(\Delta_1)$. We may then assume that $<_1$ and $<_2$ are determined by lexicographical orderings such that $0 <_1 \alpha <_1 \beta$ for all $\beta \in \Delta_{1} \setminus \{\alpha\}$ and $0 <_2 -\alpha <_2 \beta$ for all $\beta \in \Delta_2\setminus \{\alpha\}$. We notice that then $\Sigma_{>_1 0} \setminus \{\alpha\} = \Sigma_{>_2 0} \setminus \{-\alpha\}$, since 
			for $\beta \in \Sigma_{>_1 0}\setminus \{\alpha\}$ there are $\lambda_{\delta} \in \Z_{\geq 0}$, at least one of which is strictly positive for some $\delta \in \Delta_1 \setminus \{\alpha\}$, such that
			\begin{align*}
				\beta &= \sum_{\delta\in \Delta_1} \lambda_\delta \delta 
				= \sum_{\delta \in \Delta_1} \lambda_\delta \left(r_\alpha(\delta) - 2\frac{\langle \delta, \alpha\rangle}{\langle \alpha, \alpha\rangle} r_\alpha(\alpha)\right) \\
				&= \sum_{\delta \in \Delta_1 \setminus \{ \alpha\}} \lambda_\delta r_\alpha(\delta) + \left(\sum_{\delta \in \Delta_1 \setminus \{\alpha\}} - 2 \frac{\langle\delta,\alpha\rangle}{\langle \alpha, \alpha \rangle} \lambda_\delta - \lambda_\alpha \right) r_\alpha(\alpha)  \\
				& \in \sum_{\delta \in \Delta_1\setminus \{\alpha\}} \mathbb{Z}_{\geq 0} r_\alpha(\delta) + \mathbb{Z} r_\alpha(\alpha) = \sum_{\delta \in \Delta_2\setminus \{-\alpha \}} \mathbb{Z}_{\geq 0} \delta + \mathbb{Z} (-\alpha) ,
			\end{align*}
			where we first used that $\langle \delta, \alpha \rangle$ is nonpositive for all $\delta \in \Delta_1 \setminus \{\alpha\}$. Bases of a root system have the property that any element of the root system written in that basis has either all non-negative or all non-positive coefficients. Since there is a strictly positive coefficient for the element $\beta$ written in the basis $\Delta_2$ as above, all coefficients have to be non-negative, so $\beta \in \Sigma_{>_20} \setminus \{-\alpha\}$.  
			It now suffices to show that 
			$$
			U_\Omega^- \cdot U_\Omega^+ \cdot(\operatorname{Nor}_{G_\F}(A_\F)\cap P_{\Omega})
			$$ is the same for $<_1$ and $<_2$ to deduce that it is independent of the chamber used to define $\Sigma_{>0}$, since the reflections determined by the simple roots $\Delta$ generate the spherical Weyl group which acts transitively on the chambers.
			
			We abbreviate $Y := (\operatorname{Nor}_{G_\F}(A_\F)\cap P_{\Omega})$. We use Proposition \ref{prop:BTUOmega} repeatedly to obtain 
			\begin{align*}
				U_{\Sigma_{<_1 0},\Omega} \cdot U_{\Sigma_{>_1 0},\Omega} \cdot Y 
				&= \prod_{\beta <_1 0} U_{\beta, \Omega} \cdot \prod_{\beta >_1 0} U_{\beta, \Omega} \cdot Y \\
				&= \prod_{\substack{\beta <_1 0\\ -\alpha \neq \beta}} U_{\beta, \Omega} \cdot  U_{-\alpha, \Omega} \cdot \prod_{\substack{\beta >_1 0\\ \alpha \neq \beta}} U_{\beta, \Omega}  \cdot U_{\alpha, \Omega} \cdot Y \\
				&= \prod_{\substack{\beta <_1 0\\ -\alpha \neq \beta}} U_{\beta, \Omega} \cdot U_{\Sigma_{>_2 0},\Omega}  \cdot U_{\alpha, \Omega} \cdot Y \\
				&= \prod_{\substack{\beta <_1 0\\ -\alpha \neq \beta}} U_{\beta, \Omega} \cdot  \prod_{\substack{\beta >_1 0\\ \alpha \neq \beta}} U_{\beta, \Omega} \cdot U_{-\alpha, \Omega} \cdot U_{\alpha, \Omega} \cdot Y 
			\end{align*}
			at which point we invoke Lemma \ref{prop:BTL_l+l'} to continue
			\begin{align*}
				&= \prod_{\substack{\beta <_1 0\\ -\alpha \neq \beta}} U_{\beta, \Omega} \cdot  \prod_{\substack{\beta >_1 0\\ \alpha \neq \beta}} U_{\beta, \Omega} \cdot U_{\alpha, \Omega} \cdot U_{-\alpha, \Omega} \cdot Y \\
				&= \prod_{\substack{\beta <_1 0\\ -\alpha \neq \beta}} U_{\beta, \Omega} \cdot  \prod_{\beta >_1 0} U_{\beta, \Omega} \cdot U_{-\alpha, \Omega} \cdot Y \\
				&= \prod_{\substack{\beta <_1 0\\ -\alpha \neq \beta}} U_{\beta, \Omega} \cdot  U_{\alpha, \Omega}  \cdot \prod_{\substack{\beta >_1 0\\ \alpha \neq \beta}} U_{\beta, \Omega} \cdot  U_{-\alpha, \Omega} \cdot Y \\
				&=  \prod_{\beta <_2 0} U_{\beta, \Omega}  \cdot \prod_{\beta >_2 0} U_{\beta, \Omega} \cdot Y \\
				&= U_{\Sigma_{<_2 0}, \Omega} \cdot U_{\Sigma_{>_2 0},\Omega} \cdot Y.
			\end{align*}
			Now that we have shown that it is independent of the chamber defining the order on $\Sigma$, we show the direction $\subseteq$ by induction on word length. The base case is clear. Now let $u_1 \in U_{\alpha, \Omega}$ for some $\alpha \in \Sigma$ and let $u \in P_\Omega$. We may choose the order on $\Sigma$ such that $u_1 \in U_\Omega^-$. Then by the induction assumption, we have $u \in U_\Omega^- \cdot U_\Omega^+ \cdot Y$ and hence also $u_1u \in U_\Omega^- \cdot U_\Omega^+ \cdot Y$.
			
			It remains to show $\hat{P} \subseteq U_\Omega^- U_{\Omega}^+ N_\Omega$. We use the fact that for every $n \in N_\Omega$ we have $n U_{\alpha,\Omega} n^{-1} = U_{[n](\alpha),n.\Omega} U_{[n](\alpha), \Omega}$ where $[n]$ is the representative of the spherical Weyl group corresponding to $n$. This can be used to show
			$$
			\hat{P}_\Omega = \langle N_\Omega, U_{\alpha,\Omega} \colon \alpha \in \Sigma \rangle = \langle U_{\alpha, \Omega} \colon \alpha \in \Sigma \rangle \cdot N_\Omega = P_\Omega N_\Omega = U_\Omega^- U_\Omega^+ N_\Omega, 
			$$ 
			making use of the description of $P_\Omega$.
		\end{proof}
		
		Next, we will obtain a 'mixed Iwasawa'-decomposition for the group $G_\F$, upgrading the result of Proposition \ref{prop:BT_mixed_Iwasawa_rank_1} in rank 1.
		\begin{theorem}\label{thm:BT_mixed_Iwasawa}
			For every order on the root system $\Sigma$,
			$G_\F = U^+ \cdot \operatorname{Nor}_{G_\F}(A_\F) \cdot \hat{P}_o$.
		\end{theorem}
		\begin{proof}
			We abbreviate $\tilde{G} := U^+ \cdot \operatorname{Nor}_{G_\F}(A_\F) \cdot \hat{P}_o$. It suffices to show that $g\tilde{G} \subseteq \tilde{G}$ for every $g\in G_\F$, since then $G_\F \subseteq G_\F \tilde{G} \subseteq \tilde{G} \subseteq G_\F$. By \cite[Theorem \ref{I-thm:BWB}]{AppAGRCF}, and the fact that non-trivial elements of the spherical Weyl group $W_s$ of the form $m(u)$ can be obtained from elements in $U^+$ and $U^-$, $G_\F = \langle U^+, U^-, T_\F  \rangle$, where 
			\begin{align*}
				U^+ &= \langle (U_\alpha)_\F \colon \alpha >0 \rangle \\
				U^- &= \langle (U_\alpha)_\F \colon \alpha <0 \rangle \\
				T_\F &= \operatorname{Cen}_{G_\F} (A_\F)
			\end{align*}
			as earlier. Since $U^+ \tilde{G} \subseteq \tilde{G}$ and $T_\F \tilde{G} \subseteq \tilde{G}$, it suffices to show $U^- \tilde{G} \subseteq \tilde{G}$. Since equality holds in $[(\frakg_\alpha)_\F, (\frakg_\beta)_\F] = (\frakg_{\alpha + \beta})_\F$, 
			$$
			U^- = \langle (U_{-\alpha})_\F \colon \alpha \in \Delta \rangle
			$$
			for the basis $\Delta \subseteq \Sigma$ associated to the chosen order, compare \cite[Remark 14.5(2)]{Bor}. Therefore it suffices to show $(U_{-\alpha})_\F \tilde{G} \subseteq \tilde{G}$ for all $\alpha\in \Delta$ to show the theorem.
			
			Now let $\alpha \in \Delta$ and consider the complete lexicographical order on $\Sigma$ such that $\alpha < \delta$ for all $\delta \in \Delta$. Let $U_\alpha' := \langle (U_\beta)_\F \colon \beta >0, \beta \neq \alpha\rangle \subseteq U^+$. Then by \cite[Lemma \ref{I-lem:BCH_normalizer}]{AppAGRCF}, $U^+ = (U_\alpha)_\F \cdot U_\alpha' = U_\alpha' \cdot (U_\alpha)_\F$. We note that when we instead consider the order with basis $r_\alpha(\Delta)$, $-\alpha$ is the smallest positive element as in the proof of Proposition \ref{prop:BTUOmega} and then \cite[Lemma \ref{I-lem:BCH_normalizer}]{AppAGRCF} gives $(U_{-\alpha})_\F \cdot U_\alpha' = U_\alpha' \cdot (U_{-\alpha})_\F$. We use Proposition \ref{prop:BT_mixed_Iwasawa_rank_1} to show
			\begin{align*}
				(U_{-\alpha})_\F \tilde{G} &= (U_{-\alpha})_\F U_\alpha' (U_\alpha)_\F \operatorname{Nor}_{G_\F}(A_\F) \hat{P}_o \\
				&= U_\alpha' (U_{-\alpha})_\F (U_\alpha)_\F \operatorname{Nor}_{G_\F}(A_\F) \hat{P}_o \\
				&\subseteq U_\alpha' (U_\alpha)_\F (\operatorname{Nor}_{G_\F}(A_\F) \cap (L_{\pm\alpha})_\F) L_o \cdot \operatorname{Nor}_{G_\F}(A_\F) \hat{P}_o \\
				&\subseteq  U^+ T_\F L_o \cdot \operatorname{Nor}_{G_\F}(A_\F) \hat{P}_o \ \amalg \ U^+ T_\F  m L_o \cdot \operatorname{Nor}_{G_\F}(A_\F) \hat{P}_o,
			\end{align*}
			where $L_o = \langle U_{\alpha, o}, U_{-\alpha, o} \rangle$ and $m\in \operatorname{Nor}_{G_\F}(A_\F)$ represents the reflection $r_\alpha$ in the spherical Weyl group $W_s$. By Lemma \ref{prop:BTL_l+l'}, $L_o \subseteq U_{\alpha, o} U_{-\alpha,o} \operatorname{Nor}_{G_\F}(A_\F)$ and $L_o \subseteq U_{-\alpha, o} U_{\alpha,o} \operatorname{Nor}_{G_\F}(A_\F)$. Then
			\begin{align*}
				(U_{-\alpha})_\F \tilde{G} & \subseteq U^+ T_\F (U_\alpha)_\F (U_{-\alpha})_\F \operatorname{Nor}_{G_\F}(A_\F) \hat{P}_o \ \amalg \ U^+ T_\F m (U_{-\alpha})_\F (U_{\alpha})_\F \operatorname{Nor}_{G_\F}(A_\F) \hat{P}_o \\
				&=  U^+ T_\F (U_\alpha)_\F (U_{-\alpha})_\F \operatorname{Nor}_{G_\F}(A_\F) \hat{P}_o \\
				&= U^+ (U_\alpha)_\F (U_{-\alpha})_\F T_\F \operatorname{Nor}_{G_\F}(A_\F) \hat{P}_o \\
				&= U^+ (U_{-\alpha})_\F \operatorname{Nor}_{G_\F}(A_\F) \hat{P}_o
			\end{align*}
			where we used $m (U_{\alpha})_\F m^{-1} = (U_{-\alpha})_\F$ and $m (U_{-\alpha})_\F m^{-1} = (U_{\alpha})_\F$. We claim that for every $u' \in (U_{-\alpha})_\F$ and $n \in \operatorname{Nor}_{G_\F}(A_\F)$, $u'n \in \tilde{G}$. To see this, consider $\beta = [n^{-1}](\alpha)$ and $v := n^{-1}u'n\in (U_{[n^{-1}](-\alpha)})_\F = (U_{-\beta})_\F $. Then we can apply the rank one mixed Iwasawa decomposition, Proposition \ref{prop:BT_mixed_Iwasawa_rank_1}, for $\beta$ to obtain 
			\begin{align*}
				u'n = nv &\in n(L_{\pm \beta})_\F \subseteq n(U_{\beta})_\F \operatorname{Nor}_{G_\F}(A_\F) \hat{P}_o \\
				&= n (U_{\beta})_\F n^{-1} n \operatorname{Nor}_{G_\F}(A_\F) \hat{P}_o  \\
				&= (U_\alpha)_\F \operatorname{Nor}_{G_\F}(A_\F) \hat{P}_o \subseteq \tilde{G}.
			\end{align*}
			Assuming the claim, $(U_{-\alpha})_\F \tilde{G} \subseteq U^+ (U_{-\alpha})_\F \operatorname{Nor}_{G_\F}(A_\F) \hat{P}_o \subseteq U^+ \tilde{G} \hat{P}_o = \tilde{G}$ completes the proof. 
		\end{proof}
		
		The mixed Iwasawa decomposition can be used to show that $\hat{P}_o = \operatorname{Stab}_{G_\F}(o)$. In fact we will eventually show that $\hat{P}_\Omega$ is the whole pointwise stabilizer of $\Omega$ in Theorem \ref{thm:BTstab}.
		\begin{proposition}\label{prop:BTstab_point}
			The stabilizer of a single point $p \in \A$ satisfies
			$$
			\operatorname{Stab}_{G_\F}(p) = \hat{P}_p = \langle N_{p} , U_{\alpha, p} \colon \alpha \in \Sigma \rangle = U_p^- U_{p}^+ N_p.
			$$
		\end{proposition}
		\begin{proof} The two expressions on the right coincide with $\hat{P}_p$, see Proposition \ref{prop:BTPOmegaUUN}.	The inclusion $\operatorname{Stab}_{G_\F}(p)  \supseteq \hat{P}_p$ is clear. For the other direction, we first consider $p=o$ and use the mixed Iwasawa decomposition, Theorem \ref{thm:BT_mixed_Iwasawa}, $G_\F = U^+ \operatorname{Nor}_{G_\F}(A_\F) \hat{P}_o$. For $g \in \operatorname{Stab}_{G_\F}(o)$, let $u \in U^+, n \in \operatorname{Nor}_{G_\F}(A_\F)$ and $p \in \hat{P}_o$ with $g=unp$. Then $o = g.o = unp.o = un.o = n.o$ by Proposition \ref{prop:UonA}. Thus $n \in N_o$ and hence also $u \in U_o^+ = \langle U_{\alpha,o} \colon \alpha>0 \rangle$.
			
			For general $p=a.o\in \A$ and elements $g\in G_\F$, $g.p=p$ if and only if $a^{-1}ga.o=o$. Use Proposition \ref{prop:BTPOmegaUUN} to write $a^{-1}ga = uu'n$ as a product of elements $u\in U^+_o, u' \in U_o^-$ and $n \in N_o$. We note that $aua^{-1} \in U_p^+, au'a^{-1} \in U_p^-$ and $ana^{-1} \in N_p$, whence $g \in \hat{P}_p$.
		\end{proof}

		We state a Lemma that allows us to prove that $\hat{P}_\Omega$ is the pointwise stabilizer of any finite subset $\Omega\subseteq \A$. 
		\begin{lemma}\label{lem:BTstab_dir_lemma}
			For $p,q\in \A$ there is an order on $\Sigma$ such that $U_q^+ \subseteq U_p^+$.
		\end{lemma}
		\begin{proof}
			If $q=o$ and $p$ lies in the fundamental Weyl chamber $C_0 \subseteq \A$, then we can take the standard order $>$ associated to $C_0$. Elements $g\in U_q^+ = U_\F(O)$ stabilize all of $C_0$ pointwise, see Theorem \ref{thm:NO_fixes_s0}, so $U_q^+ \subseteq U_p^+$.
			
			In general, there is an element $n \in \operatorname{Nor}_{G_\F}(A_\F)$ with $n.q = 0$ and $n.p \in C_0$ ($n$ is a translation by $-q$ followed by a representative of a unique element in the spherical Weyl group). Then by the above $U_{n.q}^+ \subseteq U_{n.p}^+$ with respect to the standard order associated to $C_0$. If $>_2$ and $+_2$ denote the order with positive roots $\Sigma_{>_2 0} = \left\{ \alpha \in \Sigma \colon [n^{-1}](\alpha) > 0 \right\} $, then $U_{n.p}^+ =n^{-1} U_{p}^{+_2} n$, so
			\begin{align*}
	U_q^{+_2} &= n U_{n.q}^+ n^{-1} \subseteq n U_{n.p}^+ n^{-1}= U_p^{+_2}. \qedhere
			\end{align*}
			
		\end{proof}
		
		\begin{theorem}\label{thm:BTstab_fin}
			Let $\Omega \subseteq \A$ be a finite subset. Then the pointwise stabilizer of $\Omega$ satisfies
			$$
			\operatorname{Stab}_{G_\F}(\Omega) = \hat{P}_\Omega = \langle N_\Omega, U_{\alpha,\Omega} \colon \alpha \in \Sigma \rangle =  U_\Omega^+ U_\Omega^- N_\Omega.
			$$
		\end{theorem}
		\begin{proof}
			We claim that
			$$
			\hat{P}_{\Omega \cup \{p\}} = \hat{P}_\Omega \cap \hat{P}_p
			$$
			for all non-empty, finite $\Omega \subseteq \A$ and $p \in \A$. The inclusion $\subseteq$ is clear. Now let $g \in \hat{P}_\Omega \cap \hat{P}_p$. For some $q \in \Omega$, choose an ordering of $\Sigma$ such that $U_q^+ \subseteq U_p^+$ by Lemma \ref{lem:BTstab_dir_lemma}. 
			Now use Proposition \ref{prop:BTPOmegaUUN} to write $g = nu'u$ with $n \in N_{\Omega}$, $u' \in U_{\Omega}^-$, $u\in U_{\Omega}^+$. We have $u \in U_{\Omega}^+ \subseteq U_{q}^+ \subseteq U_{p}^+$, so $gu^{-1} = nu' \in \hat{P}_p$. Since now $u'.p = n^{-1}.p$, we have $u'.p = p$ by Proposition \ref{prop:UonA}. This means that all three elements $u, u'$ and $n$ fix $\Omega \cup \{p\}$, so $g \in N_{\Omega \cup \{p\}} U_{\Omega \cup \{p\}}^- U_{\Omega \cup \{p\}}^+ = \hat{P}_{\Omega \cup \{p\}}$. 
			
			We now show $\operatorname{Stab}(\Omega) = \hat{P}_{\Omega} $ by induction over the size of $\Omega$. If $|\Omega|=1$, this is Proposition \ref{prop:BTstab_point}. Now assume the statement holds for $\Omega$. Given any $p\in \A$, we have
			$$
			\operatorname{Stab}_{G_\F}(\Omega \cup \{p\}) = \operatorname{Stab}_{G_\F}(\Omega) \cap \operatorname{Stab}_{G_\F}(p) 
			= \hat{P}_\Omega \cap \hat{P}_p = \hat{P}_{\Omega \cup \{p\}},
			$$
			where we used the induction assumption and the claim.
		\end{proof}
		
\subsection{Axiom (A2)} \label{sec:A2_subsubsection}
		
		In this section we assume $\Sigma$ to be reduced. We are now equipped to prove axiom (A2). We first show a special case.
		
		\begin{proposition}\label{prop:A2_b}
			Let $g\in G_\F$ and $\Omega = g^{-1}\A \cap \A$. Then there exists an element $w\in W_a$ such that for all $p \in \Omega$, $g.p = w(p)$.
		\end{proposition}
		\begin{proof}
			We may assume that $\Omega \neq \emptyset$. For subsets $Y \subseteq \Omega$ we consider
			$$
			N_{g,Y} := \left\{ n \in \operatorname{Nor}_{G_\F}(A_\F) \colon  g.p = n.p \text{ for all } p\in Y \right\},
			$$
			so that our goal is to prove that $N_{g,\Omega} \neq \emptyset$.
			If $Y=\{p\}$ is a singleton, then $N_{g,\left\{p\right\}}$ is non-empty, since if $p=a.o$ and $g.p = b.o$ for $a,b\in A_\F$, then $ba^{-1} \in N_{\{p\}}$. We will now show by induction on the size of $Y$, that $N_{g,Y} \neq \emptyset$ for all \emph{finite} sets $Y \subseteq \Omega$.
			
			Let $Y\subseteq \Omega$ finite and $p \in \Omega$. Assume that there are $n_{Y} \in N_{g,Y}$ and $n_p \in N_{g,\{p\}} $, then $g^{-1}n_Y$ and $g^{-1}n_p $ lie in the pointwise stabilizers $\operatorname{Stab}_{G_\F}(Y)$ and $ \operatorname{Stab}_{G_\F}(\{p\})$. We may choose an order on $\Sigma$ such that $U_{Y}^+ \subseteq U_{\{p\}}^+$, by making sure that the defining chamber for the	order based at some point in $Y$ contains $p$, see Lemma \ref{lem:BTstab_dir_lemma}.
			Then, by Theorem \ref{thm:BTstab_fin},
			\begin{align*}
				n_Y^{-1} n_p &\in N_Y U_Y^- U_Y^+ U_{\{p\}}^+ U_{\{p\}}^- N_{\{p\}} 
				=  N_Y U_{Y}^- U_{\{p\}}^+ U_{\{p\}}^- N_{\{p\}}\\
				&=  N_Y U_{Y}^- U_{\{p\}}^- U_{\{p\}}^+ N_{\{p\}}
				\subseteq N_Y U^- U^+ N_{\{p\}}
			\end{align*}
			Let $n_Y' \in N_Y, n_p' \in N_{\{p\}}$ such that $(n_Y')^{-1}n_Y^{-1}n_pn_p' \in U^{-} U^{+} \cap \operatorname{Nor}_{G_\F}(A_\F)$. 
			Since every element of $U^+$ stabilizes some affine chamber pointwise, so does every element of $U^-U^+ \cap \operatorname{Nor}_{G_\F}(A_\F)$, by Proposition \ref{prop:UonA}. The element of the affine Weyl group represented by $(n_Y')^{-1}n_Y^{-1}n_pn_p'$ thus acts trivially on some affine chamber, hence acts trivially on all of $\A$. 
			Therefore $[n_Yn_Y'] = [n_pn_p'] \in W_a$ and $ n_Y n_Y' \in N_{g,Y} \cap N_{g,\{p\}} = N_{g,Y\cup \{p\}}$, concluding the induction.
			
			Recall that the affine Weyl group $W_a = W_s \ltimes \A$ is isomorphic to the quotient $\operatorname{Nor}_{G_\F}(A_\F) / \operatorname{Cen}_{G_\F}(A_\F)$. Let $\pi \colon \operatorname{Nor}_{G_\F}(A_\F) \to W_s$ be the induced map to the finite group $W_s$. Let $Y_0 \subseteq \Omega$ be a finite subset, so that $|\pi(N_{g,Y_0})|$ is minimal. For any $p \in \Omega$,
			$$
			N_{g,Y_0 \cup \{p\}} = N_{g,Y_0} \cap N_{g,\{p\}} \subseteq N_{g,Y_0}
			$$
			and thus by minimality $\pi(N_{g,Y_0 \cup \{p\}}) = \pi(N_{g,Y_0})$. We claim that in fact $N_{g,Y_0 \cup \{p\}} = N_{g,Y_0}$ for every $p \in \Omega$.
			Pick $n_0\in N_{g,Y_0}$. 
			We decompose $w := [n_0] =  (t_a,\pi(n_0)) \in W_a = \A \ltimes W_s$, so that $t_a$ is translation given by multiplication of some $a\in A_\F$. For any $p \in \Omega$, there exists $n' \in N_{g,Y_0 \cup \{p\}}$ such that $\pi(n_0) = \pi(n')$ and we decompose similarly $w' := [n'] = (t_{a'},\pi(n')) \in W_a$. Acting on some $q\in Y_0$, we have
			$$
			a.(\pi(n_0)(q)) = n_0.q = g.q = n'.q = a'.(\pi(n')(q)) = a'.(\pi(n_0)(q)),
			$$
			which implies $t_a = t_{a'}$ and thus $w= (t_a,\pi(n_0)) = (t_{a}', \pi(n')) = w'$. For any $p \in \Omega$ we thus have
			$$
			n_0.p = w(p) = w'(p) =  n'.p = g.p,
			$$
			so $n_0 \in N_{g, Y_0\cup\{p\}}$ as required. Finally
			$$
			N_{g,\Omega} 
			= \bigcap_{p \in \Omega} N_{g, Y_0\cup \{p\}} = N_{g,Y_0} \neq \emptyset,
			$$
			so taking any $n \in N_{g,Y_0}$ provides the required element $w = [n] \in W_a$ with $w(p) = g.p$ for all $p \in \Omega$.
		\end{proof}
		From the above proof, we can also extract the following Lemma, by taking $g=\operatorname{Id}$ and noting that then $N_\Omega =  N_{g,\Omega} =  N_{g,Y_0} = N_{Y_0}$.
		\begin{lemma}\label{lem:BTNY0=NOmega}
			Let $\Omega \subseteq \A$. There exists a finite subset $Y_0 \subseteq \Omega$ such that $N_{Y_{0}} = N_{\Omega}$.
		\end{lemma}
		To be able to prove the $W_a$-convexity part of axiom (A2), we want to describe the stabilizer of a (possibly infinite) subset $\Omega$ as in Theorem \ref{thm:BTstab_fin}. For this we first need a Lemma.
		
		\begin{lemma}\label{lem:U+U-fix_U+fix}
			Let $u^+ \in U^+$, $u^- \in U^-$ and $p\in \A$. If $u^+.p = u^-.p$, then $u^+.p = p = u^-.p$.
		\end{lemma}
		\begin{proof}
			We start by noting that if $a \in \F^{n\times n}$ is an upper triangular matrix and $b \in \F^{n\times n}$ is a lower triangular matrix, both with ones on the diagonal such that $ab \in O^{n\times n}$, then $a, b \in O^{n\times n}$, as can be checked by matrix calculations. 
			
			We will now prove the Lemma in the case that $p = o \in \B$. We know that $u^+.p = u^-.p$ which is equivalent to $(u^+)^{-1}u^- \in G_\F(O)$. The adjoint map $\operatorname{Ad} \colon G_\F \to \operatorname{GL}(\frakg_\F)$ is a $\K$-morphism that sends elements of $U^+$ to upper triangular matrices and elements of $U^-$ to lower triangular matrices, see \cite[Lemma \ref{I-lem:kan_form}]{AppAGRCF}. Moreover, since $\operatorname{Ad}(g)X = gXg^{-1}$ for $X \in \frakg_\F$, we can use Lemma \ref{lem:orthogonal_valuation} on a basis to see that $\operatorname{Ad}(g)$ is defined by polynomials and if $g \in G_\F(O)$, then these polynomials have coefficients in $O$. 
			Thus $\operatorname{Ad}(G_\F(O)) \subseteq \operatorname{Ad}(G_\F)(O)$ and $\operatorname{Ad}((u^+)^{-1})\operatorname{Ad}(u^-) \in O^{n\times n}$ and by the preceding remark, $\operatorname{Ad}(u^+) \in O^{n\times n}$ and $\operatorname{Ad}(u^-) \in O^{n\times n}$.
			
			While $\operatorname{Ad} \colon G_\F \to \operatorname{GL}(\frakg_\F)$ may have a nontrivial finite kernel, its restriction $\operatorname{Ad}|_{U^+} \colon U^+ \to \operatorname{GL}(\frakg_\F) $ is an isomorphism onto its image, since the exponential map gives an isomorphism $U^+ \cong \bigoplus_{\alpha>0} \frakg_\alpha$ and since $\operatorname{ad}$ restricted to the Lie algebra of $U^+$ is an isomorphism to its image defined over $\K$. Then, the inverse map $\operatorname{Ad}(U^+) \to U^+$ is also defined by polynomials with coefficients in $\K$, so $u^+ \in O^{n\times n}$. Similarly $u^- \in O^{n\times n}$. This means that $u^+.o = o$ and $u^-.o = o$ as required.
			
			If now $p = a.o$ for some $a \in A_\F$, then $(u^+)^{-1}u^-.p = p$ if and only if $$a^{-1}(u^+)^{-1}aa^{-1}u^- a \in G_\F(O).$$
			By the above argument $a^{-1}u^+ a$, $a^{-1}u^-a \in G_\F(O)$ and thus $u^+.p = p$ and $u^-.p = p$.
		\end{proof}
		
		We upgrade the description of $\hat{P}_{\Omega}$ as the pointwise stabilizer of a finite subset in Theorem \ref{thm:BTstab_fin}, to arbitrary subsets $\Omega \subseteq \A$. 
		\begin{theorem}\label{thm:BTstab}
			The pointwise stabilizer of any subset $\Omega \subseteq \A$ satisfies
			$$
			\operatorname{Stab}_{G_\F}(\Omega) = \hat{P}_\Omega = U_\Omega^+ U_{\Omega}^- N_{\Omega} .
			$$
		\end{theorem}
		\begin{proof} The inclusions $\supseteq$ are clear. We apply Lemma \ref{lem:BTNY0=NOmega} to obtain a finite subset $Y_0 \subseteq \Omega$ such that $N_{Y_0} = N_{\Omega}$. Then $\operatorname{Stab}(\Omega) \subseteq \operatorname{Stab}(Y_0) = U_{Y_0}^+ U_{Y_0}^- N_{Y_0}$ by Theorem \ref{thm:BTstab_fin}. Let $g\in \operatorname{Stab}(\Omega)$ and $u^+ \in U_{Y_0}^+, u^- \in  U_{Y_0}^-, n \in N_{Y_0} = N_\Omega$ with $g=u^+u^-n$. Then $u^+ u^- = gn^{-1} \in \operatorname{Stab}(\Omega)$. Thus, $(u^+).p = u^-.p$ for all $p\in \Omega$ and by Lemma \ref{lem:U+U-fix_U+fix}, $u^+.p = p = u^-.p$, in particular $u^+ \in U_\Omega^+$ and $u^- \in U_{\Omega}^-$. We now know
			$$
			\operatorname{Stab}_{G_\F}(\Omega) \subseteq U_\Omega^+ U_\Omega^- N_\Omega \subseteq \hat{P}_\Omega \subseteq \operatorname{Stab}_{G_\F}(\Omega)
			$$
			concluding the proof.
		\end{proof}
		We now prove a special case of $W_a$-convexity.
		
		\begin{proposition}\label{prop:A2_a}
			Let $g\in G_\F$. Then $\Omega = g^{-1}\A \cap \A$ is a finite intersection of affine half-apartments. 
		\end{proposition}
		\begin{proof}
			By Proposition \ref{prop:A2_b}, we know that there is an $n\in \operatorname{Nor}_{G_\F}(A_\F)$ such that $g^{-1}n \in \operatorname{Stab}(\Omega)$ pointwise and by Theorem \ref{thm:BTstab} $g^{-1}n \in U_{\Omega}^+ U_{\Omega}^- N_{\Omega}$. So let $u^+ \in  U_{\Omega}^+, u^- \in  U_{\Omega}^- $ and $n' \in  N_{\Omega}$ with $g^{-1} n = u^+ u^- n'$. We note that for $\tilde{n} := n(n')^{-1}$ we have $g^{-1}\tilde{n} = u^+ u^- $.
			
			Recall that the affine half-space given by $\alpha \in \Sigma$ and $k\in \Lambda$ is
			$$
			H_{\alpha,k}^+ = \left\{ a.o \in \A \colon (-v)(\chi_\alpha(a)) \geq k  \right\}.
			$$
			In Proposition \ref{prop:Uconv}, we showed that for $u^+$ there are $k_\alpha \in \Lambda \cup \{-\infty\}$ for every $\alpha >0$ such that 
			$$
			\left\{ p\in \A \colon u^+.p \in \A \right\} = \bigcap_{\alpha >0} H_{\alpha, k_\alpha}^+.
			$$	
			where we take the convention that $H_{\alpha, -\infty}^+ = \A$. By changing the order on $\Sigma$, we similarly obtain for $u^-$ some $k_\alpha\in \Lambda \cup \{-\infty\}$ for $\alpha <0$ such that
			$$
			\left\{ p\in \A \colon u^-.p \in \A \right\} = \bigcap_{\alpha <0} H_{\alpha, k_\alpha}^+.
			$$
			We now show that
			$$
			\Omega = \bigcap_{\alpha \in \Sigma} H_{\alpha, k_\alpha}^+.
			$$
			By Lemma \ref{lem:U+U-fix_U+fix}, for any $p\in \Omega$ we have $u^+.p =p$  and $u^{-}.p=p$, hence $p \in \bigcap H_{\alpha, k_\alpha}^+$. If on the other hand $p \in \bigcap H_{\alpha, k_\alpha}^+$, then $u^+u^-.p=u^+.p=p$, in particular $p \in u^+u^- \A \cap \A = g^{-1}\tilde{n}\A \cap \A = g^{-1} \A \cap \A = \Omega$. This concludes the proof that $\Omega$ is a finite intersection of half-spaces.
		\end{proof}
		
		We now put together Propositions \ref{prop:A2_b} and \ref{prop:A2_a} to prove axiom (A2).
		\begin{theorem}\label{thm:A2}
			Axiom (A2) holds: 
			\begin{enumerate}
				\item [(A2)] For every $f,f' \in \Fun$, the set $B := f^{-1}(f(\A) \cap f'(\A)) \subseteq \A$ is a finite intersection of affine half-apartments and there is $w\in W_a$ such that $f|_B = f'\circ w |_{B}$. 
			\end{enumerate}
		\end{theorem}
		\begin{proof}
			By definition of $\Fun$, there are $h,h' \in G_\F$ such that $f=h.f_0$ and $f' = h'.f_0$, where $f_0 \colon \A \to \B$ is the inclusion. Then for $g:= (h')^{-1}h$ we have $B = g^{-1}\A \cap \A$ since
			$$
			f(g^{-1}\A \cap \A) = h.(g^{-1}\A \cap \A) = h'.\A \cap h.\A = f'(\A) \cap f(\A) = f(B).
			$$ 
			By Proposition \ref{prop:A2_a}, $B$ is a finite intersection of affine half-apartments. By Proposition \ref{prop:A2_b}, there is $w\in W_a$ such that $g.p = w(p)$ for all $p\in B$, so that for all $p \in B$
			$$
			f(p) = h.p = h'g.p = h'.w(p) = h'.f_0(w(p)) = (f' \circ w)(p),
			$$
			concluding the proof.
		\end{proof}
		
		Now we are able to describe the (not necessarily pointwise) stabilizer of $\A$.
		
		\begin{proposition}\label{thm:stab'_A}
			The (not necessarily pointwise) stabilizer of $\A$ is $\operatorname{Stab}_{G_\F}'(\A) = \operatorname{Nor}_{G_\F}(A_F) = A_\F N_\F$. 
		\end{proposition} 
	\begin{proof}
		The inclusions $\supseteq$ follow directly from the definitions. Now let $g\in \operatorname{Stab}_{G_\F}(A_\F)$. By axiom (A2), 
		there exists $w\in W_a$ such that $g.f_0 = f_0 \circ w$. Let $a\in A_\F$ and $k\in N_\F$ with $ak.f_0 = f_0 \circ w = g.f_0$. Then $k^{-1}a^{-1}g \in A_\F(O)M_\F$ by Theorem \ref{thm:stab_A}. Then $g\in A_\F N_\F A_\F(O) M_\F = A_\F N_\F$. 
	\end{proof}

\subsection{Axiom (A4)}\label{sec:A4}
		
		In this section the root system $\Sigma$ does not need to be reduced, except for being able to apply axiom (A2) in Lemma \ref{lem:subsector_nice}. Axiom (A4) is a statement about sectors. Let $s_0 = f_0(C_0) \subseteq \B$ be the fundamental sector corresponding to the fundamental Weyl chamber $C_0 \subseteq \A$. All sectors are of the form $g.s_0$ for some $g\in G_\F$. If a sector $s$ is a subset of another sector $s'$, then $s$ is called a \emph{subsector} of $s'$. 
		
		From (A2) we get that subsectors of $s_0$ are of the form $a.s_0$ for $a\in A_\F$.
		\begin{lemma}\label{lem:subsector_nice}
			For every subsector $s \subseteq s_0$ there exists $a\in A_\F$ such that $s=a.s_0$.
		\end{lemma}
		\begin{proof}
			A general subsector of $s_0$ is of the form $g.s_0$ for some $g\in G_\F$. By Axiom (A2) we know that $s_0 \to g.s_0$ is realized by an element $w\in W_a$ of the affine Weyl group, $g.b.o = w(b.o)$ for all $b.o \in s_0$. We decompose $w = ( t_a,w_s) \in \A \rtimes W_s$ for some $a\in A_\F$. Since $o\in s_0$ and $a.o = w(o) \in g.s_0 \subseteq s_0$, we know that $\chi_\alpha(a) \geq 1$ for all $\alpha >0$. We also know that $w_s(s_0)$ is one of the finitely many sectors based at $o$. If $w_s(s_0) \neq s_0$, then there is some $\alpha >0$ with $\chi_\alpha(b) \leq 1$ for all $b.o \in w_s(s_0)$. Since $w_s(s_0)$ is a cone with open interior, there are $b.o\in w_s(s_0)$ with arbitrary negative $\chi_\alpha(b)$, in particular there is some $b.o\in w_s(s_0)$ with $(-v)(\chi_\alpha(b)) < (-v)(\chi_\alpha(a)^{-1})$, so that $(-v)(\chi_\alpha(ab)) < 0$. But this contradicts $a.w_s(s_0) \subseteq s_0$, since $a.b.o \notin s_0$. We conclude that $w_s(s_0) = s_0$ and thus $g.s_0 = a.s_0$.
		\end{proof}

		While studying the model apartment $\A$, Bennett \cite[Prop 2.9]{Ben1} proved the following lemma.
		
		\begin{lemma}
			For all $p \in \A$ there is a $q\in \A$ such that $q+C_0 \subseteq (p+C_0) \cap C_0$.
		\end{lemma}
		In our setting this lemma translates in terms of $A_\F.o = \A\subseteq \B$.
		\begin{lemma} \label{lem:A4:a_sub}
			For all $a\in A_\F$ there is a $b\in A_\F$ such that $b.s_0 \subseteq a.s_0 \cap s_0$.
		\end{lemma}
		
		We also get a slightly more general statement, illustrated in Figure \ref{fig:subsectors_lemma}.
		\begin{lemma}\label{lem:A4:a_sub_cor}
			For all subsectors $s' \subseteq s_0$ and for all $a\in A_\F$ there is a subsector $s \subseteq a.s' \cap s_0$.
		\end{lemma}
		\begin{proof}
			Let $s'=\tilde{a}.s_0$ for some $\tilde{a} \in A_\F$, see Lemma \ref{lem:subsector_nice}. We apply Lemma \ref{lem:A4:a_sub} with $a\tilde{a}$ to get $b\in A_\F$ with $s:= b.s_0 \subseteq a\tilde{a}.s_0 \cap s_0 = a.s' \cap s_0$. 
		\end{proof}
		
		\begin{figure}[h]
			\centering
			\includegraphics[width=0.5\linewidth]{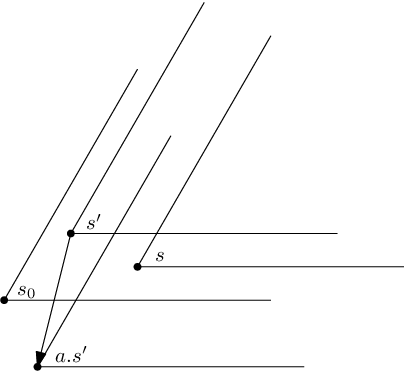}
			\caption{Lemma \ref{lem:A4:a_sub_cor} states that for every subsector $s' \subseteq s_0$ and $a \in A_\F$ there is a sector $s$ contained in both $s'$ and $a.s'$.   }
			\label{fig:subsectors_lemma}
		\end{figure}

		We now use Proposition \ref{prop:exists_a_NO} to show that while elements of $U_\F$ may not fix $s_0$ itself, they at least fix a subsector.
		\begin{lemma}\label{lem:N_fixes_subsector}
			For every $u\in U_\F$ there is a subsector $s\subseteq s_0$ with $u.p=p$ for all $p\in s$.	
		\end{lemma} 
		\begin{proof}
			Let $a\in A_\F$ with $ a^{-1}ua \in U_\F(O)$ as in Proposition \ref{prop:exists_a_NO}. For all $b.o \in s_0$ we have $a^{-1}ua.b.o = b.o$ by Corollary \ref{cor:NO_fixes_s0}.
			Therefore $u$ fixes $a.s_0$ pointwise. We don't know whether $a.s_0$ is a subsector of $s_0$, but we can apply Lemma \ref{lem:A4:a_sub} to find a subsector $s$ of $s_0$, which is also a subsector of $a.s_0$ and therefore is fixed pointwise by $u$. 
		\end{proof}
		Now that we understand the action of $U_\F$ better, let us turn to the group $B_\F=M_\F A_\F U_\F$ introduced in \cite[Section \ref{I-sec:BWB}]{AppAGRCF}, where $M_\F=\operatorname{Cen}_{K_\F}(A_\F) ,
		A_\F$ and $U_\F$ are as before.
		\begin{lemma}\label{lem:A4:B_acts}
			For all $b\in B_\F = M_\F A_\F U_\F$ there is a subsector $s\subseteq s_0$ with $b.s \subseteq s_0$.
		\end{lemma}
		\begin{proof}
			We first recall the elements of $M_\F$ fix all of $\A$ pointwise. Let $m\in M_\F$, $u \in U_\F$ and $a\in A_\F$ such that $b=mua$. Using Lemma \ref{lem:N_fixes_subsector}, we find a subsector $s'\subseteq s_0$ which is fixed by $u$ pointwise. Applying Lemma \ref{lem:A4:a_sub_cor} to $s'$ and $a^{-1}$ we get a subsector $s\subseteq s_0 \cap a^{-1}.s'$. We now have
			$$
			b.s \subseteq b.a^{-1}.s' = m u .s' = m .s' = s' \subseteq s_0,
			$$
			as claimed.
		\end{proof}
		We are now ready to prove axiom (A4) using the Bruhat decomposition $G_\F=B_\F W_sB_\F$ \cite[Theorem \ref{I-thm:BWB}]{AppAGRCF}.
		\begin{theorem}\label{thm:A4}
			Axiom 
			\begin{itemize}
				\item [(A4)] For any sectors $s_1,s_2 \subseteq \B$ there are subsectors $s_1' \subseteq s_1, s_2' \subseteq s_2$ such that there is an $f\in \Fun$ with $s_1', s_2' \subseteq f(\A)$.
			\end{itemize}
			holds for $(\B,\Fun)$. 
		\end{theorem}
		
		\begin{proof}
			The action of $G_\F$ on the sectors is transitive by definition and we may hence assume without loss of generality that one of the sectors in (A4) is $s_0$ and the other is given by $g.s_0$ for some $g \in G_\F$. We have to prove 
			\begin{itemize}
				\item [(A4)'] For all $g\in G_\F$, there are subsectors $s \subseteq s_0, s' \subseteq g.s_0$ such that there is an $f\in \Fun$ with $s, s' \subseteq f(\A)$.
			\end{itemize}
			\begin{figure}[h]
				\centering
				\includegraphics[width=0.8\linewidth]{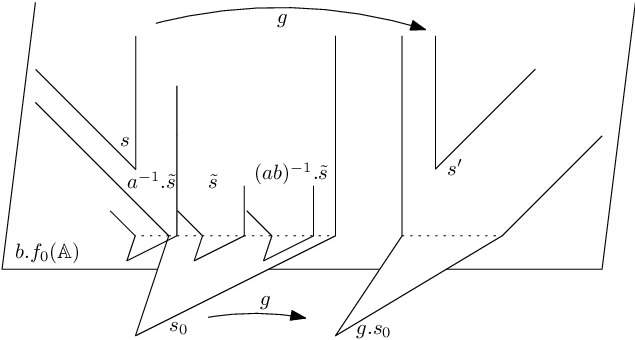}
				\caption{Axiom (A4) states that while the sectors $s_0,g.s_0$ may not lie in a common flat, they contain subsectors $s,s'$ contained in a common flat $f(\A)$. }
				\label{fig:A4proof}
			\end{figure}
			The situation is illustrated in Figure \ref{fig:A4proof}. The Bruhat-decomposition \cite[Theorem \ref{I-thm:BWB}]{AppAGRCF}, states that $G_\F$ is a disjoint union 
			$$
			G_\F = \bigcup_{w=[k] \in W_s} B_\F k B_\F
			$$
			of double cosets. 
			Let $g \in G_\F$ and $b,b' \in B_\F, k \in N_\F := \operatorname{Nor}_{K_\F}(A_\F)$ with $g=bkb'$. We further decompose $b'=ma'ua$ for $m\in M_\F, u\in U_\F(O)$ and $a,a'\in A_\F$ using $B_\F=M_\F U_\F A_\F$ and Proposition \ref{prop:exists_a_NO}. Since $(ab)^{-1} \in B_\F$, apply Lemma \ref{lem:A4:B_acts} to obtain a subsector $\tilde{s} \subseteq s_0$ with $(ab)^{-1}.\tilde{s} \subset s_0$. We apply Lemma \ref{lem:A4:a_sub_cor} to $\tilde{s}$ and $a^{-1}\in A_\F$ to obtain a subsector $s \subset a^{-1}.\tilde{s} \cap s_0$. We claim that $s$ and $s':=g.s$ are the required subsectors in (A4)' and that $f= b.f_0$ defines the common apartment.
			
			We have by construction that $s\subseteq s_0$ and clearly $s'=g.s \subseteq g.s_0$. It remains to show that $s$ and $s'$ are in the apartment $b.f_0(\A)$. We see
			$$
			b^{-1}.s \subseteq b^{-1}.a^{-1}.\tilde{s} \subseteq s_0 \subseteq f_0(\A)
			$$
			hence $s \subseteq b.f_0(\A)$ and
			\begin{align*}
				s' &= g.s = bkma'ua.s \subseteq bkma'u.\tilde{s} = bkma'.\tilde{s} \\
				&\subseteq bkma.f_0(\A) = bkm.f_0(\A) = bk.f_0(\A) = b.f_0(\A),
			\end{align*}
			where we used Corollary \ref{cor:NO_fixes_s0} for $u.\tilde{s} = \tilde{s}$ since $\tilde{s} \subseteq s_0$ and $u\in U_\F(O)$ . 
		\end{proof}

\subsection{Axiom (EC)}\label{sec:EC}
		In this section, the Jacobson-Morozov-construction as well as axiom (A2) is used, so we require $\Sigma$ to be reduced. We recall that a \emph{half-apartment in the model apartment} is a set of the form
		$$
		H_{\alpha,\ell}^+ = \left\{ a.o \in \A \colon (-v)(\chi_\alpha(a)) \geq \ell \right\}
		$$
		for some $\alpha \in \Sigma$ and some $\ell \in \Lambda$. A \emph{half-apartment in the building} is $g.H_\alpha^+$ for any $g \in G_\F$. We first use axiom (A2) to make sure that if a half-apartment of the building is included in $\A$, then it is an affine half-apartment in the model apartment.
		\begin{lemma}\label{lem:subhalfapartment_nice}
			Let $g\in G_\F, \alpha \in \Sigma, \ell \in \Lambda$ such that $g.H_{\alpha,\ell}^+ \subseteq \A$. Then there is $\alpha' \in \Sigma, \ell' \in \Lambda$ such that 
			$$
			g.H_{\alpha,\ell}^+ = H_{\alpha', \ell'}^+ .
			$$
		\end{lemma}
		\begin{proof}
			By axiom (A2), Theorem \ref{thm:A2}, there is $w \in W_a$ such that $g.H_\alpha^+ = w(H_\alpha^+)$. We decompose $w = (t_a, w_s) \in \A \rtimes W_s = W_a$ and define $\alpha' := w_s(\alpha)$. Then $g.H_{\alpha,\ell}^+ = w(H_{\alpha,\ell}^+) = a.H_{\alpha',\ell}^+ = H_{\alpha', \ell'}^+$ for $\ell' := (-v)(\chi_{\alpha'}(a)) + \ell'$. 
		\end{proof}
		We start by proving axiom
		\begin{enumerate}
			\item [(EC)] For $f_1,f_2 \in \Fun$, if $f_1(\A)\cap f_2(\A)$ is a half-apartment, then there exists $f_3 \in \Fun$ such that $f_i(\A)\cap f_3(\A)$ are half-apartments for $i\in \{1,2\}$. Moreover $f_3(\A)$ is the symmetric difference of $f_1(\A)$ and $f_2(\A)$ together with the boundary wall of $f_1(\A) \cap f_2(\A)$.
		\end{enumerate}
		in the special case where 
		$$
		f_1(\A)\cap f_2(\A) = H_{\alpha,\ell}^+ :=  \left\{ a.o \in \A \colon (-v)(\chi_\alpha(a)) \geq \ell \right\}
		$$
		for some $\alpha \in \Sigma$ and $\ell \in \Lambda$, before deducing the full statement in Theorem \ref{thm:EC}. The situation is illustrated in Figure \ref{fig:EC}.
		\begin{proposition}\label{prop:EC_alpha}
			Let $g \in G_\F$ such that $g^{-1}.\A \cap \A = H_{\alpha, \ell}^+$ for some $\alpha \in \Sigma, \ell \in \Lambda$. Then there exists an $h\in G_\F$ such that $h^{-1}.\A \cap \A = H_{\alpha,\ell}^-$ and $ h.H_{\alpha,\ell}^+ = g.H_{\alpha,\ell}^-$. Moreover there is some $n\in \operatorname{Nor}_{G_\F}(A_\F)$ such that pointwise $g.H_{\alpha,\ell}^+ = n.H_{\alpha,\ell}^+$ and $h.H_{\alpha,\ell}^- = n.H_{\alpha,\ell}^-$.
		\end{proposition}
		\begin{proof}
			Since $g.H_{\alpha,\ell}^+ = \A \cap g.\A \subseteq \A$, we can apply axiom (A2), Theorem \ref{thm:A2}, to obtain $n \in \operatorname{Nor}_{G_\F}(A_\F)$ such that $g.H_{\alpha,\ell}^+ = n.H_{\alpha,\ell}^+$ pointwise. Hence $n^{-1}g$ fixes $H_{\alpha,\ell}^+$ pointwise and by Corollary \ref{cor:NalphaO_fixes_H_affine}, $n^{-1}g \in U_{\alpha, \ell} A_\F(O) M_\F$. So let $u \in U_{\alpha, \ell}$ such that $n^{-1}g.p = u.p$ for all $p \in \A$. By Lemma \ref{lem:BTm_uuu}, the Jacobson-Morozov-homomorphism can be used to define
			\begin{align*}
				u' = \varphi_\F \begin{pmatrix}
					1 & 0 \\ - 1/t & 0
				\end{pmatrix} \in (U_{-\alpha})_\F
			\end{align*}
			with $\varphi_{-\alpha} (u') = -\varphi_\alpha(u) = -\ell$ and $m(u) = u'uu' \in \operatorname{Nor}_{G_\F}(A_\F)$ which acts as the reflection along $M_{\alpha, \ell} = \{ a.o \in \A \colon (-v)(\chi_\alpha(a)) = \ell \}$ by Proposition \ref{prop:BTmu_in_Wa}. For $h := n(u')^{-1} \in G_\F$ we then have
			$$
			h.H_{\alpha,\ell}^+ = n (u')^{-1} m(u) .H_{\alpha,\ell}^- = n u u' .H_{\alpha,\ell}^- = nu.H_{\alpha,\ell}^- = g.H_{\alpha,\ell}^-
			$$
			with $(\A \cap h.H_{\alpha,\ell}^+) = g.M_{\alpha,\ell} = n.M_{\alpha,\ell} = h.M_{\alpha,\ell} \subseteq h.H_{\alpha,\ell}^-$ and therefore 
			$$
			h.(h^{-1}.\A\cap \A) = \A \cap h.\A = (\A \cap h.H_{\alpha,\ell}^+ )\cup(\A \cap h.H_{\alpha,\ell}^-) = \A \cap h.H_{\alpha,\ell}^- = n.H_{\alpha,\ell}^-
			$$
			which implies $h^{-1}.\A \cap \A = h^{-1}n.H_{\alpha,\ell}^- = u'.H_{\alpha,\ell}^- = H_{\alpha,\ell}^- $ as required.
		\end{proof}
		
		\begin{figure}[h]
			\centering
			\includegraphics[width=0.5\linewidth]{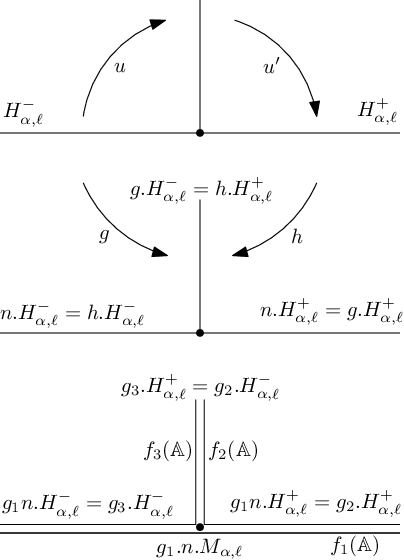}
			\caption{Axiom (EC) states that if $f_1(\A) \cap f_2(\A)$ is a half-apartment, then there exists $f_3 \in \F$ situated as illustrated. Proposition \ref{prop:EC_alpha} deals with the case where the half-apartment is contained in $\A$, illustrated in the first two parts. Theorem \ref{thm:EC} then tackles the general case of axiom (EC). }
			\label{fig:EC}
		\end{figure}

		\begin{theorem}\label{thm:EC}
			Axiom 
			\begin{enumerate}
				\item [(EC)] For $f_1,f_2 \in \Fun$, if $f_1(\A)\cap f_2(\A)$ is a half-apartment, then there exists $f_3 \in \Fun$ such that $f_i(\A)\cap f_3(\A)$ are half-apartments for $i\in \{1,2\}$. Moreover $f_3(\A) = f_1(\A)\cap f_3(\A) \cup f_2(\A)\cap f_3(\A)$ and $\partial (f_1(\A)\cap f_2(\A) ) = f_1(\A)\cap f_3(\A) \cap f_2(\A)\cap f_3(\A)$. 
			\end{enumerate}
			holds.
		\end{theorem}
		\begin{proof}
			Let $g_1, g_2 \in G_\F$ such that $f_1 = g_1.f_0$ and $f_2 = g_2.f_0$. For $g := g_1^{-1} g_2$, we have $f_2 = g_1 g.f_0$. Since $f_1(\A)\cap f_2(\A)$ is a half-apartment, so is
			$$
			(g_2)^{-1}(f_1(\A)\cap f_2(\A)) = g_2^{-1}(g_1.\A \cap g_2.\A) = g^{-1}.\A \cap \A \subseteq \A
			$$
			and by Lemma \ref{lem:subhalfapartment_nice}, there is some $\alpha \in \Sigma$ and $\ell \in \Lambda$ such that $g^{-1}.\A \cap \A = H_{\alpha, \ell}^+$. Now we use Proposition \ref{prop:EC_alpha} to obtain $h\in G_\F$ and $n\in\operatorname{Nor}_{G_\F}(A_\F)$ with $g.H_{\alpha,\ell}^+ = n.H_{\alpha,\ell}^+$, $h.H_{\alpha,\ell}^- = n.H_{\alpha,\ell}^-$ and $h.H_{\alpha,\ell}^+ = g.H_{\alpha,\ell}^-$. We define $g_3 := g_1h $ and $f_3 = g_3.f_0$.
			We have as in Figure \ref{fig:EC}.
			\begin{align*}
				g_3.H_{\alpha,\ell}^- &= g_1h.H_{\alpha,\ell}^- = g_1.n.H_{\alpha,\ell}^- \\
				g_2.H_{\alpha,\ell}^+ &= g_1g.H_{\alpha,\ell}^+ = g_1.n.H_{\alpha,\ell}^+ \\
				g_3.H_{\alpha,\ell}^+ &= g_1h.H_{\alpha,\ell}^+ = g_1g.H_{\alpha,\ell}^- = g_2.H_{\alpha,\ell}^- 
			\end{align*}
			so $f_1(\A)\cap f_3(\A) = f_3(H_{\alpha,\ell}^-)$ and $f_2(\A)\cap f_3(\A) = f_3(H_{\alpha,\ell}^+)$ are half-apartments and
			$$
			f_3(\A) = f_3(H_{\alpha,\ell}^+) \cup f_3(H_{\alpha,\ell}^-) = f_1(\A)\cap f_3(\A) \cup f_2(\A)\cap f_3(\A).
			$$
			The wall of the half-apartment $f_1(\A)\cap f_2(\A) = g_2.H_{\alpha,\ell}^+ $ is given by 
			\begin{align*}
				\partial (f_1(\A)\cap f_2(\A) ) &= g_2.M_{\alpha,\ell} = g_1g.M_{\alpha,\ell} = g_1.n.M_{\alpha,\ell} =
				g_1n.H_{\alpha,\ell}^+ \cap g_1n.H_{\alpha,\ell}^- \\
				&= f_1(\A)\cap f_2(\A) \cap f_1(\A)\cap f_3(\A). \qedhere
			\end{align*}
		\end{proof} 
		
		\noindent
		This concludes the proof of the last remaining axiom and the proof of Theorem \ref{thm:B_is_building}.

\subsection{Beyond reduced root systems}\label{sec:beyond}
		
		The main theorem of this article relies on the assumption that the root system $\Sigma$ is reduced.
		\begin{reptheorem}{thm:B_is_building}
			If the root system $\Sigma$ is reduced, then the pair $(\B,\Fun)$ is an affine $\Lambda$-building.
		\end{reptheorem}
		We remark that $\B$ is defined for any self-adjoint semisimple linear algebraic $\K$-group, independent of its root system. We expect Theorem \ref{thm:B_is_building} to still hold without the assumption on $\Sigma$. 
		\begin{question}
			Is the pair $(\B,\Fun)$ an affine $\Lambda$-building even when $\Sigma$ is not reduced?
		\end{question}
		We outline here how our proof relies on the assumption, how the assumption cannot be removed using our strategy and a possible alternative proof strategy that might be of use to eliminate the assumption. 
		
		In our proof, the assumption first comes up in the Jacobson-Morozov Lemma, both in the setting of Lie algebras, see \cite[Lemma \ref{I-lem:JM_basic}]{AppAGRCF} and in the semialgebraic setting, see \cite[Proposition \ref{I-prop:Jacobson_Morozov_real_closed}]{AppAGRCF}. Explicit calculations in $\mathfrak{su}_{1,2}$ show that \cite[Lemma \ref{I-lem:JM_basic}]{AppAGRCF} does not hold and similarly, \cite[Proposition \ref{I-prop:Jacobson_Morozov_real_closed}]{AppAGRCF} does not hold for $\operatorname{SU}(1,2)$. For given $\alpha \in \Sigma$ and $X\in \frakg_\alpha \oplus \frakg_{2\alpha}$, the task of the Jacobson-Morozov Lemma is to find $Y \in \frakg_{-\alpha}\oplus \frakg_{-2\alpha}$ and $H\in \fraka \subseteq \frakg_0$ such that $(X,Y,H)$ is an $\mathfrak{sl}_2$-triplet. While $Y$ can be found as a multiple of $\theta(X)$, there is no guarantee that $H :=[X,Y] \in \fraka$ when $\frakg_{2\alpha}\neq 0$.
		
		In the proof of Theorem \ref{thm:B_is_building}, the Jacobson-Morozov Lemma is used to associate to each $u \in (U_\alpha)_\F$ an element $m(u) \in \operatorname{Nor}_{G_\F}(A_\F)$ representing a reflection as in Proposition \ref{prop:BTmu_in_Wa}. We expect the following construction in the rank one subgroup $L_{\pm \alpha}$ to give a definition of $m(u)$ also when $\Sigma$ is not reduced. We use the Bruhat-decomposition of $L_{\pm\alpha}$.
		
		\begin{corollary}\label{cor:levi_Bruhat}
			Let $(B_{\alpha})_\F := (M_{\pm \alpha})_\F (A_{\pm \alpha})_\F (U_\alpha)_\F$. Then there is a representative $m \in (N_{\pm \alpha})_\F$ of the unique non-trivial element in $W_{\pm \alpha}$, so that
			$$
			(L_{\pm \alpha})_\F = (B_{\alpha})_\F \ \amalg \   (B_{\alpha})_\F \cdot m \cdot (B_{ \alpha})_\F.
			$$
		\end{corollary}
		If $u \in (U_\alpha)_\F$, let $u \in (B_{-\alpha})_\F \ \amalg \   (B_{-\alpha})_\F \cdot m \cdot (B_{ -\alpha})_\F$ and since $(U_{\alpha})_\F \cap (B_{-\alpha})_\F = \{\operatorname{Id}\}$, we can find $b,b' \in (B_{-\alpha})_\F$ such that $u = bmb'$. Writing $b = \bar{u}\bar{a}\bar{m}$ and $b' = \bar{m}'\bar{a}'\bar{u}'$ with $\bar{m}, \bar{m}' \in (M_{\pm \alpha})_\F, \bar{a},\bar{a}' \in (A_{\pm \alpha})_\F$ and $\bar{u}, \bar{u}'\in (U_{-\alpha})_\F$, we define
		$$
		m(u):= \bar{a} \bar{m} m \bar{m}' \bar{a}' = \bar{u}^{-1} u (\bar{u}')^{-1}  \in \operatorname{Nor}_{(L_{\pm\alpha})_\F }((A_{\pm \alpha})_\F ) \cap (U_{-\alpha})_\F (U_{\alpha})_\F (U_{-\alpha})_\F.
		$$
		We note that as a consequence of \cite[Lemma \ref{I-lem:levi_fixes_A}]{AppAGRCF}, $\operatorname{Nor}_{(L_{\pm\alpha})_\F }((A_{\pm \alpha})_\F ) \subseteq \operatorname{Nor}_{G_\F}(A_\F)$. In the case where $\Sigma$ is reduced, the uniqueness statement of Lemma \ref{lem:BTUUUm} then shows that $m(u)$ defined here agrees with $m(u)$ defined via the Jacobson-Morozov Lemma. This suggests that the definition of $m(u)$ as outlined here should be used when $\Sigma$ is not reduced. However, our proof also relies on the explicit description of $m(u)$ as $m(u) = u'uu''$ with $\varphi_{-\alpha}(u') = \varphi_{-\alpha}(u'') = -\varphi_\alpha(u)$, see the second part of Lemma \ref{lem:BTUUUm}, which follows from the explicit description of the root group valuation in Lemma \ref{lem:BTphiu_is_vt} relying heavily on the Jacobson-Morozov description. One way to allow a similar level of understanding of the root group valuations in the case of non-reduced root systems could be to do a case by case analysis of all rank one groups. 
		
		In the work of Bruhat Tits, the existence of such a $m(u)$ is related to axiom (DR4) of a \emph{donné radicielles} \cite[(6.1.1)]{BrTi}, which holds for even more general groups \cite[(6.1.3)c)]{BrTi} and \cite{BoTi65}. However, it is not clear how directly this is useful here, as we defined $\varphi_\alpha$ in terms of matrix entries and we do not work with a Chevalley-basis.

%% file: local_global.tex
\section{Residue building and building at infinity}\label{sec:assoc-buildings}

In this section, we discuss the residue building and the spherical building at infinity of $\B$.
Let $(\B,\Fun)$ be an affine $\Lambda$-building. Recall that the \emph{Weyl chamber associated to a basis} $\Delta \subseteq \Phi$ is the set
$$
C_\Delta = \left\{ x \in \mathbb{A} \colon \alpha(x) \geq 0 \text{ for all } \alpha \in \Delta \right\}.
$$ 
For $J\subseteq \Delta$, the subsets
$$
C_J = \left\{x \in \A \colon \begin{array}{ll}
\alpha(x) \geq 0 & \text{for }  \alpha \in J \\
\alpha(x) = 0 & \text{for } \alpha \in \Delta \setminus J
\end{array}\right\}
$$
are called \emph{Weyl simplices}. For a chart $f \in \Fun$, $f(C_\Delta)$ and $f(C_J)$ are also called \emph{sectors} and \emph{sector facets based at} $f(0)$. We say that two sector facets $f(C_J)$, $f'(C_J')$ both based at some point $p \in \B$ \emph{share the same germ} if there exists $\varepsilon \in \Lambda_{>0}$ such that $f(C_J \cap B_{\varepsilon}(0)) = f'(C_J' \cap B_{\varepsilon}(0))$. The equivalence class of all sector facets sharing the same germ as some sector facet $s$ is called the \emph{germ of} $s$ and denoted by $\Delta_p s$. The set of all germs based at $p$ is a spherical building \cite[Theorem 5.17]{Sch09}, called the \emph{residue building $\Delta_p\B$}. Every apartment $A\subseteq \B$ containing $p$ determines an apartment $\Delta_p A$ in $\Delta_p \B$.

On the large scale, we call two sector facets $s,s'$ \emph{parallel} if there exists $s=s_1, s_2 , \ldots , s_n=s'$ sector facets such that $s_i=f_i(C_J)$ and $s_{i+1}=f_i \circ t_i (C_J)$ for some charts $f_i\in \Fun$ and translations $t_i\in W_a$. The parallel class of a sector facet $s$ is called \emph{simplex at infinity} and is denoted by $\partial_\infty s$. The set of simplices at infinity is a spherical building \cite[Section 3.4]{Ben1}, called the \emph{building at infinity} $\partial_{\infty}\B$. There is a one-to-one correspondence of apartments $A\subseteq \B$ and apartments $\partial_\infty A$ in $\partial_\infty \B$.

\subsection{Residue building}

Let $\B$ be the affine $\Lambda$-building defined in Section \ref{sec:building_def}, $C_0 \subseteq \A$ the fundamental Weyl chamber and $k \coloneqq O / J$ the residue field as defined in Section \ref{sec:residue_field}. Recall that a group acts \emph{strongly transitively} on the spherical building $\Delta_0\B$, if it acts transitively on pairs $(\Delta_os,\Delta_o A)$, where $s$ is a sector based at $o$ contained in an apartment $A$ in $\B$. In this section, we will prove the following characterization of the residue building.

\begin{theorem}\label{thm:germ_building}
	The group $G_k$ acts strongly transitively on the residue building $\Delta_o\B$, which can therefore be identified with the spherical building associated to the BN-pair $B=B_k = U_kA_kM_k =  \operatorname{Stab}_{G_k}(\Delta_oC_0)$ and $N=\operatorname{Nor}_{G_k}(A_k) = A_kN_k = \operatorname{Stab}_{G_k}(\Delta_o \mathbb{A})$.
\end{theorem}
By \cite[Theorem 4.3.7]{EnPr05}, the residue field $k$ of a real closed field $\F$ is real closed. 
For any semialgebraic set $V_\F \subseteq \F^{n\times n}$, the entrywise reduction $\pi \colon V_\F(O) \coloneq V_\F \cap O^{n \times n} \to V_k$ is well-defined and surjective \cite[Proposition 2.18]{BIPP23arxiv}, in particular we obtain a surjective group homomorphism
\begin{align*}
\pi \colon	G_\F(O) \twoheadrightarrow G_k.
\end{align*}
By Theorem \ref{thm:stab}, $G_\F(O) = \operatorname{Stab}_{G_\F}(o)$ acts on the set of sector facets based at $o$. Since it acts by isometries, it preserves the equivalence relation of sharing the same germ, hence $G_\F(O)$ acts on the residue building $\Delta_o\B$. The following Lemma implies that this action descends to an action of $G_k$ on $\Delta_o \B$.

\begin{lemma}
	The subgroup $\operatorname{ker}( G_\F(O) \twoheadrightarrow G_k )$ acts trivially on $\Delta_o\B$.
\end{lemma}
\begin{proof}
	Let $g\in \ker( G_\F(O) \twoheadrightarrow G_k )$, which means that the matrix $g-\operatorname{Id}$ has entries in the maximal ideal $J$. In the language of Section \ref{sec:matrixvaluation}, where we defined $(-v)(M) \coloneqq \max_{ij}\{(-v)(M_{ij})\}$ for $M\in \F^{n\times n}$ this means $(-v)(g-\operatorname{Id}) < 0$.
	Now let $C' \in \Q_{\geq 0}$ be the constant from Lemma \ref{lem:matrixval-bound} such that whenever $a\in A_\F$ satisfies $(-v)(\chi_\delta(a)) \leq \lambda $ for all $\delta \in \Delta$ and some $\lambda \in \Lambda_{\geq 0}$, then $(-v)(a) \leq C'\ \lambda$. Let $\lambda \in \Lambda_{>0}$ be such that  
	$$
	(-v)(g-\operatorname{Id}) \leq - 2 C' \lambda.
	$$ 
	We will show that $g$ fixes all points $p\in \B$ with $d(p,o) \leq \lambda$, which then shows the result.
	
	Let $a\in A_\F$, $k \in K_\F$ so that $p = ka.o$ by the Cartan decomposition \cite[Theorem \ref{I-thm:KAK}]{AppAGRCF}. We have
	\begin{align*}
		\lambda \geq  d(p.o,o) &= d(a.o,o) = \sum_{\alpha \in \Sigma} \left\lvert (-v)(\chi_\alpha(a))\right\rvert ,
	\end{align*} 
so in particular $(-v)(\chi_\delta(a)) \leq \lambda$ for all $\delta \in \Delta$. Thus $(-v)(a) \leq C' \lambda$. By Lemma \ref{lem:matrixvaluation}, 
$$
(-v)(ka) \leq (-v)(k) + (-v)(a) = (-v)(a) \leq C' \lambda
$$
and since $ka$ has determinant $1$, the inverse is given by the adjugate matrix whose entries have the same valuations, so
$$
(-v)\left((ka)^{-1}\right) = (-v)(ka)  \leq C'\lambda.
$$
Now $g$ fixes $p$ if and only if $gka.o = ka.o$, or equivalently $(ka)^{-1}gka \in G_\F(O)$ by Theorem \ref{thm:stab}, and 
\begin{align*}
	(-v)\left( (ka)^{-1}gka \right) &= (-v)\left( (ka)^{-1}(g - \operatorname{Id}) ka + (ka)^{-1}\operatorname{Id}ka   \right) \\
	&\leq \max \left\{ (-v)\left( (ka)^{-1}(g - \operatorname{Id}) ka   \right) , (-v)(\operatorname{Id})\right\} \\
	&\leq \max \left\{ (-v)\left((ka)^{-1}\right) + (-v)(g-\operatorname{Id})  + (-v)(ka) , 0 \right\} \\
	& \leq  \max \left\{ C'\lambda -2C'\lambda + C'\lambda , 0 \right\} = 0
\end{align*}
where we used Lemma \ref{lem:matrixvaluation}.
\end{proof}

The following Lemma will help us with the computation of the stabilizers.

\begin{lemma}\label{lem:neg_kernel}
Let $\Sigma_{>0}$ be the positive roots with respect to some order. If $u\in U_\F^+$ stabilizes $o$ and a point $a.o \in \A$ with $\chi_{\alpha}(a) \in J$, for all $\alpha \in \Sigma_{>0}$, then $u \in \operatorname{ker}(\pi \colon G_\F(O) \twoheadrightarrow G_k)$.
\end{lemma}
\begin{proof}
	Let $\Omega = \{o,a.o\}$, then $U_\Omega^{+} = \langle u' \in U_{\alpha,\Omega} \colon \alpha > 0 \rangle$ by Proposition \ref{prop:BTUOmega}. For all $\alpha\in \Sigma_{>0}$, $u' \in U_{\alpha,\Omega}$, we have
	$
	\varphi_{\alpha}(u') \leq (-v)(\chi_{\alpha}(a)) < 0,
	$
	in particular $(u'-\Id)_{ij} \in J$ for all $i,j$ by Lemma \ref{lem:stab_Ualpha}. Since $u'u''-\Id = (u'-\Id)(u''-\Id) + (u' - \Id) + (u'' - \Id)$, we have $(u-\Id)_{ij} \in J$ for all $u\in U^+$. In particular $\pi(u) = \Id \in G_k$, which is what we claimed.
\end{proof}

\begin{proof}[Proof of Theorem \ref{thm:germ_building}]
	Let $s,s'$ be sectors based at $o$ contained in apartments $A,A' \subseteq \B$. To check that $G_k$ acts strongly transitively on $\Delta_o \B$, it suffices to find an element $h\in G_k$ with $h.\Delta_o s = \Delta_o s'$ and $h.\Delta_o A = \Delta_o A'$. 
	By \cite{BeScSt}, the following axiom holds for affine $\Lambda$-buildings.
	\begin{itemize}
		\item [(GG)] Any two germs of sectors based at the same vertex are contained in a common apartment.
	\end{itemize}
    We apply (GG) to find an apartment $A'' \subseteq \B$ that contains the germs $\Delta_o s$ and $\Delta_o s'$. 
    Since $G_F$ acts transitively on apartments in $\B$, $G_\F(O)$ acts transitively on apartments containing $o$. Therefore $G_\F(O)$ and $G_k$ also act transitively on germs of apartments in $\Delta_o \B$.
    Let $g \in G_k$ be such that $g.\Delta_o A = \Delta_o A''$ and up to precomposing with an element in $N_k$, we may assume that $g$ stabilizes $\Delta_o s$. Similarly, let $g'\in G_k$ such that $g'.\Delta_oA'' = \Delta_o A'$ with $g'.\Delta_o s' = \Delta_o s'$. Let $g''\in G_k$ be an element that stabilizes the apartment $\Delta_oA''$, but such that $g''.\Delta_os = \Delta_o s'$. Note that $ h \coloneq g' g'' g \in G_k$ satisfies $h.\Delta_o s = \Delta_o s'$ and $h.\Delta_o A = \Delta_o A''$, proving that $G_k$ acts strongly transitively on $\Delta_o \B$. 
    
    By \cite[(5.2), (5.3) and Remark 2]{Ron} the strongly transitive action of $G_k$ determines the BN-pair $B= \operatorname{Stab}_{G_k}(\Delta_oC_0)$ and $N=\operatorname{Stab}_{G_k}(\Delta_o \mathbb{A})$, which in turn determines a building, which is isomorphic to $\Delta_o \B$.
    
    We now compute $\operatorname{Stab}_{G_k}(\Delta_0s_0)$ and $\operatorname{Stab}_{G_k}(\Delta_{o}\A)$. Let $\overline{g} \in \operatorname{Stab}_{G_k}(\Delta_os_0)$. By \cite[Proposition 2.18]{BIPP23arxiv}, there is a $g\in \operatorname{Stab}_{G_\F(O)}(\Delta_0s_0)$ with $\pi(g)=\overline{g}$. There exists an $\varepsilon \in \Lambda_{>0}$ such that $g.f_0(C_0 \cap B_{\varepsilon}(0)) = f_0(C_0 \cap B_{\varepsilon}(0))$. Actually by Axiom (A2) and the fact that $W_s$ acts simply transitively on Weyl-chambers, $g$ stabilizes $\Omega \coloneqq f_0(C_0 \cap B_{\varepsilon}(0))$ pointwise. 
    By Theorem \ref{thm:BTstab}, $g \in \operatorname{Stab}_{G_\F(O)}(\Omega) = U_\Omega^{-}U_\Omega^+N_\Omega$. We note that by Proposition \ref{thm:apt}, there is a point $p=a.o \in \Omega$ with $(-v)(\chi_{\alpha}(a)) > 0$, so we can apply Lemma \ref{lem:neg_kernel} with respect to the negative order on $\Sigma$ to see that $U_\Omega^- \subseteq \ker(\pi \colon G_\F(O) \twoheadrightarrow G_k)$. Writing $g=u^{-}u^+n$ for $u^- \in U_\Omega^-$, $u^+ \in U_\Omega^+$ and $n \in N_\Omega$, we notice that $n\in A_\F(O)M_{\F}$ by Theorem \ref{thm:stab_A}. We have $\overline{g}=\pi(u^{-})\pi(u^+)\pi(n) = \pi(u^+)\pi(n) \in U_k^+A_kM_k$. The other direction $B_k = U_k^+A_kM_k \subseteq \operatorname{Stab}_{G_k}(\Delta_os_0)$ follows from Theorem \ref{thm:NO_fixes_s0}, which states that $B_\F(O) = U_\F(O)A_\F(O)M_\F = \operatorname{Stab}_{G_{\F}}(s_0)$.
    
    If now $\overline{g} \in \operatorname{Stab}_{G_k}(\Delta_o\A)$, there exists $g\in \operatorname{Stab}_{G_\F(O)}(\Delta_o\A)$ with $\pi(g)=\overline{g}$. Let $\varepsilon \in \Lambda_{>0}$ such that $g.f_0(B_{\varepsilon}(0)) = f_0(B_{\varepsilon}(0))$. By axiom (A2) and Proposition \ref{thm:stab'_A} there is some $n\in A_\F(O)N_\F$ such that $n^{-1}g$ stabilises $\Omega \coloneqq  f_0(B_{\varepsilon}(0))$ pointwise. Then by Theorem \ref{thm:BTstab}, $n^{-1}g \in U_\Omega^-U_\Omega^+N_\Omega$ where we can apply Lemma \ref{lem:neg_kernel} to both the positive and negative order on $\Sigma$ to see that $\pi (U_\Omega^+) = \pi (U_\Omega^{-}) = \{\operatorname{Id}\} \subseteq G_k$. Thus $\overline{g} \in \pi(N_\Omega)\pi(n) \subseteq A_k N_k = \operatorname{Nor}_{G_k}(A_k)$, see Proposition \ref{thm:stab'_A}. The other direction $\operatorname{Nor}_{G_k}(A_k) = A_kN_k \subseteq \operatorname{Stab}_{G_k}(\Delta_o\A)$ follows directly since $A_\F(O)N_\F$ acts on $\A$ preserving $o$. 
\end{proof}

\subsection{Building at infinity}

Let $\B$ be the affine $\Lambda$-building defined in Section \ref{sec:building_def}.
In this section, we will prove the following characterization of the building at infinity.

\begin{theorem}\label{thm:building_at_infty}
	The group $G_\F$ acts strongly transitively on the building at infinity $\partial_\infty\B$, which can therefore be identified with the spherical building associated to the BN-pair $B=B_\F = U_\F A_\F M_\F =  \operatorname{Stab}_{G_\F}(\partial_\infty s_0)$ and $N=\operatorname{Nor}_{G_\F}(A_\F) = A_\F N_\F = \operatorname{Stab}_{G_\F}(\partial_\infty \mathbb{A})$.
\end{theorem}
\begin{proof}
	Since $G_\F$ acts both on $\B$, but also on $\Fun$, the action preserves parallelism. Hence $G_\F$ acts on the building at infinity $\partial_\infty\B$. Let $s,s'$ be sectors contained in apartments $A,A'$ such that $( \partial_\infty s, \partial_\infty A), ( \partial_\infty s', \partial_\infty A') $ are pairs at infinity such that $ \partial_\infty s \subseteq  \partial_\infty A$ and $ \partial_\infty s' \subseteq  \partial_\infty A'$. 
	By the definition of apartments, $G_\F$ acts transitively on apartments in $\B$, hence also on apartments in $\partial_\infty \B$, so let $g\in G_\F$ such that $g.A = A'$. Now $g^{-1}.s'$ is some Weyl chamber in $A$. By the definition of Weyl chambers, $W_a$ acts transitively on Weyl chambers, so there exists an element $w\in W_a$ such that $w.s = g^{-1}.s'$. By axiom (A1) there exists $g' \in G_\F$ such that $g'.f = g.f \circ w$, so $g'.A = A'$ and $g'.s = s'$. This shows that $G_\F$ acts strongly transitively on $\partial_\infty \B$.
	
	By \cite[(5.2), (5.3) and Remark 2]{Ron} the strongly transitive action of $G_\F$ determines the BN-pair $B= \operatorname{Stab}_{G_\F}(\partial_\infty s_0)$ and $N=\operatorname{Stab}_{G_\F}(\partial_\infty \mathbb{A})$, which in turn determines a building, which is isomorphic to $\partial_\infty \B$. 
	
	We now compute $\operatorname{Stab}_{G_\F}(\partial_\infty s_0)$ and $\operatorname{Stab}_{G_\F}(\partial_\infty \A)$. Let $g\in \operatorname{Stab}_{G_\F}(\partial_\infty s_0)$. This means that there are sectors $s_0, s_1, \ldots, s_k \eqqcolon g.s_0$ and charts $f_i \in \Fun$ and translations $t_i\in T=\A$ for $i\in \{0,1, 2, \ldots , k-1\}$ with $s_i=f_i(C_0)$ and $s_{i+1} = f_i(t_i(C_0))$. Let $a_i\in A_\F$ with $a_i.f_0 = f_0\circ t_i$ and let $g\in G_\F$ with $f_i = g_i.f_0$. Notice $g_0=\Id \in B_\F$ and $g_k = g$. We prove by induction that if $g_i\in B_\F$, then also $g_{i+1}\in B_\F$: we have
	$$
	g_{i+1}.f_0(C_0) = s_{i+1} = g_i.f_0(t_i(C_0)) = g_i a_i .f_0(C_0),
	$$
	so $g_{i+1}^{-1} g_i a_i \in \operatorname{Stab}_{G_\F}(f_0(C_0)) \subseteq B_\F$, by Theorem \ref{thm:NO_fixes_s0} (note that the pointwise and setwise stabilizers of $f_0(C_0)$ coincide). 
	Since $a_i \in A_\F \subseteq B_\F$ and $g_i\in B_\F$ by the induction assumption, we also have $g_{i+1} \in B_\F$.
	The other direction $B_\F = U_\F A_\F M_\F \subseteq \operatorname{Stab}_{G_\F}(\partial_\infty s_0)$ follows from Theorem \ref{thm:NO_fixes_s0} and Proposition \ref{prop:exists_a_NO}. 

Now let $g\in \operatorname{Stab}_{G_\F}(\partial_\infty \A)$. Then $\partial_\infty(g.\A) = g.\partial_\infty(\A) = \partial_\infty (\A)$, which is equivalent to $g.\A = \A$, since there is a one-to-one correspondence of apartments in $\B$ and in $\partial_\infty \B$. Then by Proposition \ref{thm:stab'_A},  $\operatorname{Stab}_{G_\F}(\partial_\infty \A) = \operatorname{Stab}_{G_\F}(\A) = \operatorname{Nor}_{G_\F}(A_\F) = A_\F N_\F$. 
\end{proof}

%% file: appendixSLn.tex
\section{Appendix: The building for \texorpdfstring{$\operatorname{SL}(n,\F)$}{SL(n,F)}}\label{sec:appendixSLn}

To obtain Theorem \ref{thm:B_is_building}, that $\B$ is an affine $\Lambda$-building, a relatively large amount of effort goes into proving axiom (A2). The development of the theory following \cite{BrTi} in Subsections \ref{sec:BT_root_groups}, \ref{sec:BT_rank_1} and \ref{sec:BT_higher_rank} is not needed when the group $G$ is well understood. In this appendix, we give an alternative proof of axiom (A2) in the case where $G_\F = \operatorname{SL}_n(\F)$. The proof still relies on the general theory developed for semialgebraic groups in \cite[Sections \ref{I-sec:split_tori} and \ref{I-sec:decompositions}]{AppAGRCF}, but is significantly shorter.

Let $\F$ be a non-Archimedean real closed field with order compatible valuation $(-v) \colon \F \to \Lambda \cup \{ -\infty\}$ and valuation ring $O = \{a \in \F \colon (-v)(a) \leq 0\}$. We consider the semisimple linear algebraic group $G= \operatorname{SL}_n$ with maximal $\K$-split torus
$$
S = \left\{ \begin{pmatrix}
	\star & & \\ & \ddots & \\ && \star
\end{pmatrix} \in \operatorname{SL}_n \right\}.
$$ 
Then the groups showing up in the various decompositions of \cite[Section \ref{I-sec:decompositions}]{AppAGRCF} are given by 
\begin{align*}
	K_\F &= \operatorname{SO}_n(\F)\\
	A_{\F} &= \left\{ a = (a_{ij}) \in S_{\F} \colon a_{ii}>0   \right\} \\
	U_{\F} &= \left\{g=(g_{ij}) \in \operatorname{SL}_{n}(\F) \colon  g_{ii}= 1, \, g_{ij} = 0 \text{ for } i>j  \right\} \\
	N_{\F} &= \left\{ \text{ permutation matrices with entries in }\pm 1 \right\} \\
	M_{\F} & = \left\{ a=(a_{ij}) \in S_{\F} \colon a_{ii} \in \{ \pm 1\}    \right\} \\
	B_{\F} &= \left\{ g = (g_{ij}) \in \operatorname{SL}_n(\F) \colon g_{ij} = 0 \text{ for } i > j \right\}.
\end{align*}
Then $\fraka = \operatorname{Lie}(A_\R) = \{ H \in \R^{n\times n} \colon \operatorname{tr}(H) = 0 \text{ and $H$ is diagonal }  \}$ and the root system $\Sigma$ associated with $\operatorname{SL}(n,\R)$ is given by $\Sigma = \{ \alpha_{ij} \in \fraka^\star \colon i\neq j \in \{1, \ldots , n\} \}$ for the roots
\begin{align*}
	\alpha_{ij} \colon \fraka &\to \R \\
	H& \mapsto H_{ii} - H_{jj}.
\end{align*}
The spherical Weyl group $W_s$ is then isomorphic to the symmetric group $S_n$ on $n$ letters. If we choose the ordered basis $\Delta = \{ \alpha_{12}, \ldots , \alpha_{(n-1)n}\}$, we obtain the positive Weyl chamber
$$
A_\F^+ = \{  \operatorname{Diag}(a_1, \ldots , a_n) \in A_\F \colon a_1 \geq a_2 \geq \ldots  \geq a_n  \}.
$$
As in Section \ref{sec:building_def} we define the non-standard symmetric space
$$
P_{1}(n,\F) := \left\{ A \in \F^{n\times n} \colon A\tran = A, \ \det(A) = 1 ,\ A \text{ is positive definite } \right\}
$$
and note that $\operatorname{SL}(n,\F)$ acts transitively on $P_1(n,\F)$. The multiplicative norm \begin{align*}
	N_\F \colon A_\F &\to \F_{\geq 1} \\
	\operatorname{Diag}(a_1, \ldots , a_n) &\mapsto \prod_{i\neq j} \max \left\{ \frac{a_i}{a_j}, \frac{a_j}{a_i}   \right\}
\end{align*}
then gives a $G_\F$-invariant $\Lambda$-pseudo distance $d  = (-v) \circ N_\F \circ \delta_\F$ which allows us to define the $\Lambda$-metric space $\B := P_1(n,\F)/\!\!\sim$. We endow $\B$ with the apartment structure $\Fun = \{ g.f_0 \colon g \in \operatorname{SL}(n,\F) \}$, where $f_0 \colon \A  \to \B$ is the inclusion of the apartment $\A := A_\F.o$ for the basepoint $o = [\operatorname{Id}] \in \B$. The following is a special case of Theorem \ref{thm:B_is_building}.
\begin{theorem}
	The $\Lambda$-metric space $\B$ is an affine $\Lambda$-building of type $\A = \A(\Sigma^\vee,\Lambda,\Lambda^n)$, where $\Sigma^\vee$ is a root system of type $\operatorname{A}_{n-1}$.
\end{theorem} 
The axioms (A1), (A3) and (TI) follow as described in Section \ref{sec:axiomsA1A3TI}. We will give a hands on proof of axiom (A2) that only relies on Theorem \ref{thm:stab} about the stabilizer of $o \in \B$. Axiom (A4) can then be proved as in Section \ref{sec:A4}, relying on (A2) and the fact that every $u \in U_\F$ stabilizes a sector, which can be proven as in Section \ref{sec:Wconvexity_for_U}, or looking at matrix entries directly. Finally, axiom (EC) relies on (A2) and the fact that Jacobson-Morozov morphisms can be found, which can be seen for $\operatorname{SL}(2,\F)$ explicitly. 

We first look at the axiom
\begin{enumerate}
	\item [(A2)] For all $f_1,f_2 \in \Fun$, if $f_1(\A)\cap f_2(\A)\neq \emptyset$, then the set $\Omega := f_2^{-1}\circ f_1(\A)$ is a $W_a$-convex set and there exists $w \in W_a$ such that $f_2|_\Omega = f_1\circ w |_\Omega$. 
\end{enumerate}
in the special case where $f_2=f_0$ and $f_1=g.f_0$ for some $g\in G_\F$ with $g.o=o$. The set $\Omega$ is defined by $f_0(\Omega)=f_0(\A)\cap g.f_0(\A)$. Let $B\subset A_\F$ be the set of all elements $a\in A_\F$ with $g.a.o\in f_0(\Omega)$. Note that elements in $B$ are diagonal matrices whose diagonal entries are positive elements in the real closed field $\F$, we can thus take their $n$-th roots. We would like to prove that $\Omega$ is a finite intersection of closed affine half-apartments. The reason we define $W_a$-convexity this way is that general convex combinations of points in the apartment $\A$ are not possible, since we do not have a vector space-structure on $\A$. Some convex combinations however are still possible and $\Omega$ contains them.  

\begin{lemma}
	\label{lem:A2conv}
	For all $a,a' \in B$ and all $n,m \in \N$, $
	\sqrt[n+m]{a^n a'^m} \in B$.
\end{lemma}
\begin{proof}
	Let $b,b' \in A_\F$ such that $g.a.o = b.o$ and $g.a'.o = b'.o$. By Theorem \ref{thm:stab} $b^{-1}ga, (b')^{-1}g a' \in G_\F(O)$. We can now exploit the explicit structure of $A_\F$ to write in coordinates
	$$
	g_{ij}\frac{a_j}{b_i} \in O \quad \text{and} \quad g_{ij}\frac{a'_j}{b'_i} \in O.
	$$
	Thus also 
	$$
	g_{ij}^{n+m} \frac{a_j^n {a'}_j^m}{b_i^n {b'}_i^m} \in O \quad \text{and so}\quad
	g_{ij} \frac{\sqrt[n+m]{a_j^n {a'}_j^m}}{\sqrt[n+m]{b_i^n {b'}_i^m}} \in O.
	$$
	We note that if $a = \operatorname{Diag}\left(a_1, \ldots , a_n\right) \in A_\F$, then also 
	$$
	\sqrt[n+m]{a} :=\operatorname{Diag}(\sqrt[n+m]{a_1} , \ldots , \sqrt[n+m]{a_n}  ) \in A_\F,
	$$
	since this is a first order statement that is true over the reals. Then
	$$
	g.\sqrt[n+m]{a^n a'^m}.o = \sqrt[n+m]{b^n b'^m}.o \in \A
	$$
completes the proof.
\end{proof}

We now introduce a notion of regularity. Elements $a.o \in f_0(\A)$ with $a\in A_\F$ are represented by diagonal matrices with entries $a_1, \ldots, a_n \in \F$. Consider the amount of distinct entries $(-v)(a_1), \ldots ,(-v)(a_n)$ as elements in $\Lambda = (-v)(\F_{>0})$. 
Let $\pmb{a}\in B$ be a maximal element with respect to the amount of distinct entries. The intuition is that elements with $a_i=a_j$ for $i\neq j$ lie on a wall. The more walls they lie on, the less regular they are. 
Other elements in $B$ may have as many distinct entries as $\pmb{a}$ but the next Lemma states that those are the same ones as the ones for $\pmb{a}$. 

\begin{lemma}
	\label{lem:A2regular}
	Let $b \in B$. If $(-v)(b_i)\neq(-v)(b_j)$ for some $i,j$, then $(-v)(\pmb{a}_i)\neq(-v)(\pmb{a}_j)$. Equivalently if $(-v)(\pmb{a}_i) =(-v)(\pmb{a}_j)$, then $(-v)(b_i) = (-v)(b_j)$.
\end{lemma}
\begin{proof}
	For any element $b \in B$, we define $I_b = \{ (i,j) \colon (-v)(b_i) \neq (-v)(b_j) \}$. The maximality of $\pmb{a}$ means that $|I_b|\leq |I_{\pmb{a}} |$ for all $b \in B$. We want to prove that $I_b \subseteq I_{\pmb{a}}$ for all $b\in B$. We note that there are certainly less than $n^2$ elements in $I_{\pmb{a}}$, and define for $k \in \{0 , 1 , \ldots , n^2  \}$ the element
	$$
	c(k)=\sqrt[n^2]{\pmb{a}^{k} b^{n^2-k}} \in B 
	$$
	which is in $B$ by Lemma \ref{lem:A2conv}. Let $(i,j) \in I_{\pmb{a}} \cup I_b$. Thus $(-v)(\pmb{a_i}/\pmb{a_j}) \neq 0$ or $(-v)(b_i/b_j) \neq 0$. We define the map
	$$
	k \mapsto (-v)( c(k)_i / c(k)_j ) \in \Lambda
	$$
	which is either constant, but not $=0$, or strictly monotonous ascending or descending. Either way, there is at most one value of $k$ for which the map is $0$. Let $k_{ij}$ be this value, if it exists. 
	
	Now pick a $k$ which does not appear as $k_{ij}$ for any $(i,j) \in I_{\pmb{a}} \cup I_b$. This means that $c(k)_i \neq c(k)_j$ for all $(i,j) \in I_{\pmb{a}} \cup I_b$. We conclude that $I_{\pmb{a}} \cup I_b \subseteq I_{c(k)}$. But since $\pmb{a}$ is maximal, we have $I_{\pmb{a}} = I_{c(k)}$ and $I_b \subseteq I_{\pmb{a}}$.
\end{proof}

Recall that $O\subsetneqq \F$ is a strict subring of a non-Archimedean real closed field $\F$. Denote the units of $O$ by $O^\times$. The following is a standard fact about valuation rings.
\begin{lemma}
	In any valuation ring $O$, the set of nonunits $O\setminus O^\times$ is an ideal.
\end{lemma}
\begin{proof}
	Let $x,y\in O\setminus O^\times, r \in O $. Since $x$ is not a unit, $x^{-1}=r\cdot (rx)^{-1} \not\in O$. Therefore $(rx)^{-1}\not\in O$, so $rx \in O \setminus O^\times$. If $x$ and $y$ are nonzero, then $x/y\in O$ or $y/x\in O$, since $O$ is a valuation ring. Writing $x+y=y\cdot(1+x/y)=x\cdot(1+y/x)$, we see that also $x+y\in O\setminus O^\times$.
\end{proof}
The determinant of a matrix is a polynomial of its entries. By the previous Lemma, if all entries of a matrix are in $O\setminus O^\times$, then so is its determinant. We conclude that if the determinant of a matrix with entries in $O$ is $1$, then at least one entry has to be a unit. Actually, even more holds.
\begin{lemma}
	\label{lem:A2unit}
	Let $g$ be a matrix with entries in $O$ and $\operatorname{det}(g)=1$. Then in every row and every column there exists at least one entry that is a unit. In fact, there is a permutation $\sigma$, such that $g_{i\sigma(i)}\in O^\times$.
\end{lemma}
\begin{proof}
	Consider the $i$'th row $(g_{i1}, g_{i2}, \ldots, g_{in})$ of $g$. Using Laplace's formula for the determinant
	$$
	\det (g) = \sum_{j=1}^n (-1)^{i+j} g_{ij} \cdot \det M_{ij},
	$$
	where $M_{ij}$ are the minor matrices after deleting row $i$ and column $j$, we see that since $\det (g)=1 \in O^\times$, at least one of the $g_{ij}$ for $j=1, \ldots , n$ has to be a unit. Equivalently there has to be a unit in every column. We can extract a permutation by alternating the row and column argument. 
\end{proof}

We will use Lemma \ref{lem:A2regular} and Lemma \ref{lem:A2unit} to prove a certain rigidity of the action of $G_\F$ on $f_0(\A)$, namely when an isometry $g\in G_\F$ fixes the identity, then its action on $f_0(\A)$ can be described by an element in the spherical Weyl group $W_s$. The apartment $\A$ itself could have many other symmetries, but the isometry can only be extended to all of $\B$, when it is one of the finitely many elements in $W_s$. Consider the fundamental sector
$$
s_0 = \{ a.o \in \A \colon a_1 \geq \ldots \geq a_n \}
$$
based at $o$. The sectors in $f_0(\A)$ based at $o$ are fundamental domains of the action of $W_s$, see for instance \cite[Chapter 10.3]{Hum2}. So two points that lie inside a common sector based at $o$ can only be related by an element in $W_s$ if they are the same. This implies a first restricted version of axiom (A2). 

\begin{proposition}
	\label{pro:A2id}
	Let as before $\pmb{a}\in B$ be maximal with respect to the amount of distinct entries and $g \in G_\F$ such that $g.o=o$. If both $\pmb{a}.o$ and $g.\pmb{a}.o$ lie in the same sector based at $o$, then $\pmb{a}.o=g.\pmb{a}.o$ and moreover the action of $g$ is the identity on the set $\Omega := f_0^{-1}\left( f_0(\A)\cap g.f_0(\A) \right)$, i.e. $f_0|_\Omega = g.f_0|_\Omega$.
\end{proposition}
\begin{proof} Let $\pmb{a} = \diag{\pmb{a}}$ and $\pmb{b} = \diag{\pmb{b}} \in A_\F$ such that $g.\pmb{a}.o = \pmb{b}.o$. Since $\pmb{a}.o$ and $\pmb{b}.o$ lie in the same sector based at $o$, their entries satisfy the same ordering. Formally we can find a permutation $\rho$, such that 
	\begin{align}\label{eq:ord}
		(-v)(\pmb{a}_{\rho(1)}) & \leq (-v)(\pmb{a}_{\rho(2)}) \leq \ldots \leq (-v)(\pmb{a}_{\rho(n)}) \quad \text{and} \\
		(-v)(\pmb{b}_{\rho(1)}) & \leq (-v)(\pmb{b}_{\rho(2)}) \leq \ldots \leq (-v)(\pmb{b}_{\rho(n)}). \nonumber
	\end{align}
	From $g.o=o$ follows that $g\in G_\F(O)$ by Proposition \ref{thm:stab}. By the previous Lemma \ref{lem:A2unit} we get a permutation $\sigma$, such that $g_{i\sigma(i)}\in O^\times$. In coefficients, $\pmb{b}^{-1}g\pmb{a} \in G_\F(O)$ implies $\pmb{b}_i^{-1}g_{i\sigma(i)}\pmb{a}_{\sigma(i)} \in O$ and therefore also $\pmb{b}_i^{-1}\pmb{a}_{\sigma(i)} \in O$. This means that $\operatorname{Diag}( \pmb{a}_{\sigma(1)}, \ldots , \pmb{a}_{\sigma(n)}).o=\diag{\pmb{b}}.o$. So although the elements $\pmb{a}_{\sigma(i)}$ and $\pmb{b}_{i}$ may not be exactly the same, they satisfy $(-v)(\pmb{a}_{\sigma(i)}) = (-v)(\pmb{b}_i)$. So we can order the entries as
	\begin{alignat}{3}
		\label{eq:ineq}
		(-v)(\pmb{b}_{\rho(1)}) & \leq (-v)(\pmb{b}_{\rho(2)})&& \leq \ldots \leq (-v)(\pmb{b}_{\rho(n)}) \quad \text{and thus} \\
		\rotatebox{90}{$=$}\ \ \ \ & \quad \ \quad \ \rotatebox{90}{$=$} && \quad \quad \quad \quad \quad \ \rotatebox{90}{$=$} \nonumber
		\\
		(-v)(\pmb{a}_{\sigma(\rho(1))}) & \leq (-v)(\pmb{a}_{\sigma(\rho(2))})&& \leq \ldots \leq (-v)(\pmb{a}_{\sigma(\rho(n))}). \nonumber
	\end{alignat}
	But we already have a decreasing ordering of the diagonal entries of $\pmb{a}$ in equation (\ref{eq:ord}), so they have to be the same, i.e. for all $i$, $(-v)(\pmb{a}_{\rho(i)})=(-v)(\pmb{a}_{\sigma(\rho(i))})$ and thus $(-v)(\pmb{a}_i) = (-v)(\pmb{a}_{\sigma(i)})=(-v)(\pmb{b}_i) \in \Lambda$. We now know that $\pmb{a}.o=\pmb{b}.o$, but we still need to show that $g$ is the identity on all of $\Omega$. 
	
	If all the inequalities in (\ref{eq:ineq}) were strict, then $\sigma$ would necessarily be the identity. For indices $i,j$ where there is an equality, $\sigma$ could change these entries. However by Lemma \ref{lem:A2regular} we know that whenever $(-v)(\pmb{a}_i)=(-v)(\pmb{a}_j)$, then also $(-v)(a_i) = (-v)(a_j)$ for any $a=\diag{a} \in B$. So for all $i$, $(-v)(a_{\sigma(i)})=(-v)(a_i)$. In fact we see
	$$
	g_{i\sigma(i)}\frac{a_{\sigma(i)}}{a_i}\in g_{i\sigma(i)} O^\times \subseteq O
	$$
	and thus $g.a.o=a.o$ for all $a \in B$, so $f_0|_\Omega = g.f_0|_\Omega$.
\end{proof}
We will now allow slightly more general points $\pmb{a}.o$ and $\pmb{b}.o$.
\begin{proposition}
	\label{pro:A2W}
	Let $g\in G_\F$ with $g.o = o \in \B$, then there exists an element $w \in W_s$ such that $f_0|_\Omega = g.f_0 \circ w|_\Omega$, where $\Omega := f_0^{-1}\left( f_0(\A)\cap g.f_0(\A) \right)$.
\end{proposition}
\begin{proof}
	Use Lemma \ref{lem:A2regular} to get $\pmb{a}\in B$ maximal with respect to the amount of distinct entries. Let $\pmb{b}\in A_\F$ such that $\pmb{b}.o=g.\pmb{a}.o \in f_0(\Omega)$. Since the sectors in $f_0(\A)$ are fundamental domains for $W_s$, there is a $w \in W_s$ such that $w.\pmb{b}.o$ and $\pmb{a}.o$ are in the same sector. Equivalently $\pmb{b}.o$ and $w^{-1}.\pmb{a}.o$ are in the same sector. Note that $gw.o = o$, $\pmb{b}.o = (gw).(w^{-1}\pmb{a}).o$, $f_0(\Omega)=f_0(\A)\cap g.f_0(\A)=f_0(\A)\cap (gw).f_0(\A)$ (since $w.\A = \A$) and $w^{-1}\pmb{a}$ is maximal with respect to the amount of distinct entries, so we can apply Proposition \ref{pro:A2id} to get $f_0|_\Omega = (gw).f_0|_\Omega=g.f_0 \circ w |_\Omega$.
\end{proof}
To show that axiom (A2) holds, there is one more obstacle, namely $g$ may not preserve $o$ in general. In fact, $\Omega$ may not even contain $o$, so we have to first translate $\Omega$ to be able to use the previous propositions. 
\begin{proposition}
	\label{pro:A2aW} 
	Let $g \in G_\F$ and define $\Omega := f_0^{-1}\left( f_0(\A)\cap g.f_0(\A) \right)$. Then there exists an element $\overline{w} \in W_a$ such that $f_0|_\Omega = g.f_0 \circ \overline{w}|_\Omega$.
\end{proposition}
\begin{proof}
	If $\Omega=\emptyset$, then any element of $W_a$ will suffice. Otherwise choose and fix $\pmb{a}\in B$ (not necessarily maximal) and $\pmb{b}\in A_\F$ with $\pmb{b}.o =g.\pmb{a}.o \in f_0(\Omega)$. Now consider any $a,b \in A_\F$ with $b.o = g.a.o \in f_0(\Omega)$. We translate the problem by $\pmb{b}^{-1}$ and get
	$$
	\pmb{b}^{-1}.b.o = (\pmb{b}^{-1}.g.\pmb{a}).\pmb{a}^{-1}.a.o
	$$
	where $(\pmb{b}^{-1}.g.\pmb{a}).o = o$. We are almost in the situation of Proposition \ref{pro:A2W}, but we have a different $\Omega$. In fact since $\pmb{b}^{-1}.f_0(\A)=f_0(\A)=\pmb{a}.f_0(\A)$,
	\begin{align*}
		f_0(\A)\cap (\pmb{b}^{-1}.g.\pmb{a}).f_0(\A) &=\pmb{b}^{-1}.f_0(\A)\cap \pmb{b}^{-1}.g.f_0(\A) \\ &=\pmb{b}^{-1}.f_0(\Omega)=f_0(\pmb{b}^{-1}.\Omega),
	\end{align*} 
	we can apply Proposition \ref{pro:A2W} to get a $w \in W_s$ such that
	$$
	f_0|_{\pmb{b}^{-1}.\Omega}= (\pmb{b}^{-1}.g.\pmb{a}).f_0 \circ w |_{\pmb{b}^{-1}.\Omega} \ \ ,
	$$
	which can be rewritten as
	$$
	\pmb{b}^{-1}.f_0|_\Omega =  \pmb{b}^{-1}.g.f_0\circ (\pmb{a} w \pmb{b}^{-1}) |_\Omega.
	$$
	Renaming $\overline{w}=\pmb{a} w \pmb{b}^{-1} \in W_a$ and translating back results in the required formula
	\begin{align*}
	f_0|_\Omega &=  g.f_0\circ \overline{w} |_\Omega. \qedhere
		\end{align*}
\end{proof}
We can now conclude the proof of axiom (A2).
\begin{proposition}\label{prop:A2_SLn}
	\label{pro:A2} The axiom
	\begin{enumerate}
		\item [(A2)] For all $f_1,f_2 \in \Fun$, if $f_1(\A)\cap f_2(\A)\neq \emptyset$, then the set $\Omega := f_2^{-1}\circ f_1(\A)$ is a $W_a$-convex set and there exists $w \in W_a$ such that $f_2|_\Omega = f_1\circ w |_\Omega$. 
	\end{enumerate}
	holds for $(\B,\Fun)$.
\end{proposition}
\begin{proof}
	Let $g,h \in G_\F$ with $f_1=g.f_0$ and $f_2=h.f_0$ such that $g.f_0(\A)\cap h.f_0(\A)\neq \emptyset$. We define $\Omega := (h.f_0)^{-1}(h.f_0(\A)\cap g.f_0(\A))$. We will first show the second part of (A2). To be able to apply Proposition \ref{pro:A2aW}, we act with $h^{-1}$ on
	$$
	h.f_0(\Omega)=h.f_0(\A)\cap g.f_0(\A),
	$$
	to get
	$$
	f_0(\Omega)=f_0(\A)\cap (h^{-1}g).f_0(\A).
	$$
	We can apply Proposition \ref{pro:A2aW} with $h^{-1}g$ to get an element $w \in W_a$ such that $
	f_0|_\Omega = (h^{-1}g).f_0 \circ w |_\Omega, 
	$
	and thus $f_2|_\Omega = h.f_0|_\Omega = g.f_0\circ w |_\Omega=f_1\circ w |_\Omega$.
	
	It remains to show that $\Omega$ is $W_a$-convex. Elements $a.o \in f_0(\Omega)$ are exactly those elements which have the special property that
	$
	h.a.o = g.w.a.o,
	$
	which is equivalent to $a^{-1}h^{-1}gwa \in G_\F(O)$ by Theorem \ref{thm:stab}. If $a = \diag{a} \in A_\F$, then
	$$
	(h^{-1}gw)_{ij}\frac{a_j}{a_i} \in O
	$$
	in matrix entries. Taking the valuation we obtain
	$$
	(-v)(\chi_{\alpha_{ij}}(a)) = (-v)(a_i/a_j) \geq (-v)((h^{-1}gw)_{ij}) \in \Lambda
	$$
	and these are exactly the inequalities that define affine half-apartments
	$$
	H_{\alpha_{ij},k}^+ := \left\{  a.o \in \A \colon (-v)(\chi_{\alpha_{ij}}(a)) \geq k  \right\}
	$$
	for $k = (-v)((h^{-1}gw)_{ij})  $. Since there are only finitely many pairs $(i,j)$, $\Omega$ is a finite intersection of half-apartments, i.e. $W_a$-convex.
\end{proof}